\documentclass[12pt]{article}
\usepackage{graphicx}
\usepackage{color}
\usepackage{amsmath}
\usepackage{amssymb}
\usepackage{amscd}
\usepackage{amsthm}
\usepackage{amsopn}
\usepackage{xspace}
\usepackage{verbatim}
\usepackage{amsfonts}
\usepackage{epic}
\usepackage{eepic}
\usepackage{stmaryrd}
\usepackage{empheq}
\usepackage{booktabs}
\usepackage{tikz}
\usepackage{pgfplots}

\usepackage[active]{srcltx}
\usepackage[a4paper,margin=3cm]{geometry}
\definecolor{red}{rgb}{1,0,0}
\definecolor{green}{rgb}{0,1,0}
\definecolor{SeaGreen}{RGB}{46,139,87}
\definecolor{Maroon}{RGB}{128,0,0}

\newcommand{\N}{\mathbb{N}}

\newcommand{\C}{{\mathbb{C}}}
\newcommand{\R}{{\mathbb{R}}}

\newcommand{\B}{\mathcal B}

\newcommand{\D}{\mathcal D}

\def\Hg {{\mathcal H}}

\newcommand{\LL}{\mathcal L}

\newcommand{\OO}{\mathcal O}

\def\Rg {{\mathcal R}}
\def\Sg {{\mathcal S}}

\def\Wg {{\mathcal W}}

\def\Tg {{\mathcal T}}

\def\Df {{\mathfrak D}}

\def\curl{\text{\rm curl}}

\def\curl{\text{\rm curl\,}}

\def\Div{\text{\rm div\,}}

\renewcommand {\Re}{{\rm Re\,}}
\renewcommand{\Im}{{\rm Im\,}}

\def\Ai{\text{\rm Ai\,}}

\def\0{\mathbf  0}

\def\XXint#1#2#3{{\setbox0=\hbox{$#1{#2#3}{\int}$ }
\vcenter{\hbox{$#2#3$ }}\kern-.6\wd0}}

 

\errorcontextlines=0 \numberwithin{equation}{section}
\theoremstyle{plain}
\newtheorem{theorem}{Theorem}[section]
\newtheorem{lemma}[theorem]{Lemma}

\newtheorem{proposition}[theorem]{Proposition}

\newtheorem{remark}[theorem]{Remark}
\newtheorem{corollary}[theorem]{Corollary}

\title{Eigenvalues of the linearized Navier-Stokes operator near
  locally Couette laminar flows.\\
\author{ Y. Almog, Department of
  Mathematics, \\ Braude College of Engineering, \\ 
    Carmiel 2161002, Israel \\~\\
  and \\~\\
\noindent   B. Helffer, Laboratoire de Math\'ematiques Jean Leray, \\CNRS and  Nantes Universit\'e, \\
 44000 Nantes  Cedex France}}

\begin{document}
\maketitle
\bibliographystyle{siam}
\begin{abstract}
  We consider the spectrum of the Orr-Sommerfeld operator, which is
  obtained from the linearized Navier-Stokes (LNS) operator near a
  laminar flow in an infinite two-dimensional channel in the large
  Reynolds number limit. For a rather general class of laminar flows,
  which behave locally near the boundary like a Couette flow and are
  bounded from below (above) by a linear increasing (decreasing)
  function, we show that there exists an eigenvalue of the LNS
  operator which coincides, to leading order, with the eigenvalue of LNS
  near Couette flow, which was formally  obtained by Wasow in 1953. 
  \end{abstract}

\section{Introduction}
\label{sec:1}

Consider the incompressible Navier-Stokes equations in the
   two-dimensional pipe  $D=\R\times (-1,1)$
   \begin{equation}
   \label{eq:1}
   \begin{cases}
   \partial_t {\mathbf v} - \epsilon \Delta {\mathbf v} + {\mathbf v} \cdot \nabla {\mathbf v} = -
   \nabla p & \text{in } \R_+\times D \,, \\
   {\mathbf v}=v_b\; \hat{i}_1  & \text{on } \R_+\times\partial D  \,,
   \end{cases} 
   \end{equation}
   where  $\hat{i}_1 = (1,0)$, ${\mathbf v}=(v_1,v_2)$ is the fluid
   velocity, and $p$ is the pressure. \\
   The parameter
    \begin{equation} \label{defreynolds}
    R :=\frac 1 \epsilon
    \end{equation}  is the Reynolds
    number of the  flow and 
    \begin{displaymath}
      v_b:\partial D\to\R
    \end{displaymath}
is the boundary velocity. Since the
    flow is incompressible we must have
   \begin{displaymath}
     \Div {\mathbf v}=0\,.
   \end{displaymath}
   We linearize \eqref{eq:1} near the laminar flow (cf. \cite{AH-arma}) 
   \begin{displaymath}
     {\mathbf v}=U(x_2)\, \hat{i}_1 \,,
   \end{displaymath}
   to obtain the linearized equation 
\begin{displaymath}
  {\bf u}_t-\mathcal T_0({\bf u},q)=0\,,
\end{displaymath}
   where  ${\bf u}=(u_1,u_2)$ and $q$  are defined on $\mathbb R_+\times
   D$, and $\mathcal T_0$ is the map 
   \begin{subequations}
   \label{eq:2}
         \begin{equation}
   ({\bf u},q) \mapsto  {\mathcal T}_0 ({\bf u} , q ):=  \mathfrak T {\mathbf u} - \nabla q\,,
   \end{equation}
where
\begin{equation}
 \mathfrak T {\mathbf u}= -  \epsilon \,\Delta {\mathbf u} + U\, \frac{\partial{\mathbf
       u}}{\partial x_1}+ \, u_2\, U^\prime\, \hat{i}_1 \,.
\end{equation}
   \end{subequations}
The associated resolvent equation for
 $\mathcal T_0$ assumes the form
   \begin{equation}
   \label{eq:3}
     \Tg_0({\bf u},q)-\Lambda{\bf u}={\bf f} \,,
   \end{equation}
where $\Div {\bf u}=0$ and  $\Lambda\in\C$ is the spectral parameter.

For $0<L\ll R$ we consider $ \mathfrak T$ acting on a subspace  $\Hg_{\Div}^0$ of $\Hg_{\Div}$
(see \cite[Section 2]{AH-arma}) which is
defined by
\begin{equation}
 \Hg_{\Div}^0:=\Big\{ {\mathbf u} \in  \Hg_{\Div}\,|\,   \int_{(0,L)\times (-1,+1)} u_1(x_1,x_2)\, dx_1 dx_2  =0\Big\}\,,
 \end{equation}
where
\begin{displaymath}
  \Hg_{\Div}=\{ {\mathbf u} \in L^2_{loc}(\overline{D},\R^2)\,|\, \Div {\mathbf u} =0 \,; \; {\mathbf
    u}\cdot{\mathbf n}|_{\partial D}=0\,; \; {\mathbf u}(\cdot+L,\cdot)={\mathbf u}(\cdot,\cdot)\}  \,.
\end{displaymath}
Set
  \begin{equation}
     \Hg=\{{\mathbf u} \in L^2_{loc}(\overline{D},\R^2) \,|\, {\mathbf u}(\cdot+L,\cdot)={\mathbf u}(\cdot,\cdot)\}  \,, 
  \end{equation}
and let $P:\Hg\to \Hg_{\Div}^0$ denote the standard $L^2$ projection
(which is sometimes called the Leray projection). Applying $P$ to
\eqref{eq:3} we obtain (recall that ${\bf u}\in \Hg_{\Div}^0$)
\begin{equation}
\label{eq:4}
  (P \mathfrak T -\Lambda){\mathbf u}=P{\bf f}\,.
\end{equation}
To account for the regularity of ${\bf u}$ and the boundary conditions
it satisfies we consider ${\bf u}\in  \Hg_{\Div}^0\cap\Wg_{\mathfrak D}$ where
\begin{displaymath}
  \Wg_{\mathfrak D}= \Big\{{\mathbf u} \in H^2_{loc}(\overline{D},\R^2)
  \,\Big|\,  \mathbf{u}\Big|_{\partial D}=0\Big\}\,.
\end{displaymath}
Note  (see \cite[Section 2.1]{AH-arma}) that 
   \begin{displaymath}
   \Hg=\Hg_{\curl}\oplus\Hg_{\Div}^0 \mbox{ where }
  \Hg_{\curl}=\{ {\mathbf u} \in \Hg\,|\, \curl {\mathbf u}=0 \,\}  \,.
  \end{displaymath}

We proceed with a formal derivation of the Orr-Sommerfeld equation.
   Since $\Div {\bf u}=0$ we may define a stream function
   \begin{displaymath}
     {\mathbf u}=\nabla_\perp\psi=(-\psi_{x_2},\psi_{x_1}) \,.
   \end{displaymath}
  Substituting the above into \eqref{eq:3} and then taking the curl of
  the ensuing equation for $\psi$ yields
  \begin{equation}
  \label{eq:5}
    \Big(-  \epsilon \Delta^2 + U\frac{\partial}{\partial x_1}\Delta -
    U^{\prime\prime}\frac{\partial}{\partial x_1}  - \Lambda \,  \Delta\Big)\psi=F \,,
  \end{equation}
where $F=\curl {\bf f}$.  Substituting $\psi(x_1,x_2)=\phi(x_2) \, e^{i\alpha
  x_1}$ (where $\alpha=2\pi/L$) into
  \eqref{eq:5} with $ \phi:(-1,1)\to\C$ yields  the equation
  \begin{subequations}
  \begin{displaymath}
    \B_{\lambda,\alpha,\beta}\,\phi=f \,, 
  \end{displaymath}
  where (setting $x_2=x$)
   \begin{equation}
\label{eq:6}
     \B_{\lambda,\alpha,\beta} =(\LL_\beta -\beta\lambda)\Big(\frac{d^2}{dx^2}-\alpha^2\Big)  -i\beta U^{\prime\prime} \,,
   \end{equation}
in which
  \begin{equation}
  \label{eq:7}
    \LL_\beta  = -\frac{d^2}{dx^2}+i\beta U\,,\,\beta=\alpha R\,, \mbox{ and } \lambda=\Lambda/\alpha-\alpha/R\,.
    \end{equation}
  \end{subequations}
We consider its realization  $ \B^\D_{\lambda,\alpha,\beta} $ on the
 following domain, corresponding  to no-slip boundary conditions,
  \begin{equation}
\label{eq:8}
  D(\B_{\lambda,\alpha,\beta}^\D)=\{u\in H^4(-1,1)\,,\, u(1)=u^\prime(1)= u(-1) =u^\prime (-1) =0 \}\,.
  \end{equation}

Note that when, for given $\alpha$ and $\beta$, if for some
$(\phi,\lambda)\in D(\B_{\lambda,\alpha,\beta}^\D)\times\C$ such that $\|\phi\|_2=1$ it holds that
$\B_{\lambda,\alpha,\beta}^\D\phi=0$ then the pair $({\bf u},\Lambda)$ given by
\begin{displaymath}
  {\bf u}=\nabla_\perp(\phi(x_2)e^{i\alpha x_1})\quad , \quad \Lambda=\alpha\lambda+\frac{\alpha^2}{R}\,
\end{displaymath}
is an eigenpair of $P\mathfrak{T}$ and satisfies \eqref{eq:4} for
${\bf f}=0\,$. \\
Obviously ${\bf u}$ belongs to $ \Wg_{\mathfrak D}\,,$ given that
$\phi(\pm1)=\phi^\prime(\pm1)=0\,$. Furthermore, ${\bf u}$ belongs to $\Hg_{\Div}^0$ for any $\alpha>0$
by construction, and since
\begin{displaymath}
  \nabla\times(\mathfrak T -\Lambda){\mathbf u}=0 \Leftrightarrow \B_{\lambda,\alpha,\beta}^\D\phi=0\,,
\end{displaymath}
it follows that $(\mathfrak T -\Lambda){\bf u}\in\Hg_{curl}$ and that \eqref{eq:4} is
satisfied.

We assume that $U\in C^4([-1,1])$ satisfies
\begin{subequations}
\label{eq:9}
  \begin{equation}
  U^\prime(-1) >0\,,
\end{equation}
and 
\begin{equation}
  U(x)>U(-1)\,,\; \forall x\in(-1,1]\,. 
\end{equation}
\end{subequations}
Note that an immediate consequence of \eqref{eq:9} is that there
exists $\Xi>0$ such that
\begin{equation}
\label{eq:10}
  U(x)\geq U(-1)+  \Xi \, (1+x) \,.
\end{equation}

  \begin{subequations}\label{eq:11}
Let  $\LL_{0,\beta}$ be the differential operator given by
\begin{equation}
  \LL_{0,\beta}=-\frac{d^2}{dx^2}+i\beta(1+x)\,,
\end{equation}
and consider its realization $\LL^{(0)}_{0,\beta}$  in $(-1,+\infty)$ on
\begin{equation}
   D(\LL^{(0)}_{0,\beta})=\{u\in H^2(-1,+\infty)\,|\,\langle1,u\rangle=0\,,\; (1+x)u\in L^2(-1,\infty)\}\,.
\end{equation}
\end{subequations}
 
Following \cite[Eqs. (6.6)-(6.10)]{AH-arma} we can relate the
spectrum of 
$\mathcal L^{(0)}_{0,\beta}$ 
with the zeroes of $A_0$ where $A_0$  is
the holomorphic extension of
 \begin{equation}
\label{eq:defA0}
 A_0(z):= e^{i\frac \pi 6}\int_z^{+\infty} {\rm Ai}(e^{i\frac \pi 6} t)\, dt\,.
 \end{equation}
By \cite[Eq. (A.4)] {AH-arma} combined with a shift by $1$ of the origin  (by
homogeneity, we can reduce the question to the case $\beta=1$, see 
\eqref{eq:12} and  (\ref{eq:106}a) below)  eigenfunctions of  $\LL_{0,\beta}$ for $\beta=1$ must
assume the form, for some non zero constant $C$, 
\begin{displaymath}
  u(x)= C \, {\rm Ai}(e^{i\frac \pi 6}[1+x+i\lambda]) \,.
\end{displaymath}
To have $u\in D(\LL^{(0)}_{0,\beta})$ we must then have
\begin{displaymath}
  \int_{-1}^{+\infty } {\rm Ai}(e^{i\frac \pi 6}[1+x+i\lambda]) \,dx=0\,,
\end{displaymath}
which in terms of $A_0$ can be rewritten as
 \begin{equation}\label{eq:zero}
 A_0 (i\lambda)=0\,,
 \end{equation}
 which is a necessary condition so that $\lambda\in\sigma(\LL^{(0)}_{0,\beta})$.  We
 can then use \cite[Corollary A.7]{AH-arma} to prove that
 $\sigma(\LL^{(0)}_{0,\beta})$ lies in the positive half-plane.  Notice also for
 future use that, using the localization of the zeroes of the Airy
 function, the zeroes of $\lambda \mapsto A_0 (i\lambda)$ are simple:
 \begin{equation}\label{eq:simplezero}
  A_0 (i\lambda)=0 \mbox{ implies }  A_0^\prime (i\lambda)\neq 0\,.
  \end{equation}
  We refer the reader to \cite[Section A2]{AH-arma} and to \cite{wa53}
  for more details on the operator $\LL^{(0)}_{0,\beta}$.

Denote by $\check \lambda_0(\beta)$ the leftmost eigenvalue of $\LL^{(0)}_{0,\beta}$ and introduce 
\begin{equation}\label{eq:7abcd}
\lambda_0(\beta):= \check \lambda_0(\beta)/\beta\,.
\end{equation}
 Note that by homogeneity
\begin{equation}
\label{eq:12}
\lambda_0(\beta)= \beta^{-1/3} \lambda_0(1)\,,
\end{equation}
and recall that
\begin{displaymath}
\frac \pi 6 < {\rm arg} \lambda_0(1) < \frac{\pi}{ 2}\,.
\end{displaymath}
Some illustrating numerics can be found in the appendix (see also
\cite[Appendix A]{drre04}).  Our main result is that $\lambda_0(\beta)$ is
close, for large values of $\beta$, to the "spectrum" of the
Orr-Sommerfeld operator $\B^\D_{\lambda,\alpha,\beta}\,$:
\begin{theorem}
\label{thm:main}
Let $U$ satisfy \eqref{eq:9} and $U^\prime(-1)=1$.  For $\alpha_0>0$ and $\varepsilon>0$, there exists 
  $\beta_0 >0$, such that,  for $\beta\geq \beta_0$ and $\alpha \in [0,\alpha_0]\,$,
  and some $\lambda\in \mathbb C$ satisfying $|\lambda-\lambda_0(\beta)-i\,U(-1)|\leq \beta^{-1/2+\varepsilon}$ it
  holds that   ${\rm Ker \,} \B^\D_{\lambda,\alpha,\beta}\neq \{0\}$.
\end{theorem}
We recall from above that $\Lambda=\alpha\lambda+\epsilon\alpha^2$ is an eigenvalue of the linearized
Navier-stokes operator. 

\begin{remark}
  Note that since
\begin{displaymath}
  \B^\D_{\lambda,\alpha,\beta}[U]= \overline{\B^\D_{\bar{\lambda},\alpha,\beta}}[-U]
\end{displaymath}
we can use Theorem \ref{thm:main} to establish existence of a critical
value near $\bar{\lambda}_0(\beta)+i\,U(-1)$ for laminar flows which satisfy instead of
\eqref{eq:10} the condition
\begin{displaymath}
   U(x)\leq U(-1)-  \Xi \, (1+x) \,.
\end{displaymath}
Furthermore, applying the transformation $\Rg \phi(x)=-x$ yields that
\begin{displaymath}
  \Rg^{-1}\B^\D_{\lambda,\alpha,\beta}[U]\Rg= \B^\D_{\lambda,\alpha,\beta}[\Rg U]
\end{displaymath}
Consequently, if $U^\prime(1)>0$ and
\begin{displaymath}
  U(x)-U(1)\leq-\Xi\,(1-x)
\end{displaymath}
we may use Theorem \ref{thm:main} to establish the existence of a
critical value in the vicinity of $\lambda_0(\beta)+i\,U(1)$.
\end{remark}

The spectral analysis of the Orr-Sommerfeld operator \cite{Orr1907}
has received significant attention recently, see
\cite{chen2020transition,AH-arma,AH-ems,almog2025stability,jia2023uniform,chen2024enhanced,grenier2016spectral}
to name just a few of the works addressing the linear operator only.
Most of these works, with the exception of \cite{grenier2016spectral},
prove stability of the laminar flow, and find the order of magnitude of
$\inf\Re\sigma(P\mathfrak T)$ in the large Reynolds number limit. A
significant body of literature deals with weakly non-linear analysis
of the laminar flow, see for instance
\cite{BGM,Mas17,chen2024transition}.  

In a recent work \cite{jezequel2025orr} the authors show (for
$\alpha\gtrsim1$ and holomorphic $U$) that existence of eigenvalues for the Orr-Sommerfeld operator
depend on their existence as eigenvalues of the corresponding Rayleigh
operator except for the cases $|\lambda-i\,U(\pm1)|\ll1$ in the limit $\beta\to\infty$.  For
monotone shear flow, when
no eigenvalues exist for the Rayleigh operator (see
\cite{almog2025stability}), if combined with
the present contribution, \cite{jezequel2025orr} hints towards
the conjecture 
\begin{displaymath}
 \lim_{\beta\to\infty} \frac{\beta^{1/3}}{\alpha}\inf \Re \sigma(P\mathfrak T)= \lambda_0(1)\,.
\end{displaymath}

While the construction of the quasi-mode is fairly simple, much effort
has to be invested in the estimate of $(\B^\D_{\lambda,\alpha,\beta})^{-1}$ for $\lambda$
in the vicinity of $\lambda_0(\beta)$. At first, we need to estimate it near the
boundary at $x=-1$, in $[-1,-1+\delta]$ for some $\beta^{-1/6}\ll\delta\ll1$.  Then, we
estimate it in the outer domain $(-1+\delta,1)$.
Once we have obtained a bound for $\|(\B^\D_{\lambda,\alpha,\beta})^{-1}\|$ along a suitable path around $\lambda_0(\beta)$,
  existence of eigenvalues can be proved by using Cauchy's
Theorem .

The rest of the contribution is arranged as follows: In the next
section we construct a quasimode based on the principal eigenfunction
of $\LL^{(0)}_{0,\beta}$. Section 3 is mostly devoted to resolvent estimates of
various Schr\"odinger operators for $\lambda$ in the vicinity of $\lambda_0(\beta)$. In
Section 4 we consider a no-slip Schr\"odinger operator, as defined in
\cite[Section 6]{AH-arma}, once again for $\lambda$ in the vicinity of
$\lambda_0(\beta)$. Finally, in the last section we complete the proof of
Theorem \ref{thm:main}.

\section{Quasimode construction}
\label{sec:2}

In this section we construct an approximation,
for given $(\alpha,\beta)\in [0,\alpha_0]\times [\beta_0,+\infty)$, to a pair
$(\lambda,\phi)\in D(\B^\D_{\lambda,\alpha,\beta})\times\C$ satisfying
\begin{equation}\label{eq:13}
\B^\D_{\lambda,\alpha,\beta}\, \phi=0\,.
\end{equation}
 More precisely, we seek some $L^2$-normalized $\phi^{app}\in D(\B^\D_{\lambda,\alpha,\beta})$ and $\lambda^{app}\in\C$ such that
\begin{displaymath}
\B^\D_{\lambda^{app},\alpha,\beta}\, \phi^{app}= r^{app}\,,
\end{displaymath}
where $r^{app}$ is sufficiently small as $\beta \to +\infty$ to allow us the
proof of Theorem \ref{thm:main}. It turns out in the sequel that it is
enough, for given $\alpha_0 >0$ to find $C$ and $\beta_0$ such that, for $\beta \geq
\beta_0$ and $\alpha\in [0,\alpha_0]$, to construct a non trivial quasimode
$\phi^{app}$ such that
\begin{equation}
\label{eq:goal}
 \|r^{app}\|_2\leq C\beta^{1/2}\|(\phi^{app})^{\prime\prime}-\alpha^2\phi^{app}\|_2\,.
\end{equation}

 Without any loss of generality we assume that 
\begin{equation}
\label{eq:14}
U(-1)=0\,,
\end{equation}
otherwise we
redefine the spectral parameter by setting
\begin{displaymath}
 \tilde{U}=U-U(-1)\quad ; \quad \tilde{\lambda}= \lambda+i\,U(-1) \,.
\end{displaymath}
To simplify the problem even further we assume that 
\begin{equation}
\label{eq:15}
U^\prime(-1)=1\,,
\end{equation}
without any loss of generality. Indeed, for $U^\prime(-1)>0$, if we set
$\tilde{U}=U/U^\prime(-1)$ and $\tilde{\beta}=U^\prime(-1)\beta$, then  \eqref{eq:13} remains
unaltered.  \\

\paragraph{ Some heuristics}
In the sequel, we assume \eqref{eq:9}, \eqref{eq:14} and
\eqref{eq:15}.\\

We begin by heuristically explaining the construction of the leading order term.
To this end we neglect   the term $i\beta U^{\prime\prime}\phi$  and assume that
\begin{displaymath}
  U=1+x\,.
\end{displaymath}
We then obtain the approximate equation in $(-1,+1)$
\begin{subequations}
\label{eq:16}
\begin{equation}
  (\LL_{0,\beta}-\beta\lambda)v=0\,,
\end{equation}
where 
\begin{equation}
  v=-\phi^{\prime\prime}+\alpha^2\phi\,,
\end{equation}
\end{subequations}
for some $\phi\in D(\B_{\lambda,\alpha,\beta}^\D)$ and 
\begin{equation}
\label{eq:17}
  \LL_{0,\beta}=-\frac{d^2}{dx^2}+i\beta(1+x)\,.
\end{equation}
It can be easily verified that for any $(\phi,\lambda)$ in
$D(\B^\D_{\lambda,\alpha,\beta})\times\C$ satisfying \eqref{eq:16}it holds that
\begin{subequations}
\label{eq:18}
  \begin{equation}
  \langle\zeta_\pm,v\rangle=0 \,,
\end{equation}
where $\zeta_\pm: [-1,1]\to\R$ satisfy
\begin{equation}
  \begin{cases}
    -\zeta_\pm^{\prime\prime}+\alpha^2 \zeta_\pm=0 & \text{in }(-1,1)\,, \\
    \zeta_\pm(\pm1)=1\,, & \zeta_\pm(\mp1)=0 \,.
  \end{cases}
\end{equation}
\end{subequations}

Next, let $\check{\psi}_\pm(\cdot.\lambda,\beta)$ denote the $L^2(-1,\infty)$
solutions to the problems (as in  \cite[Eqs. (6.8), (8.86)]{AH-arma} 
but with no cut-off function $\Theta_\pm$, and with   $U(-1)=0$, and $J_-=U'(-1)=1$):
\begin{subequations}
\label{eq:19}
  \begin{equation}
  \begin{cases}
 (\LL_{0,\beta}-\beta\lambda)\check {\psi}_-=0 &\mbox{ for }  -1<x \\
     \check  \psi_-(-1)=1 \,,
  \end{cases}
\end{equation}
and 
\begin{equation}
  \begin{cases}
     \Big(-\frac{d^2}{dx^2}+i\beta[U(1)-J_+(1-x)+i\lambda]\Big)\check {\psi}_+=0 &\mbox{ for } x<1 \\
     \check   \psi_+(1)=1 \,,
  \end{cases}
\end{equation}
 where $J_+=U^\prime(1)$. \\
 Note that in this section it holds
  that $U(1)=2$ and $J_+=1$. For later reference we need however a slightly
  more  general definition.
\end{subequations}
   These solutions are given, for   $\lambda\not\in
   \beta^{-1}\,\sigma(\LL_{0,\beta}^{\D,(-1,\infty)})$ where $\LL_{0,\beta}^{\D,(-1,\infty)}$ is
   the realization of $\LL_{0,\beta}$ on $H^2(-1,\infty)\cap H^1_0(-1,\infty)$ (see
   Section \ref{sec:dirichlet-infinite}), 
 by
    \begin{subequations}
\label{eq:20} 
    \begin{equation} 
\check \psi_-(x)= \frac{{\rm Ai}\big( \beta^{1/3}e^{ i\pi/6}\big[(1+x)+
    i\lambda \big]\big)} {{\rm Ai}\big( \beta^{1/3}e^{ 2i\pi/3} \lambda\big)}\,, 
\end{equation}
and 
\begin{equation}
\overline{\check \psi_+ (x)}= \frac{{\rm Ai}\big((J_+ \beta)^{1/3}e^{
    i\pi/6}\big[(1- x)+ iJ_+^{-1}(\bar \lambda+ i\,U( 1))\big]\big)}
{{\rm Ai}\big(J_+^{-2/3}\beta^{1/3}e^{2 i\pi/3}[\bar \lambda+  i\,U(1)]\big)}\,.
\end{equation}
\end{subequations}
 In the rest of this section we continue to assume $U(1)=2$ and
  $J_+=1$ (given that $U=1+x$).\\
In the construction of the
quasi-mode, we choose 
\begin{equation}
  \label{eq:21}
\lambda \sim \lambda_0(\beta) \mbox{  with } \lambda_0(\beta) \mbox{  introduced in
\eqref{eq:12}}\,.
\end{equation}
Hence, for $\beta$ large enough,
\begin{displaymath}
 {\rm Ai}\big( \beta^{1/3}e^{ 2i\pi/3} \lambda\big) \sim   {\rm Ai}\big(e^{ 2i\pi/3} \lambda_0(1) \big)\neq 0\, , 
   \mbox{ and }
   {\rm Ai}\big(\beta^{1/3}e^{2 i\pi/3}  (\bar \lambda+  2i )\big)\neq 0\,.
   \end{displaymath}

  To prove the first inequality we refer the reader to Lemma
   \ref{lem:no-zeroes} from which we can conclude, via analytic
   dilation and \eqref{eq:12} that $\arg \lambda_0(\beta)\neq\pi/3$, and consequently
   also that $\arg (e^{ 2i\pi/3}\lambda_0(1))\neq\pi$.  A numerical confirmation is
   provided in the appendix.

  The second inequality follows from the fact that by
   \eqref{eq:12} we have, as $\beta \to +\infty$, 
   \begin{displaymath}
       \arg (e^{2
     i\pi/3}[\bar \lambda+  2i ])\to 7\pi/6 \equiv -5  \pi/6  \; ({\rm mod.}\; 2\pi) \,,
   \end{displaymath}
   whereas the zeroes of Airy's
   function lie all on the negative real axis.  Using  the asymptotic behaviour of  Airy
function  we can obtain an asymptotic approximation of $\check{\psi}_\pm$
in the limit $\beta\to\infty$. We recall (see for example  \cite[Eq. (10.4.59)]{abst72})) that:
  \begin{equation}\label{eq:asymptAiry}
  {\rm Ai}(z)\sim \frac 12 \pi^{-1/2} z^{-1/4} e^{-\frac 23 z^{3/2}}
  \Big(\sum_{k=0}^\infty c_k z^{-3k/2}\Big) \mbox{ for } |{\rm arg } z|< \pi\,,|z| \to +\infty\,, 
\end{equation}
where $c_0=1$.\\
For later reference we also recall from
  \cite[Eq. (10.4.61)]{abst72} that
 \begin{equation}
\label{eq:22}
  {\rm Ai}^\prime(z)\sim -\frac 12 \pi^{-1/2} z^{1/4} e^{-\frac 23 z^{3/2}}
  \Big(\sum_{k=0}^\infty d_k z^{-3k/2}\Big) \mbox{ for } |{\rm arg } z|< \pi\,,|z| \to +\infty\,, 
\end{equation}
where $d_0=1$.\\
Since $|\lambda|$ is small it holds for sufficiently large $\beta$ that the
argument of \break $e^{ i\pi/6}[(1- x)+ i (\bar \lambda+ 2i )]$ is nearly
$-5\pi/6$ (mod. $2\pi$) and hence we can apply \eqref{eq:asymptAiry}. From
\eqref{eq:asymptAiry} we can obtain the precise asymptotic behavior of
$\check \psi_+$ except for a small semi-neighborhood of $x=-1$ where
$\check \psi_+$ is exponentially small as is manifested below.
Similarly, for $\check \psi_-$, the argument of $e^{ i\pi/6} [(1+x)+ i\lambda )]$
tends to $\pi/6$ as $\beta\to\infty$.  Thus, we can use \eqref{eq:asymptAiry} to
obtain the precise asymptotic behavior of $\check \psi_-$ except for a
small semi-neighborhood of $x=-1$. More precisely for $M$ large
enough, we get a proper estimate for $(1+x) \geq M \beta^{-1/3}$ using 
\eqref{eq:asymptAiry} for both the numerator and the denominator in (\ref{eq:20}a).
Then, we observe that $\check \psi_-$ is bounded for $0 \leq (1+x) \leq M
\beta^{-1/3}$.  For latter reference we also obtain here a lower bound for
$\| \check \psi\|_p^p$. For $m>0$ small enough we observe, using
(\ref{eq:20}a), that $\check \psi_-$ remains sufficiently close to $1$
on $[-1,-1+m\beta^{-1/3}]$ so that for any $p\in [-1,+\infty)$ it holds that
 \begin{equation}\label{eq:23} 
 \int_{-1}^{1 } | \check \psi_-(x)|^p \, dx  \geq  \int_{-1}^{-1 +m \beta^{-1/3})  } | \check \psi_-(x)|^p\, dx \geq \frac{1}{C_{p,m}} \beta^{-1/3}\,.
 \end{equation}

By the foregoing discussion, applying \eqref{eq:asymptAiry} to
\eqref{eq:20} yields
\begin{equation} 
\label{eq:24}
 \overline{ \check{\psi}_+}(x)=\exp \Big\{-\frac{2}{3}e^{i\pi/4}\beta^{1/2}[-(1+x- i \bar
  \lambda)^{3/2}+(2-i \bar \lambda)^{3/2}])\Big\}[1+\OO(\beta^{-1/2})] \,,
\end{equation}
which is valid for any $x\in[-1+M\beta^{-1/3},1]$ for sufficiently large
$M>0\,$.\\
 Note that,  for $x\in[-1,-1+M\beta^{-1/3}]\,$,  it holds that
\begin{displaymath}
  \beta^{1/3}\big|(1-  x)+ i (\bar \lambda+  2i )\big|\leq 2M\,,
\end{displaymath}
and consequently, by (\ref{eq:20}b) and \eqref{eq:asymptAiry}
(applied to the denominator in (\ref{eq:20}b)) we obtain
\begin{equation}
\label{eq:25}
  |\check \psi_+ (x)|\leq C_M\beta^{\frac{1}{12}}e^{-\beta^{1/2}} \mbox{ for } x\in[-1,-1+M\beta^{-1/3}]\,.
\end{equation}
By \eqref{eq:24} we can conclude, that for all $x\in[1-\beta^{-1/3},1]$,  
\begin{displaymath}
 \check{\psi}_+(x) =e^{- e^{i\pi/4} [2\beta]^{1/2}(1-x)}[1+\OO(\beta^{1/2}[1-x]^2)]\,,
\end{displaymath}
and hence, by \eqref{eq:24} and \eqref{eq:25} we have, for all $x\in[-1,1]$\,,
\begin{equation}\label{eq:26}
 \check{\psi}_+(x)=e^{- e^{i\pi/4} [2\beta]^{1/2}(1-x)}\big[1+\OO(\beta^{-1/6})\big]+
 \OO(e^{-\beta^{1/6}})\,.
\end{equation}
Notice that the above asymptotic approximation implies
for any $s\geq 0$ the existence of $C_s$ such that, for 
$\lambda$  satisfying \eqref{eq:21} and $\beta$ large enough,
   \begin{equation}\label{eq:216x}
   \| (1-x)^s \check \psi_+\|_1 \leq C_s  \, \beta^{-(s+1)/2}\,,
   \end{equation} 
and
\begin{equation}
  \label{eq:27}
  \| (1-x)^s \check \psi_+\|_2 \leq C_s  \, \beta^{-(2s+1)/4}\,,
\end{equation}

We have already mentioned above that for $0 \leq (1+x) \leq M \beta^{-1/3}$,
$|\check \psi_-|$ is uniformly bounded as $\beta\to\infty$. For $(1+x) \geq M \beta^{-1/3}\,$,
we use \eqref{eq:asymptAiry} to obtain
      \begin{equation*} 
|\check \psi_-(x)| \leq  C\, \big|{\rm Ai}\big( \beta^{1/3}e^{ i\pi/6}\big[(1+x)+
    i\lambda \big]\big)\big| \leq \hat C \Big| e^{-\frac 23 \beta^{1/2} (e^{ i\pi/6}[(1+x)+
    i\lambda ]) ^{3/2}}\Big| \,, 
\end{equation*}
and hence (for sufficiently large $M$)
  \begin{equation} \label{eq:28}
|\check \psi_-(x)| \leq \hat C \,  e^{-\frac 23 \cos (\pi/4) \beta^{1/2} (1+x)^{3/2} (1 - \hat C /M) }
     \,.
    \end{equation}
   Ifollows drom \eqref{eq:28} that for any $s\geq 0$, there exists
   $C_s>0$ such that,  for $\lambda$ satisfying \eqref{eq:21} and for $\beta$ large
   enough,  
   \begin{equation}\label{eq:29}
   \| (1+x)^s \check \psi_-\|_1 \leq C_s\, \beta^{-(s+1)/3}\,.
   \end{equation} 
   Similarly, we obtain for future use that
    \begin{equation}\label{eq:30}
   \| (1+x)^s \check \psi_-\|_2 \leq C_s \, \beta^{-(s+1)/6}\,,
   \end{equation} 
 and from \eqref{eq:22} and \eqref{eq:20} that
\begin{equation}
  \label{eq:31}
 \| (1+x)^s \check \psi_-^\prime \|_2 \leq C_s \, \beta^{-(s-1)/6}\,.
\end{equation}
\begin{remark}
The above estimates are already present in \cite{AH-arma} (see 
\cite[Eq. (8.91)]{AH-arma} for instance). In this context it should be
mentioned that the estimates in  \cite{AH-arma} are stated under the
assumption $\Re\lambda \leq \Re \lambda_0$ given that all the results in \cite{AH-arma}
are obtained under this assumption (or by assuming that $\beta^{1/3}\Re\lambda$ is
sufficiently small). Nevertheless, as is clear from \eqref{eq:24}
and \eqref{eq:25} it is unnecessary to stipulate that $\Re\lambda \leq \Re \lambda_0$
for \eqref{eq:29} and \eqref{eq:30} to be valid in the present
context. 
  \end{remark}

 By \eqref{eq:26} and \eqref{eq:28}, for $\beta$ large enough,
    $\check \psi_+$ and $\check \psi_-$ are respectively localized near
    $x=1$ and near $x=-1$, and hence must be linearly independent.
    Consequently, we may write the general solution of \eqref{eq:16} in
    the form
\begin{displaymath}
  v=C_+\check {\psi}_+ + C_-\check {\psi}_-  \,.
\end{displaymath}
Consequently, by \eqref{eq:18}, we obtain that
\begin{equation}
  \label{eq:32}
  \begin{bmatrix}
    a_+^+ & a_+^- \\
    a_-^+ & a_-^-   
  \end{bmatrix}
  \begin{bmatrix}
    C_+ \\
    C_-
  \end{bmatrix}
=0 \,,
\end{equation}
where
\begin{displaymath}
  a_\pm^+= \langle\zeta_\pm, \check \psi_+\rangle\quad \;   \quad a_\pm^-= \langle\zeta_\pm,\check \psi_-\rangle\,.
\end{displaymath}
To have nontrivial solutions for \eqref{eq:32} we must have
\begin{equation}
\label{eq:33}
  a_+^+ a_-^- - a_-^+ a_+^-=0 \,.
\end{equation}
Recalling that $\lambda$ satisfies \eqref{eq:21} we have by \eqref{eq:216x}
and \eqref{eq:29}  (with $s=1$) 
\begin{equation}
\label{eq:34} 
  |a_+^-|\leq C\beta^{-2/3} \quad \text{and} \quad    |a_-^+|\leq C\beta^{-1} \,.
\end{equation}

We obtain, given that $\zeta_+(1)=1$ and $ a_+^+= \langle\zeta_+, \check \psi_+\rangle$,  that
\begin{equation}
\begin{array}{ll}
  a_+^+&=\int_{-1}^1e^{-e^{i\pi/4}[2\beta]^{1/2}(1-x)}
  [1+  \OO(\beta^{-1/6})]\zeta_+(x)\,dx+\OO(e^{-\beta^{1/6}})\\
 & =e^{-i\frac{\pi}{4}}\Big[\frac{2}{\beta}\Big]^{1/2}[1+\OO(\beta^{-1/2})]\,.  
 \end{array}
\end{equation}
By  \eqref{eq:33},  \eqref{eq:34}, and \eqref{eq:26}, we obtain that
\begin{equation}
  \label{eq:35}
|a_-^-|\leq C\beta^{-7/6} \,.
\end{equation}
By \eqref{eq:32}, \eqref{eq:33}, and \eqref{eq:34} we then obtain that 
\begin{equation}
  \label{eq:36}
|C_+|\leq C\beta^{-1/2}\, |C_-| \,.
\end{equation}

\paragraph{construction of the quasimode} \strut \\

In view of \eqref{eq:36}, it seems natural to try to
construct a quasimode where the leading order term of the solution of
\eqref{eq:13} is given near $x=-1$ (we introduce later a cut-off
function to ascertain that the boundary condition at $x=1$ is
satisfied) by
\begin{displaymath}
v_{0,\beta}= \check{\psi}_- (\cdot, \lambda^{(\alpha)}_{0,\beta})  
\end{displaymath}
where $\beta \lambda^{(\alpha)}_{0,\beta}$ is one of the leftmost eigenvalues (as a
matter of fact it is numerically shown in the appendix that for $\alpha=0$
the leftmost eigenvalue is unique) of $ \LL^{(\alpha)}_{0,\beta}$ whose
associated differential operator is the same as that of $ \LL_{0,\beta}$
and whose domain is given by
\begin{equation}\label{eq:17alpha}
  D(\LL^{(\alpha)}_{0,\beta})=\{u\in H^2(-1,\infty)\,|\,\langle e^{-\alpha(1+\cdot)},u\rangle=0,\; xu\in L^2(-1,\infty)\}\,,
\end{equation}
(which is compatible with (\ref{eq:11}b) in the case $\alpha=0$).
For the purpose of quasimode construction, we neglect $C_+\psi_+$ given
that the approximation of $U$ by $(1+x)$ is only valid near $x=-1$ and
since $\psi_+$ is exponentially small near $x=-1$.  The fast decay of
$v_{0,\beta}$ results from the fast decay of
$\check{\psi}_-$ away from $x=-1$ (see \eqref{eq:28}). \\

We also choose, for the purpose of the quasimode construction, to
replace $\alpha$ by $0$ for $\alpha \in [0,\alpha_0]$, $D(\LL^{(\alpha)}_{0,\beta})$ by $
D(\LL^{(0)}_{0,\beta}) $, and $\lambda^{(\alpha)}_{0,\beta}$ by
$\lambda_{0}(\beta)$. \\
To illustrate why it is a reasonable choice we write, with $\lambda= \lambda^{(\alpha)}_{0,\beta}\,$, 
     \begin{equation}
\label{eq:37}
      0=\langle e^{-\alpha(1+\cdot)},\check \psi_-\rangle=  \langle 1,\check \psi_-\rangle-  \langle 1-e^{-\alpha(1+\cdot)},\check \psi_-\rangle\,.
    \end{equation}
For the second term on the right-hand-side, we have by \eqref{eq:29}:
\begin{equation}
\label{eq:38}
  | \langle 1-e^{-\alpha(1+\cdot)},\check \psi_-\rangle|\leq C\, \|(1+\cdot)\check \psi_-\|_1\leq \hat C\,\beta^{-2/3}\,.
\end{equation}
 Hence, we have obtained
\begin{displaymath}
 \langle 1,\check \psi_- (\cdot\,,\, \lambda^{(\alpha)}_{0,\beta})\rangle=  \mathcal O (\beta^{-2/3})\,.
\end{displaymath}
Alternatively we may write (see \eqref{eq:zero})
\begin{displaymath}
A_0 ( i\beta^{-1/3} \lambda^{(\alpha)}_{0,\beta}) = \mathcal O (\beta^{-1/3})\,,
\end{displaymath}
and since $A_0$ and $A^\prime_0$ have no common zero (see
\eqref{eq:simplezero}), we obtain that $\beta^{-1/3} \lambda^{(\alpha)}_{0,\beta}$ should
reside in an $\mathcal O(\beta^{-2/3})$ vicinity of a zero of $\lambda \mapsto A_0 (i
\lambda)$. Hence, we should have 
\begin{displaymath}
d(\lambda^{(\alpha)}_{0,\beta},\sigma(\LL^{(0)}_{0,\beta}))\leq C \beta^{-1/3}\,.
\end{displaymath}

In the sequel, we shall thus consider on  $(-1,+\infty)$
\begin{equation}\label{eq:39}
v_{0,\beta}(x) := \check{\psi}_- (x,\lambda_{0}(\beta)) \,.
\end{equation}
Note that by \eqref{eq:39}, \eqref{eq:30}, and \eqref{eq:23}
\begin{equation}
\label{eq:40}
  \|v_0\|_2\approx\beta^{-1/6} \,,
\end{equation}
where $\approx$ means that there exist  $\beta_0$ and $C >0$ such that for $\beta \geq \beta_0$ 
\begin{displaymath}
\frac 1C \beta^{-1/6} \leq  \|v_0\|_2\leq C \beta^{-1/6}\,.
\end{displaymath}
 In addition, using  the fact that $v_0(-1)=1$ and \eqref{eq:28}
it holds that
\begin{equation}
\label{eq:41}
    \|v_0\|_\infty\approx1 \,.
  \end{equation}
  
  Let $\hat \phi_0$ denote the leading order approximation of $\phi$. In
  view of (\ref{eq:16}b) we choose it to be 
  the solution of $-\hat{\phi}_0^{\prime\prime}(\cdot ,\alpha)+\alpha^2\hat{\phi}_0(\cdot,\alpha)=v_0$. For
  $\alpha =0$ it holds that
\begin{equation}
\label{eq:42}
   \hat \phi_0(x)=-\int_{-1}^x(x-\xi)\,v_0(\xi)\,d\xi \,,
\end{equation}
whereas for $\alpha >0$, 
\begin{equation}
\label{eq:43}
\hat{\phi}_0(x,\alpha)=-\frac{1}{\alpha}\int_{-1}^x\sinh(\alpha[x-\xi])\,v_0(\xi)\,d\xi \,.
\end{equation}
We now compare between $\hat{\phi}_0(\cdot,\alpha)$ and $\hat \phi_0$. 
To this end we first observe that 
 for some $\gamma>0$ it holds by \eqref{eq:28}
that
\begin{equation}\label{eq:44}
  |v_0(x)|\leq Ce^{-\gamma\beta^{1/3}(1+x)} \,.
\end{equation}
For later reference we observe that for $k\geq 1$ \eqref{eq:44} implies  
\begin{equation}
\label{eq:45}
\|(1+\xi)^kv_0\|_1 \leq C\beta^{-(k+1)/3}\,.
\end{equation}
 For $1+x\leq 2$ 
we have
  \begin{displaymath}
    \hat{\phi}_0(x,\alpha)= -\frac{1}{\alpha}\int_{-1}^x\sinh(\alpha[1+x])\,v_0(\xi)\,d\xi
    +\frac{1}{\alpha}\int_{-1}^x[\sinh(\alpha[1+x])-\sinh(\alpha[x-\xi])]\, v_0(\xi)\,d\xi  
  \end{displaymath}
Since $\langle1,v_0\rangle=0$ it follows by \eqref{eq:44} that
\begin{displaymath}
  \Big|\frac{1}{\alpha}\int_{-1}^x\sinh(\alpha[1+x])\, v_0(\xi)\,d\xi\Big|\leq C\beta^{-1/3}(1+x)e^{-\gamma\beta^{1/3}(1+x)}\,.
\end{displaymath}
Furthermore, since
\begin{displaymath}
  \frac{1}{\alpha}\big[\sinh(\alpha[1+x])-\sinh(\alpha[x-\xi])\big]=
  (1+\xi)\cosh(\alpha[1+x]) +\OO(|1+\xi|^2) 
\end{displaymath}
We obtain, using \eqref{eq:45} with $k=2$,  that
\begin{displaymath}
  \frac{1}{\alpha}\int_{-1}^x\big[\sinh(\alpha[1+x])-\sinh(\alpha[x-\xi])\big]\,v_0(\xi)\,d\xi
  =\langle(1+\xi),v_0\rangle\cosh(\alpha[1+x])  +\OO(\beta^{-1}) \,. 
\end{displaymath}
Hence, 
\begin{subequations}\label{eq:46}
for  $\alpha >0$ 
\begin{equation}
  \hat{\phi}_0(x,\alpha)=\langle(1+\xi),v_0\rangle\cosh(\alpha[1+x]) +o(\beta^{-2/3}) \,,
\end{equation}
and for $\alpha=0$ 
\begin{equation}
  \hat{\phi}_0(x,0)=\langle(1+\xi),v_0\rangle+o(\beta^{-2/3}) \,.
\end{equation}
\end{subequations}
By \eqref{eq:45} and \eqref{eq:46} we obtain  that
\begin{displaymath}
  |\hat{\phi}_0(x,\alpha)-\hat{\phi}_0(x)|\leq C\beta^{-2/3}(1+x) \,.
\end{displaymath}
For $\alpha\in [0,\alpha_0]$ we therefore choose as our leading order
approximation $ \hat{\phi}_0$ as introduced in \eqref{eq:42}.  We note
that, for $1+x\ll1$, \eqref{eq:42} is a good approximation of
\eqref{eq:43}. For $1+x\gtrsim1$, we obtain from \eqref{eq:44} and
\eqref{eq:45} that
\begin{equation}
\label{eq:47}
  |\hat \phi_0(x)|\leq C\beta^{-2/3} \,,
\end{equation}
and hence both \eqref{eq:42} and \eqref{eq:43} are sufficiently
accurate in that regime.  We note further that
\begin{equation}
  \label{eq:48}
\|\hat \phi_0\|_2\leq C\beta^{-2/3} \,.
\end{equation}
Since $<1,v_0>_{L^2(-1,\infty)}=0\,$, we  obtain
  \begin{displaymath}
 \hat \phi_0^\prime(x) = \int_{-1}^x v_0(\xi) \, d\xi = -\int_x^\infty v_0(\xi) \, d\xi \,,
  \end{displaymath}
 and then immediately deduce   from \eqref{eq:44}
\begin{equation}
\label{eq:49}
  |\hat \phi_0^\prime(x)|\leq C\beta^{-1/3} \, e^{-\gamma\beta^{1/3}(1+x)}\,.
\end{equation}

We conclude this section by constructing a quasimode $\tilde{\phi}$ for
(\ref{eq:13}a) and then estimate for this choice the error
$\B_{\lambda,\alpha,\beta}\, \tilde{\phi}$ for $\lambda=\lambda_0(\beta)$. As already mentioned above,
our choice of $\hat \phi_0$ can be a good estimate in the close vicinity
of $x=-1$ only. Hence it is natural to introduce a suitable cut-off.
Let $\eta\in C^\infty(\R)$
\begin{equation}
\label{eq:50}
  \eta(x)=
  \begin{cases}
    1 & x<-\frac{1}{2} \\
    0 & x>0
  \end{cases}
  \,.
\end{equation}
Thus, we set
\begin{equation}
  \label{eq:51}
\tilde{\phi}(x)=\eta(x)\,\hat \phi_0(x)\,.
\end{equation}
Note that 
\begin{equation}
\label{eq:52}
  \tilde{\phi}(\pm1)=\tilde{\phi}^\prime(\pm1)=0\,,
\end{equation}
and consequently $\tilde{\phi}\in D(\B^\D_{\lambda_0(\beta),\alpha,\beta})$. \\
Recall that $\tilde{\phi}$ is independent of $\alpha$ (see
\eqref{eq:39}).  

It remains necessary to estimate the error.  To this
end we write, with $\lambda_0$ introduced in \eqref{eq:12}, and $\phi_0 =\hat
\phi_0$
\begin{displaymath}
   \B_{\lambda_0,\alpha,\beta} (\eta \phi_0) = \eta \,\B_{\lambda_0,\alpha,\beta}( \phi_0) + [ \eta, \B_{\lambda_0,\alpha,\beta}] \phi_0\,.
\end{displaymath}
   Given that the support of the derivatives of $\eta$ is contained in
$[-1/2,0]$ and by \eqref{eq:44} and \eqref{eq:49} it holds that:
\begin{displaymath}
    \B_{\lambda_0,\alpha,\beta} (\eta \phi_0) = \eta  \B_{\lambda_0,\alpha,\beta}( \phi_0) +  
    \phi_0\,\big[\eta^{(4)}+i\beta(U+i\lambda)\eta^{\prime\prime}-\alpha^2\eta^{\prime\prime}\big] +\mathcal O ( e^{-\gamma\beta^{1/3}/2})\,.
 \end{displaymath}
By \eqref{eq:47} it holds that
\begin{displaymath}
  \|\phi_0\,[\eta^{(4)}+i\beta(U+i\lambda)\eta^{\prime\prime}-\alpha^2\eta^{\prime\prime}]\|_2\leq C\beta^{1/3}
\end{displaymath}
 It remains necessary to estimate the error in $L^2(-1,1)$ due to the
 replacement of $U$
 by $(1+x)$, which is given, by \eqref{eq:39}, \eqref{eq:30} and
 \eqref{eq:47} by
   \begin{displaymath}
     \beta \,  \|   (U-1-x)v_0-U^{\prime\prime}\phi_0\| \leq   C\, \beta^{1/6} + \hat C \beta^{1/3}\,,
   \end{displaymath}
   and finally the error due to $\alpha^2$ which, using \eqref{eq:40} and
   \eqref{eq:47} obeys the following bound
      \begin{displaymath}
         \|  \alpha^2 \Big[\beta  (U+i\lambda_0) \phi_0 + v_0)\Big] \| \leq  C \beta^{1/3}\,.
      \end{displaymath}
Combining the above yields
\begin{equation}
\label{eq:53}
  \|\B_{\lambda_0,\alpha,\beta}\tilde{\phi}\|_2\leq C\,\beta^{1/3} \,,
\end{equation}
and hence by \eqref{eq:40}, 
\begin{displaymath}
    \|\B_{\lambda_0,\alpha,\beta}\tilde{\phi}\|_2\leq C\,\beta^{1/2}\,\|v_0\|_2 \,. 
\end{displaymath}
By \eqref{eq:48}, we can conclude that $\tilde{\phi}$ satisfies 
\eqref{eq:goal}.

\section{Preliminaries}
\label{sec:3}

In this section we obtain some estimates that will become useful in 
Section 5 where we prove the existence of the critical value of
$\B_{\lambda,\alpha,\beta}$ in the vicinity of $\lambda_0(\beta)$.   We begin by briefly
recalling some Hardy inequalities, as we did in \cite{AH-ems}. Then,
we obtain some additional resolvent estimates for the Dirichlet realization of the 
differential Schr\"odinger operator introduced in \eqref{eq:7}.

\subsection{Hardy's inequality}
\label{sec:3.1}

We begin by recalling  \cite[Lemma 2.2.3]{AH-ems}
\begin{lemma}
   Let $w\in H^1(a,b)$ satisfy $w(b)=0$. Then,  we have
\begin{equation}
\label{eq:54}
  \|([x-a]w)^\prime\|_2^2 =\|[x-a]w^\prime\|_2^2\geq \frac{1}{4}\, \|w\|_2^2 \,.
\end{equation}
\end{lemma}

If we drop the requirement that $w(b)=0$ we can still state the
following \cite[Lemma 2.2.4]{AH-ems}
\begin{lemma}
   Let $w\in H^1(a,b)$. Then,  we have
\begin{equation}
\label{eq:55}
  \|([x-a]w)^\prime\|_2^2\geq \frac{1}{4}\, \|w\|_2^2 \,.
\end{equation}
\end{lemma}

By \eqref{eq:54} and \eqref{eq:10} we may then conclude that for $w\in
H^1(-1,1)$ satisfying $w(1)=0$ we have
\begin{equation}
\label{eq:56}
  \|[U-U(-1)]w^\prime\|_2^2 \geq\frac{ \Xi}{4}\,\|w\|_2^2\,.
\end{equation}
Moreover, for any $v\in H^1(-1,1)$ we may also conclude, by \eqref{eq:55}, that as long as
$v(-1)=0$ it holds, for any $a\in(-1,1]$
\begin{equation}
  \label{eq:57}
\Big\|\frac{v}{1+x}\Big\|_2^2 \leq
4\,\Big\|\Big[(1+x)\frac{v}{1+x}\Big]^\prime\Big\|_2^2 =4\,\|v^\prime\|_2^2 \,.
\end{equation}

\subsection{Dirichlet Schr\"odinger estimates}
\label{sec:3.2}
We split the discussion in this section according to the interval on
which the Schr\"odinger operator is acting.
\subsubsection{Resolvent estimates in $[a,1]$}
Let $a\in[-1,0]$ and $I_a=(a,1)$.  In this subsection we consider the
operator $\LL_{\beta}^{\D,I_a}$ whose associate differential operator is
given by \eqref{eq:7} and whose domain is $H^2(a,1)\cap H^1_0(a,1)$.  We
begin with the following estimate
\begin{lemma}
  Let $\mu_0>0$ and $a\in[-1,0]$. Let further $U\in C^2([a,1])$ satisfy
   for
  some $\Xi >0$
  \begin{equation}
    \label{eq:58}
U(x)\geq U(a)+ \Xi \, (x-a) \,,
  \end{equation}
(see also \eqref{eq:10}).
  Then, there exists $\beta_0>0$ and $C>0$ such that for all
  $\beta\geq \beta_0$ and $\lambda=\mu+i\nu$ satisfying $\nu\leq U(a)-\beta^{-1/3}$ and
  $\mu\leq \mu_0\beta^{-1/3}$, it holds
that
   $\LL^{\D,I_a}_{\beta} $ is invertible and satisfies
  \begin{equation}
    \label{eq:59}
\|(\LL^{\D,I_a}_{\beta,U} -\beta\lambda)^{-1}\|_2+\big[\beta|U(a)-\nu|\big]^{1/2}\Big\|\frac{d}{dx}(\LL^{\D,I_a}_{\beta,U} -\beta\lambda)^{-1}\Big\|_2\leq C \beta^{-1} (U(a)-\nu)^{-1}  \,.
  \end{equation}
\end{lemma}
\begin{proof}
  Integration by parts yields for any $u\in D(\LL_{\beta,U}^{\D,(a,1)})$ 
  \begin{equation}
\label{eq:60}
    \Im\langle u,(\LL_{\beta,U} -\beta\lambda)u\rangle= \beta\langle u,(U-\nu)u\rangle \,,
  \end{equation}
where $\langle\cdot,\cdot\rangle$ denotes the standard product in $L^2(a,1)$:
\begin{displaymath}
  \langle f,g\rangle=\int_a^1f(x)\bar{g}(x)\,dx
\end{displaymath}
From  \eqref{eq:58} and the fact that $\nu<U(a)$  we can conclude that
\begin{displaymath}
  \beta[U(a)-\nu]\,\|u\|_2^2\leq \|(\LL_{\beta,U} -\beta\lambda)u\|_2\,\|u\|_2 \,,
\end{displaymath}
from which we conclude that
\begin{equation}
  \label{eq:61}
\|u\|_2\leq \beta^{-1} \, |U(a)-\nu|^{-1} \,\|(\LL_{\beta,U} -\beta\lambda)u\|_2
\end{equation}
To obtain an estimate for $u^\prime$ we use the identity
\begin{displaymath} 
   \Re\langle u,(\LL_{\beta,U} -\beta\lambda)u\rangle=\|u^\prime\|_2^2- \mu\beta\|u\|_2^2 \,.
\end{displaymath}
Using \eqref{eq:61} and the assumed bounds $\mu\leq\mu_0\,\beta^{-1/3}$ and
$U(a)-\nu>\beta^{-1/3}$ yields
\begin{equation}
\label{eq:62}
  \|u^\prime\|_2\leq C\, [\beta|U(a)-\nu|]^{-1/2}\,\|(\LL_{\beta,U} -\beta\lambda)u\|_2\,.
\end{equation}
Given that for $\beta=0$ the Fredholm index of $\LL^{\D,I_a}_{\beta,U}$ is
zero, it follows that its index vanishes for all $\beta>0$, and hence by
\eqref{eq:61} we may conclude its bijectivity and \eqref{eq:59}. 
\end{proof}
\subsubsection{Resolvent estimates in $(-1,+\infty)$}
\label{sec:dirichlet-infinite}
We begin the discussion by considering the case $\beta=1$ on the interval $(0,+\infty)$ before
using translation and dilation to extend the validity of the results to
our case of interest .  Hence, we seek to obtain first some resolvent estimates for
a Dirichlet Schr\"odinger on $\R_+$. Let then $\LL_+$ be given by
\begin{displaymath}
  \LL_+=-\frac{d^2}{dx^2}+ix
\end{displaymath}
and be defined on  $H^2(\R_+)\cap H^1_0(\R_+)$. It is well known (see
\cite{al08,hen14}) 
that
\begin{displaymath}
  \sigma(\LL_+)=\{\nu_k\}_{k=1}^\infty  \,,
\end{displaymath}
where $\arg \nu_k=\pi/3$ for all $k\geq1$ and $|\nu_k|\to\infty\,$. 
\begin{lemma}
  Let $\Upsilon>0\,$. Then, there exists $C>0$ such that, for all $\lambda\in\C \setminus \sigma(\LL_+)$,
  satisfying $\Re\lambda\leq \Upsilon$ it holds that
  \begin{equation}
    \label{eq:63}
\|(\LL_+-\lambda)^{-1}\|_{\LL(L^2(\mathbb R_+))}+
\Big\|\frac{d}{dx}\,(\LL_+-\lambda)^{-1}\Big\|_{\LL(L^2(\mathbb R_+))} \leq
C\Big[1+\frac{1}{d(\lambda,\sigma(\LL_+))}\Big] \,. 
  \end{equation}
 Furthermore, let $\nu_0>0$, then there exists $C>0$ such that for all $\lambda$ such that 
$\Im\lambda<\nu_0$  and $\Re\lambda\leq \Upsilon$, it holds that
\begin{equation}
\label{eq:64}
  \|(\LL_+-\lambda)^{-1}\|_{\LL(L^\infty(\mathbb R_+),L^2(\mathbb R_+))}\leq
  C\Big(1+\frac{1}{d(\lambda,\sigma(\LL_+))}\Big)[1+
  \log_+^{1/2} d^{-1}(\lambda,\sigma(\LL_+))] \,. 
\end{equation}
where $\log_+x=\max(0,\log x)$.
\end{lemma}
\begin{proof}
  The proof of \eqref{eq:63} has first been established in
  \cite{he11}. 
  
To prove \eqref{eq:64},  let $(\chi_n)_{n\in \mathbb Z}$ be a sequence of  $C^\infty$ functions such that
\begin{displaymath}
\sum_{n =-\infty}^{+\infty} \chi_n^2 = 1\,,\, {\rm supp \, \chi_n} \subset [n-1,n+1] \mbox{ and
} |\chi^\prime_n|+ |\chi^{\prime\prime}_n|\leq \hat C\,.
\end{displaymath}
Let $(\lambda,v,g)\in \{\C \setminus \sigma(\LL_+)\} \times  D(\LL_+)\times L^\infty(\R_+)$ satisfy
\begin{equation}
\label{eq:65}
  (\LL_+-\lambda)v=g\,.
\end{equation}
Taking the inner product of \eqref{eq:65}  with $\chi_n^2v$ yields, for $n>N > \nu_0 +1$ where
$\nu=\Im\lambda <  \nu_0$, for the imaginary part
\begin{multline*}
  \|(x-\nu)^{1/2}\chi_nv\|_2^2+2\, \Im\langle\chi_n^\prime v,\chi_nv^\prime\rangle=\Im\langle\chi_nv,\chi_ng\rangle\leq \\
\leq   \frac{n-1-\nu}{2}\|\chi_nv\|_2^2+\frac{1}{2(n-1-\nu)}\|\chi_ng\|_2^2 \,,
\end{multline*}
from which we conclude that, for $n>N > \nu_0 +1$ ,
\begin{equation}
\label{eq:66}
  \|\chi_nv\|_2^2\leq C\Big(\frac{1}{n^2}\|\chi_n g\|_\infty^2+
  \frac{1}{n}\|\chi_n^\prime v\|_2\,\|\chi_nv^\prime\|_2\Big)\,. 
\end{equation}

From the real part of the same inner product we obtain that
\begin{displaymath}
  \|(\chi_nv)^\prime\|_2^2-\|\chi_n^\prime v\|_2^2=\mu\|\chi_nv\|_2^2+\Re\langle\chi_nv,\chi_ng\rangle\,,
\end{displaymath}
where $\mu=\Re\lambda$, yielding 
\begin{displaymath}
   \|(\chi_nv)^\prime\|_2^2\leq \|\chi_n^\prime v\|_2^2+\mu\|\chi_nv\|_2^2+\frac 12 \big(n\|\chi_nv\|_2^2
   +\frac{1}{n}\|\chi_ng\|_2^2\big)\,. 
\end{displaymath}
With the aid of \eqref{eq:66} and the condition $\mu \leq \Upsilon$, the above yields
\begin{equation*}
   \|(\chi_nv)^\prime\|_2^2\leq C_1 \Big(\|\chi_n^\prime v\|_2^2+n^{-1}\|g\|_\infty^2\Big)\,,
\end{equation*}
and then
\begin{equation}
\label{eq:67}
   \|\chi_n\, v^\prime\|_2^2\leq C_2 \Big(\|\chi_n^\prime v\|_2^2+n^{-1}\|g\|_\infty^2\Big)\,. 
\end{equation}
Substituting into \eqref{eq:66} yields
\begin{displaymath}
    \|\chi_nv\|_2^2\leq C\, \Big(\frac{1}{n^2}\|g\|_\infty^2+
  \frac{1}{n}\|\chi_n^\prime v\|_2^2\Big)\,. 
\end{displaymath}
 Summing over $n$ between $N$ and $\infty$ yields
\begin{equation}
\label{eq:68}
  \|{\bf 1}_{[N,\infty)}v\|_2^2\leq C N^{-1} \big(\|g\|_\infty^2+ \|{\bf 1}_{[N-1,N+1]}v\|_2^2\big)\,.
\end{equation}
 Given that, for $N >\nu_0+2\,$, 
\begin{displaymath}
  \|{\bf 1}_{[N-1,N+1]}v\|_2^2 \leq\|{\bf 1}_{[N-1,\infty)}v\|_2^2\leq C (N-1)^{-1} \big(\|g\|_\infty^2+ \|{\bf 1}_{[N-2,N]}v\|_2^2\big)\,,
\end{displaymath}
we may conclude from \eqref{eq:68} that
\begin{displaymath}
  \|{\bf 1}_{[N,\infty)}v\|_2^2\leq C \Big(N^{-1} \|g\|_\infty^2+ C[N(N-1)]^{-1}\|{\bf 1}_{[N-2,N]}v\|_2^2\Big)\,.
\end{displaymath}
We may recursively proceed that above process, to obtain for any
$N\geq L>\nu+1$
\begin{equation}
\label{eq:69}
    \|{\bf 1}_{[N,\infty)}v\|_2^2\leq CN^{-1} \|g\|_\infty^2+\frac{C^{N-L +1}(L-1)!}{N!}\|{\bf 1}_{[L-1,L+1]}v\|_2^2\,.
\end{equation}

We continue with the estimate of ${\bf 1}_{[0,N]}v$. We  choose $\eta_N$ such that
\begin{displaymath}
\eta_N=1 \mbox{ on } (-\infty,N) \,,\, {\rm supp \,} \eta_N \subset (-\infty, 2N)\, , |\eta^\prime_N| + N |\eta^{\prime\prime}_N| \leq \hat C N^{-1} \,.
\end{displaymath}
By \eqref{eq:65} it holds that
\begin{displaymath}
   (\LL_+-\lambda)(\eta_Nv)=2\eta_N^\prime v^\prime+\eta_N^{\prime\prime}v+ \eta_N g \,.
\end{displaymath}
From \eqref{eq:63} we then obtain that
\begin{equation}
\label{eq:70}
  \|\eta_Nv\|_2\leq  C\, \Big[1+\frac{1}{d(\lambda,\sigma(\LL_+))}\Big] \,\Big(N^{-1}\|{\bf 1}_{[N,2N]}v^\prime\|_2+N^{-2}\|{\bf
    1}_{[N,2N]}v\|_2+  N^{1/2} \|g\|_\infty \Big) \,,
\end{equation}
To estimate ${\bf 1}_{[N,2N]}v^\prime\,,$ we use \eqref{eq:67}, which after
summation over $n$  between $N$ and $2N$ yields
\begin{equation}\label{eq:71}
  \|{\bf 1}_{[N,2N]}v^\prime\|_2^2\leq C\,\big[\|{\bf 1}_{[N-1,2N+1]}v\|_2^2+\|g\|_\infty\big] \,.
\end{equation}

Combining \eqref{eq:71}  with \eqref{eq:70} yields
\begin{displaymath}
    \|\eta_Nv\|_2\leq C \, \Big[1+\frac{1}{d(\lambda,\sigma(\LL_+))}\Big]  \big(N^{-1}\|{\bf
    1}_{[N-1,2N+1]}v\|_2+{N}^{1/2}\|g\|_\infty\big) \,,
\end{displaymath}
from which we immediately obtain
\begin{displaymath}
    \|{\bf 1}_{[0,N]}v\|_2\leq C \, \Big[1+\frac{1}{d(\lambda,\sigma(\LL_+))}\Big] (N^{-1}\|{\bf
    1}_{[N-1,\infty]}v\|_2+{N^{1/2}}\|g\|_\infty) \,.
\end{displaymath}
Together with \eqref{eq:69} (with $N$ replaced by $N-1$ ) the above
inequality yields  
\begin{displaymath} 
  \|{\bf 1}_{[0,N]}v\|_2\leq C \,  \Big[1+\frac{1}{d(\lambda,\sigma(\LL_+))}\Big]
  \Big( N^{1/2}\|g\|_\infty +  \Big[\frac{C^{N-L}L!}{N!}\Big]^{1/2}\|{\bf 1}_{[L-1,L+1]}v\|_2\Big)\,,
\end{displaymath}

For any fixed $L>\nu+1$, it follows that there exists $K$ such that, if 
\begin{displaymath}
  N>K\big[1+\log_+d^{-1}(\lambda,\sigma(\LL_+))\big]\,,
\end{displaymath}
then
\begin{displaymath}
 C  \Big[\frac{C^{N-L}L!}{(N-1)!}\Big]^{1/2}\Big[1+\frac{1}{d(\lambda,\sigma(\LL_+))}\Big]\leq
  \frac{1}{2} \,,
\end{displaymath}
 and we obtain that
\begin{equation}\label{eq:60aa}
  \|{\bf 1}_{[0,N]}v\|_2\leq C \,  \Big[1+\frac{1}{d(\lambda,\sigma(\LL_+))}\Big]
  \log^{1/2}\Big(1+\frac{1}{d(\lambda,\sigma(\LL_+))}\Big)  \,  \|g\|_\infty \,.
\end{equation}
Combining  \eqref{eq:60aa} with \eqref{eq:68} yields \eqref{eq:64}. 
\end{proof}

An immediate corollary,  using dilation and
translation,  follows. 
\begin{corollary}
  Let $\LL_{0,\beta}^{\D,(-1,\infty)}$ be the operator  associated with the same differential
  operator as for $\LL_{0,\beta}$ in \eqref{eq:17} and with domain
\begin{displaymath}
  D(\LL_{0,\beta}^{\D,(-1,\infty)})=\{u\in H^2(-1,\infty)\cap H^1_0(-1,\infty)\,|\,xu\in L^2(-1,\infty)\}\,.
\end{displaymath}
Then, for any $\Upsilon>0$ there exists $C>0$ such that
\begin{multline}
  \label{eq:72}
\sup_{\Re\lambda\leq\Upsilon\beta^{-1/3}}[\|(\LL_{0,\beta}^{\D,(-1,\infty)}-\beta\lambda)^{-1}\|
+\beta^{-1/3}\|\frac{d}{dx} (\LL_{0,\beta}^{\D,(-1,\infty)}-\beta\lambda)^{-1}\|] \\
\leq  C\beta^{-2/3}\Big[1+ \frac{1}{d(\beta^{1/3}\lambda,\sigma(\LL_+))}\Big]\,.
\end{multline}
 Furthermore, for any $\nu_0>0$ it holds that
\begin{multline}
  \label{eq:73}
\sup_{
  \begin{subarray}{c}
    \Re\lambda\leq\Upsilon\beta^{-1/3} \\
   \Im\lambda\leq\nu_0\beta^{-1/3} 
  \end{subarray}}
\|(\LL_{0,\beta}^{\D,(-1,\infty)}-\beta\lambda)^{-1}\|_{\LL(L^\infty,L^2)}  \\  \leq C\beta^{-5/6}\Big[1+ \frac{1}{d(\beta^{1/3}\lambda,\sigma(\LL_+))}\Big][1+
  \log_+^{1/2} d^{-1}(\beta^{1/3}\lambda,\sigma(\LL_+))]\,.
\end{multline}
\end{corollary}

\subsubsection{Resolvent estimates in $(-1,+1)$}
Given  $U\in C^2([-1,1])$ satisfying \eqref{eq:9} and $U(-1)=0$,  we now consider  and denote by  $ \LL_{\beta,U}^{\D,(-1,1)}$ the Dirichlet realization in $(-1,+1)$  of the differential operator
\begin{displaymath}
  \LL_{\beta,U}=-\frac{d^2}{dx^2}+i\beta U \,,
\end{displaymath}
defined on $H^2(-1,1)\cap H^1_0(-1,1)$. \\
For some positive $r$ and $b$ let 
\begin{displaymath}
  \Sg_{b,r} =\{z\in\C \,| \,\max(\Re z,|\Im z|)\leq b \text{ and }d(z,\sigma(\LL_+))\geq r\}\,.
\end{displaymath}
We can now use \eqref{eq:72} and \eqref{eq:73} to establish
that (see \cite[Eq. (5.8)]{AH-arma} and  \cite[Eq. (5.42)]{AH-arma})
\begin{lemma}  For any $\Upsilon>0$
  and $r>0$ there exist positive $C$  and $\beta_0$ such that, 
    for all $\beta \geq \beta_0$ and  $\lambda\in\Sg_{\Upsilon\beta^{-1/3},r\beta^{-1/3} }$\,,  the
    operator $\LL_{\beta,U}^{\D,(-1,1)} -\beta\lambda$ is invertible and satisfies 
  \begin{equation}
  \label{eq:74}
\|(   \LL_{\beta,U}^{\D,(-1,1)}-\beta\lambda)^{-1}\| +\beta^{-1/3}\, \Big\|\frac{d}{dx} (
\LL_{\beta,U}^{\D,(-1,1)}-\beta\lambda)^{-1}\Big\| \leq  C\beta^{-2/3}\,. 
\end{equation}
Similarly,  for any $\Upsilon>0$
  and $r>0$, there exists $C>0$ and $\beta_0$ such that,  for any $g\in
L^\infty(-1,1)$  and any $\beta \geq \beta_0$,
\begin{equation}
\label{eq:75}
  \sup_{\lambda\in\Sg_{\Upsilon\beta^{-1/3},r\beta^{-1/3}}} \|(
  \LL_{\beta,U}^{\D,(-1,1)}-\beta\lambda)^{-1}g\|_1\leq C \, \min
  \Big(\frac{\log\beta}{\beta}\|g\|_\infty,\beta^{-5/6}\|g\|_2\Big)\,. 
\end{equation}
\end{lemma}
\begin{proof}~\\
{\bf Proof of \eqref{eq:74}}\\
  Let $\eta\in C^\infty(\R)$ be given by \eqref{eq:50}, and set for some $\delta>0$,
  \begin{subequations}\label{eq:76}
 \begin{equation}
  \eta_\delta(x)=\eta\Big(\frac{1+x}{2\delta}-1\Big)\,.
 \end{equation}
Hence, we have
\begin{equation}
 \eta_\delta = 1 \mbox{ for } x<-1+\delta \mbox{ and } \eta_\delta=0\mbox{ for } x > -1 + 2\delta\,.
\end{equation}
\end{subequations}
 Let $\lambda\in\Sg_{\Upsilon\beta^{-1/3},r\beta^{-1/3}}$ and $(v,g)\in D(   \LL_{\beta,U}^{\D,(-1,1)})\times L^2(-1,1)$ satisfy
\begin{equation}
\label{eq:77}
  (   \LL_{\beta,U}-\beta\lambda)v=g \,.
\end{equation}
Clearly, $\eta_\delta v\in D(\LL_{0,\beta}^{\D,(-1,\infty)})$ and we may write
\begin{displaymath}
   ( \LL_{0,\beta}-\beta\lambda)(\eta_\delta v)=-i\beta[U-(1+x)]\eta_\delta v-2\eta_\delta^\prime v^\prime-\eta_\delta^{\prime\prime}v+\eta_\delta g \,.
\end{displaymath}
By \eqref{eq:72} we then have
\begin{displaymath}
  \|\eta_\delta v\|_2 +\beta^{-1/3}\|(\eta_\delta v)^\prime\|_2
  \leq C\, \Big[\beta^{1/3}\delta^2\|\eta_\delta v\|_2+\beta^{-2/3}(\delta^{-1}\|v^\prime\|_2+\delta^{-2}\|v\|_2+\|\eta_\delta g\|_2)\Big]\,.
\end{displaymath}
For $\delta^2\beta^{1/3}\leq(2C)^{-1}$ we then obtain 
\begin{equation}
\label{eq:78}
   \|\eta_\delta v\|_2 +\beta^{-1/3}\|(\eta_\delta v)^\prime\|_2
  \leq \widehat C\beta^{-2/3}(\delta^{-1}\|v^\prime\|_2+\delta^{-2}\|v\|_2+\|\eta_\delta g\|_2)\,.
\end{equation}
   We now write,
\begin{displaymath}
  \eta_\delta v=v_1+v_2 \,,
\end{displaymath}
with $v_1,v_2 \in D(\LL_{\beta,U}^{\D,(-1,+\infty)})$ satisfying
\begin{displaymath}
    ( \LL_{0,\beta}-\beta\lambda)v_1=-i\beta[U-(1+x)]\eta_\delta v-2\eta_\delta^\prime v^\prime-\eta_\delta^{\prime\prime}v
\end{displaymath}
and
\begin{displaymath}
    ( \LL_{0,\beta}-\beta\lambda)v_2=\eta_\delta g\,.
\end{displaymath}
 By \eqref{eq:73} and the fact that $\lambda\in\Sg_{\Upsilon\beta^{-1/3},r\beta^{-1/3} }$ it holds   that
\begin{displaymath}
  \|v_2\|_2\leq C\beta^{-5/6}\|\eta_\delta g\|_\infty \,,
\end{displaymath}
whereas for $v_1$ we obtain, as in \eqref{eq:78}, 
\begin{displaymath} 
   \|v_1\|_2 \leq \widehat C\beta^{-2/3}(\delta^{-1}\|v^\prime\|_2+\delta^{-2}\|v\|_2) + \check C \beta^{1/3} \delta^2 \|\eta_\delta v\|_2\,.
\end{displaymath}
Combining the above yields, for sufficiently
small $\delta^2\beta^{1/3}$, 
   \begin{equation}
\label{eq:79}
     \|\eta_\delta v\|_2 \leq \check C\,\Big [\beta^{-2/3}\big(\delta^{-1}\|v^\prime\|_2+\delta^{-2}\|v\|_2\big)+\beta^{-5/6}\|\eta_\delta g\|_\infty\Big]\,.
  \end{equation}
~\\

Let $\tilde{\eta}_\delta=1-\eta_\delta$. Here we observe that $\tilde{\eta}_\delta v \in
D(\LL_{\beta,U}^{\D, ( a,1)} )$  with $a=-1+\delta$ and a  simple computation
yields, 
\begin{displaymath}
  (   \LL_{\beta,U}-\beta\lambda)(\tilde{\eta}_\delta v)=2\eta_\delta^\prime v^\prime+\eta_\delta^{\prime\prime}v+\tilde{\eta}_\delta g \,.
\end{displaymath}
 We now choose 
 \begin{equation}
\label{eq:80}
\delta =\hat \delta_0\beta^{-1/6} \mbox{  with } \hat \delta_0 \mbox{ small enough}.
 \end{equation}
 Note that $U(-1+\delta)\approx\beta^{-1/6}$ and hence, since $\nu\lesssim\beta^{-1/3}$, we may
 use \eqref{eq:59} to obtain  
\begin{equation*}
    \|\tilde{\eta}_\delta v\|_2 +\big[\beta|U(-1+\delta)-\nu|\big]^{-1/2}\| (\tilde{\eta}_\delta v)^\prime\|_2
  \leq\frac{C}{\beta|U(-1+\delta)-\nu|}(\delta^{-1}\|v^\prime\|_2+\delta^{-2}\|v\|_2+\|\tilde{\eta}_\delta g\|_2)\,.
\end{equation*}
The above inequality yields in view of \eqref{eq:80}
\begin{subequations}\label{eq:81}
\begin{equation}
  \|\tilde{\eta}_\delta v\|_2 \leq C\, \beta^{-5/6} \big(\delta^{-1}\|v^\prime\|_2+\delta^{-2}\|v\|_2+\|\tilde{\eta}_\delta g\|_2\big)\,,
\end{equation}
and
\begin{equation}
\| (\tilde{\eta}_\delta v)^\prime\|_2 \leq C \beta^{-5/12}  \big(\delta^{-1}\|v^\prime\|_2+\delta^{-2}\|v\|_2+\|\tilde{\eta}_\delta g\|_2\big)\,.
\end{equation}
\end{subequations}
Combining \eqref{eq:81}  with \eqref{eq:78} yields, in view of \eqref{eq:80},
\begin{subequations}
\label{eq:82}
  \begin{equation}
  \|v\|_2 \leq C\beta^{-2/3}\big(\delta^{-1}\|v^\prime\|_2+\delta^{-2}\|v\|_2+\|g\|_2\big) \,,
\end{equation}
and
\begin{equation}
  \|v^\prime\|_2\leq C\beta^{-1/3}\big(\delta^{-1}\|v^\prime\|_2+\delta^{-2}\|v\|_2+\|g\|_2\big) \,.
\end{equation}
The inequality (\ref{eq:82}b)  gives for $\beta$ large enough
\begin{equation}
 \|v^\prime\|_2\leq C\beta^{-1/3}\big(\delta^{-2}\|v\|_2+\|g\|_2\big) \,.
\end{equation}
Substituting (\ref{eq:82}c) into (\ref{eq:82}a) yields for
sufficiently large $\beta$ 
\begin{equation}
\|v\|_2 \leq C\beta^{-2/3} \|g\|_2  \,.
\end{equation}
Combining (\ref{eq:82}d) and (\ref{eq:82}c)  leads to 
\begin{equation}
 \|v^\prime\|_2\leq C\beta^{-1/3} \|g\|_2 \,.
\end{equation}
\end{subequations}
From (\ref{eq:82}d,e) we may conclude the existence of $\beta_0$
such that for  all $\beta \geq \beta_0$ \eqref{eq:74} holds true.

{\bf Proof of \eqref{eq:75}\\}
 To prove \eqref{eq:75} we first observe, that for $0\leq\nu\leq\gamma\beta^{-1/3}$
there exists a unique $x_\nu\in[-1,1]$ such that $U(x_\nu)=\nu\,$. For $\nu<0$ we
set $ x_\nu=-1\,$. Furthermore, 
it holds by \eqref{eq:9} that for some $0<\Xi<1$,
\begin{equation}
\label{eq:83}
  |U-\nu|\geq\Xi\, |x-x_\nu|\,.
\end{equation}
  By \eqref{eq:83} it holds that
\begin{equation*} 
   \|(U-\nu+i \beta^{-1/3})^{-1}\|_2^2 \leq
   C \, \beta^{1/3} \,. 
\end{equation*}
 We may then conclude that
  \begin{equation}
\label{eq:84}
  \begin{array}{ll}
    \|v\|_1 &\leq \|(U-\nu+i\beta^{-1/3})^{-1}\|_2
    \|(U-\nu+i\beta^{-1/3})v\|_2\\ & \leq C\beta^{1/6}\, \big[\|(U-\nu)v\|_2+\beta^{-1/3}\|v\|_2\big]\,.
    \end{array}
  \end{equation}
Taking the inner product of \eqref{eq:77} by $(U-\nu) v$ and integrating by
parts yields for the imaginary parts
\begin{displaymath}
  \beta\|(U-\nu)v\|_2^2+\Im\langle U^\prime v,v^\prime\rangle=\Im\langle(U-\nu)v,g\rangle\,, 
\end{displaymath}
which together with \eqref{eq:74} leads to
\begin{equation}
\label{eq:85}
  \|(U-\nu)v\|_2\leq  C \beta^{-1} \,\|g\|_2 \,.
\end{equation}
Substituting  the above into \eqref{eq:84} we obtain, with the aid of
\eqref{eq:74}
\begin{equation}
  \label{eq:86}
\|v\|_1 \leq C\beta^{-5/6}\|g\|_2\,.
\end{equation}
To complete the proof of \eqref{eq:75}, we use \eqref{eq:83} to obtain
that
\begin{equation*}
\|(U-\nu+  i \beta^{-1/3})^{-1/2}\|_2  \leq    \|(\Xi \,[x-x_\nu]+i
\beta^{-1/3})^{-1}\|_1 \leq C\, \log\beta \,.
\end{equation*}
Then, we write
\begin{equation}
\label{eq:87}
\begin{array}{ll}
    \|v\|_1&  \leq \|(U-\nu+\beta^{-1/3})^{-1/2}\|_2
    \|(U-\nu+i\beta^{-1/3})^{1/2}v\|_2\\ &
    \leq C\,\log \beta \, \big[\||U-\nu|^{1/2}v\|_2+\beta^{-1/6}\|v\|_2\big]\,.
    \end{array}
\end{equation}
Given that $|\nu|\leq\Upsilon\beta^{-1/3}$ it holds that
\begin{displaymath}
  \||U-\nu|^{1/2}v\|_2^2\leq \langle(U-\nu)v,v\rangle+C\beta^{-1/3}\|v\|_2^2 \,.
\end{displaymath}
By \eqref{eq:60} (with $a=-1$)  it holds that 
\begin{displaymath}
  \langle(U-\nu)v,v\rangle\leq \beta^{-1}\, \|v\|_1\|g\|_\infty \,,
\end{displaymath}
and hence
\begin{displaymath}
  \||U-\nu|^{1/2}v\|_2^2\leq C\,\big[\beta^{-1} \|v\|_1\|g\|_\infty +\beta^{-1/3}\|v\|_2^2\big]\,.
\end{displaymath}
Substituting the above into \eqref{eq:87} yields 
\begin{equation}
  \label{eq:88}
  \|v\|_1\leq C\,[\beta^{-1}\log \beta \, \|g\|_\infty +\beta^{-1/6} \log \beta \, \|v\|_2]\,.
\end{equation}
 To complete the proof we need to establish first that 
\begin{equation}
\label{eq:89}
  \|v\|_2\leq C\beta^{-5/6}\|g\|_\infty \,.
\end{equation}
To this end we use \eqref{eq:78}, \eqref{eq:79}, and \eqref{eq:81}. 
Substituting \eqref{eq:89} into \eqref{eq:88} yields for $\beta$ large enough 
\begin{displaymath}
  \|v\|_1\leq C\beta^{-1}\log \beta \, \|g\|_\infty \,,
\end{displaymath}
which together with \eqref{eq:86} establishes \eqref{eq:75}. 
\end{proof}

\section{A no-slip Schr\"odinger operator}
In the following, we refine \cite[Lemma 6.2]{AH-arma} to include the
case where 
\begin{displaymath}
 \delta_0^{-1}\beta^{-2/3}\leq |\lambda-\lambda_0(\beta)|\leq \delta_0\beta^{-1/3}\,,
\end{displaymath}
where $\delta_0>0$  is sufficiently small and $\lambda_0(\beta) \in\C$ is defined
 in \eqref{eq:12}. 

\subsection{No-slip Schr\"odinger on $\R_+$}

  Let $\tilde{\LL}_0$ be given by 
  \begin{displaymath}
    \tilde{\LL}_0=-\frac{d^2}{dx^2}+x
  \end{displaymath}
  where
  \begin{displaymath}
    D(\tilde{\LL}_0)=\{u\in H^2(\R_+)\,|\,\langle1,u\rangle=0\,,\; xu\in L^2(\R_+)\}\,.
  \end{displaymath}
We begin by proving that we can assume that the eigenfunctions  of
$\tilde{\LL}_0$ do not vanish on the real line
\begin{lemma}
\label{lem:no-zeroes}
  Let $(\lambda,\psi)$ denote an eigenpair of $\tilde{\LL}_0$.
 Then, $\lambda\not\in\R $ and $\psi(0)\neq0$.
\end{lemma}
\begin{proof}
  The proof was essentially obtained in \cite{wa53}. We bring it here again for
  the convenience of the reader. Suppose that $\lambda\in\R$ and let 
  \begin{displaymath}
    \Psi(x)=\int_x^\infty \psi(y)\,dy \,.
  \end{displaymath}
Note that $\Psi(0)=0$. Obviously,
\begin{displaymath}
  \Big[-\frac{d^2}{dx^2}+(x-\lambda)\Big]\Psi^\prime=0\,.
\end{displaymath}
Multiplying by $\bar{\Psi}$ and integrating over $\R_+$ yields, for the real part,
\begin{displaymath}
  -\frac{1}{2}\|\Psi\|_2^2 - \frac{1}{2}|\Psi^\prime|^2(0)=0\,,
\end{displaymath}
which yields $\Psi\equiv0$, and hence $\lambda$ cannot be real. Since the zeroes of
Airy functions are all real we may the conclude that $\psi(0)\neq0$.
\end{proof}

 Let $\tilde{\lambda}_0\in\sigma(\tilde{\LL}_0)$. We continue  with the following auxiliary lemma:
\begin{lemma}
\label{lem:airy-bound}
Let $\tilde{\psi}_\lambda\in H^2(\R_+)$
  denote the unique  solution  (see \cite[10.4.59-10.4.63]{abst72})  of
  \begin{displaymath}
    \begin{cases}
        (\tilde{\LL}_0-\lambda)\tilde{\psi}_\lambda=0 &\mbox{for } x\in\R_+ \\
        \tilde{\psi}_\lambda(0)=1\,. &
    \end{cases}
  \end{displaymath}
Then, there exists positive $\delta$ and $C$ such that for all $|\lambda-\tilde{\lambda}_0|<\delta$
it holds that 
\begin{equation}
  \label{eq:90}
C^{-1}|\lambda-\tilde{\lambda}_0|\leq  |\langle1, \tilde{\psi}_\lambda\rangle|\leq C|\lambda-\tilde{\lambda}_0|\,.
\end{equation}
\end{lemma}
\begin{proof}
  Let $ \tilde{\psi}_{\tilde \lambda_0}$ denote the eigenfunction of $\tilde{\LL}_0$
  associated with $\tilde{\lambda}_0$, normalized by setting
  $\tilde{\psi}_{\tilde \lambda_0}(0)=1$, which is possible in view of Lemma
  \ref{lem:no-zeroes} which reads $\Ai(-\tilde \lambda_0)\neq 0\,$.  More
  generally, for $\lambda$ sufficiently close to $\tilde \lambda_0$, we may write
  in terms of Airy's functions 
  \begin{displaymath}
     \tilde{\psi}_\lambda(x)=\frac{\Ai(x-\lambda)}{\Ai(-\lambda)}= \frac{\Ai\Big((x-\tilde \lambda_0) + (\tilde \lambda_0-\lambda)\Big) }{\Ai(-\lambda)} \,,
  \end{displaymath}
and hence we may write the Taylor expansion of $ \tilde{\psi}_\lambda(x) $ at $(x-\tilde \lambda_0)$ in the form
  \begin{equation}
\label{eq:91}
\begin{array}{ll}
    \tilde{\psi}_\lambda(x)& =\frac{1}{\Ai(-\lambda)}\Big[\Ai(x-\tilde{\lambda}_0) - (\lambda-\tilde \lambda_0)\Ai^\prime(x-\tilde{\lambda}_0)\\
    & \qquad  +
     \frac{1}{2}(\lambda-\tilde \lambda_0)^2\int_0^1 (1-u) \Ai^{\prime\prime}(x-\tilde{\lambda}_0+ (\lambda-\tilde \lambda_0)u)\,du\Big] \,.
  \end{array}
  \end{equation}
 Since $\tilde{\lambda}_0\in\sigma(\tilde{\LL}_0)$ we have that
$\tilde{\psi}_{\tilde{\lambda}_0}\in D(\tilde{\LL}_0)$ and hence
\begin{displaymath}
  \langle1, \tilde{\psi}_{\tilde{\lambda}_0}\rangle=0\,.
\end{displaymath}
Furthermore, it holds that
\begin{displaymath}
  \Big\langle1,\frac{(\lambda-\tilde \lambda_0)}{\Ai(-\lambda)}\Ai^\prime(x-\tilde{\lambda}_0)\Big\rangle=
  (\lambda-\tilde \lambda_0)\frac{\Ai(-\tilde \lambda_0)}{\Ai(-\lambda)}+\OO(|\lambda-\tilde{\lambda}_0|^2) \,. 
\end{displaymath}
Then, using the fact that Airy function satisfies  $\Ai^{\prime\prime}(y) = y\Ai(y)$ and its asymptotics
as $x\to\infty$ (as recalled in \eqref{eq:asymptAiry}) we can bound the
$L^1$ norm of $ \Ai^{\prime\prime}(x-\tilde{\lambda}_0+ (\lambda-\tilde \lambda_0)u)\,$ in
$\mathbb R_+$ uniformly for $u\in [0,1]$.  Consequently, the integration of
\eqref{eq:91} over $\R_+$ yields
\begin{equation}
\label{eq:92}
  \langle1, \tilde{\psi}_\lambda\rangle= (\lambda-\tilde \lambda_0)+\OO(|\lambda-\tilde \lambda_0|^2) \,,
\end{equation} 
from which \eqref{eq:90} easily follows.
\end{proof}
In the sequel we need to use \eqref{eq:90} for
$\tilde{\psi}_\lambda(e^{i\pi/6}\cdot)$ and
hence the following corollary 
\begin{corollary}
  Under the same conditions of Lemma \ref{lem:airy-bound} it holds
  that
  \begin{equation}
    \label{eq:93}
C^{-1}|\lambda-\tilde{\lambda}_0| \leq |\langle1, \tilde{\psi}_\lambda(e^{i\pi/6}\cdot)\rangle|\leq C\,|\lambda-\tilde{\lambda}_0| \,.
  \end{equation}
\end{corollary}
The proof follows immediately by observing  that by deformation of
contour, taking advantage of the asymptotic behavior of $\tilde
\psi_\lambda(z)$ in the  sector $0\leq\arg z\leq\pi/6$ as $|z|\to\infty$, as
recalled in \eqref{eq:asymptAiry}, we have  
\begin{displaymath}
  \int_0^{+\infty} \tilde \psi_\lambda (e^{i\pi/6} x)  dx= e^{-i\frac \pi 6}  \int_0^{+\infty} \tilde \psi_\lambda (t) \, dt\,.
\end{displaymath}
We can then apply \eqref{eq:90} to obtain \eqref{eq:93}. 

\subsection{No-slip Schr\"odinger on $(-1,1)$}
We can now obtain the main result of this section which is concerned
with the no-slip Schr\"odinger operator $\LL_{\beta,U}^\zeta$ 
 defined in \cite[Subsection
6.1]{AH-arma}. We recall that  $\LL_{\beta,U}^\zeta$ 
 is  the operator associated with the differential
 operator $-d^2/dx^2 + i \beta U$ with domain
  \begin{equation}
\label{eq:94}
    D(\LL_{\beta,U}^\zeta)= \{ u\in H^2(-1,1)\,| \, \langle\zeta_\pm ,u\rangle=0\,\} \,,
  \end{equation}
where $\zeta_{-}$ and $\zeta_+$ are linearly independent, $\beta$ dependent,
functions in $H^1(-1,+1)$.
For convenience we require that $\zeta_\pm$ satisfy 
  \begin{equation}
\label{eq:95}   
    \begin{bmatrix}
      \zeta_+(1) & \zeta_-(1) \\
      \zeta_+(-1) & \zeta_-(-1)
    \end{bmatrix}
=
\begin{bmatrix}
  1 & 0 \\
  0 & 1
\end{bmatrix}
\,.
  \end{equation}  
We recall also from \cite{AH-arma}  the definition the
following Schr\"odinger operator on $\R$ (see \cite[Section
5.1]{AH-arma}) 
  \begin{equation}
    \label{eq:96}
\LL_{\beta,\tilde U}^{\R}  =-\frac{d^2}{dx^2} +i \beta \tilde{U} \,,
  \end{equation}
with domain
  \begin{equation}
    \label{eq:97}
D(\LL_{\beta,\tilde U}^{\R}) = \{ u\in H^2(\R)\,|\, xu\in L^2(\R)\,\} \,,
  \end{equation}
in which
\begin{displaymath}
  \tilde{U}(x)=
  \begin{cases}
    U(x) & x\in[-1,1] \\
    U(1)+U^\prime(1)(x-1) & x>1 \\
   x+1 & x<-1
  \end{cases}\,.
\end{displaymath}
We can now make the following claim:
\begin{lemma}
  Let  $U\in C^2([-1,1])$,  $\theta>0$ and $C_1>0\,$.
  Then, there exist positive $C$ and $\varkappa_0$, so that for any
  $0<\varkappa\leq\varkappa_0\,$, subsist $\beta_0(\varkappa)$, such that for all $\beta \geq \beta_0$, $\lambda$
  satisfying
\begin{equation}
\label{eq:98}
  \varkappa^{-1}\beta^{-1/3}\leq \beta^{1/3}|\lambda- \lambda_0(\beta) |\leq\varkappa \,,
\end{equation} 
 $(\zeta_-,\zeta_+)$ satisfying the \eqref{eq:95},
\begin{equation} 
\label{eq:99}
\|\zeta_\pm \|_\infty\leq C_1\,,
\end{equation}
and
\begin{equation} 
\label{eq:100}
\|\zeta_\pm^\prime\|_\infty\leq\theta\,,
\end{equation}
and for any pair 
  $(v,g)\in D(\LL_{\beta,U}^\zeta)\times L^2(-1,1)$ satisfying
  \begin{equation}
    \label{eq:101}
(\LL_{\beta,U}^\zeta-\beta\lambda)v=g\,,
  \end{equation}
it holds that 
\begin{subequations}
  \label{eq:102} 
  \begin{equation}
\|v\|_{L^2(-1,0)}\leq C \beta^{-1}  |\lambda-  \lambda_0(\beta)|^{-1} \|g\|_2 \,,
\end{equation}
 and
\begin{equation} 
 \|v\|_{L^2(-1,0)} \leq C \beta^{-7/6}\log\beta \, |\lambda-  \lambda_0(\beta)|^{-1} \|g\|_\infty \,,
\end{equation}
\end{subequations}
where  $\lambda_0 (\beta)$  is defined in
\eqref{eq:7abcd}.\\
Furthermore, 
it holds that
\begin{subequations}
    \label{eq:103}
\begin{equation}
  \|v\|_{L^1(-1,0)}\leq C\, \log \beta\,\beta^{-7/6} |\lambda-  \lambda_0(\beta)|^{-1}\,\|g\|_2 \,,
\end{equation}
\begin{equation}
\|v\|_{L^1(-1,0)}\leq C\, \log \beta\,\beta^{-4/3} |\lambda-  \lambda_0(\beta)|^{-1}\,\|g\|_\infty \,,
\end{equation}
\end{subequations}
  \begin{subequations}
  \label{eq:104}
  \begin{equation}
|v(-1)|\leq C\, \log \beta\,\beta^{-1}\,|\lambda- \lambda_0 (\beta)|^{-1}\,\|g\|_\infty\,,
\end{equation}
and
 \begin{equation}
|v(-1)|\leq C \, \beta^{-5/6}\,|\lambda-  \lambda_0(\beta)|^{-1}\,\|g\|_2 \,.
\end{equation}
\end{subequations}
Finally, it holds that
\begin{subequations}
  \label{eq:105}
  \begin{equation}
\|v^\prime\|_{L^2(-1,0)}\leq C \beta^{-2/3}\,  |\lambda-  \lambda_0(\beta)|^{-1}\|g\|_2 \,,
\end{equation}
and that
 \begin{equation}
 \|v^\prime\|_{L^2(-1,0)} \leq C
  \beta^{-5/6}\log\beta \, |\lambda-  \lambda_0(\beta)|^{-1}\|g\|_\infty \,.
\end{equation}
\end{subequations}

\end{lemma}
\begin{proof}
  The proof is similar to the proof of \cite[Lemmas 6.1 and
  6.2]{AH-arma} (notice nevertheless that the condition 
  \eqref{eq:100} on $\zeta^\prime_\pm$ is stronger than in \cite{AH-arma} and
  allows for stronger estimates) 
  and hence we bring here only its main ingredients. Let
\begin{subequations}
\label{eq:106} 
    \begin{equation} 
\psi_-(x)= e^{i \pi/ 6} \, \frac{{\Ai}\big(\beta^{1/3}e^{ i\pi/6}\big[(1+x)+i\lambda\big]\big)}
{A_0\big(i\beta^{1/3}\lambda\big)}\,,
\end{equation}
and 
\begin{equation}
\overline{\psi_+ (x)}=-   e^{ i \pi/6}\frac{{\Ai}\big((J_+ \beta)^{1/3}e^{
    i\pi/6}\big[(1- x)+ iJ_+ ^{-1}(\bar \lambda+ i U( 1))\big]\big)} 
{A_0\big( i \beta^{1/3}J_+^{-2/3}[\bar \lambda+i\,U(1)]\big)}\,.
\end{equation}
\end{subequations}
where $J_+=U^\prime(1)$, $\Ai$ denotes Airy function and $A_0$ is
introduced in \eqref{eq:defA0}.   Note that by \eqref{eq:98} the
denominators in \eqref{eq:106} do not vanish. Note further that 
\begin{subequations}\label{eq:107}
\begin{equation}
  \psi_-(x)=\frac{\,{\rm Ai}\big( \beta^{1/3}e^{ 2i\pi/3} \lambda\big)}{A_0\big(i\beta^{1/3}\lambda\big)}\,\check
\psi_-(x)\,,
\end{equation}
where $\check \psi_-$ is given by (\ref{eq:20}a). \\
Similarly, it holds
that 
\begin{equation}
  \overline{\psi_+}(x)=\frac{\,{\rm Ai}\big( J_+^{-2/3}\beta^{1/3}e^{ 2i\pi/3} (\bar \lambda+  i\,U(1)
)\big)}{A_0\big(i\beta^{1/3}J_+^{-2/3}[\bar \lambda+i\,U(1)]\big)} \, \overline{\check \psi_+}(x)\,,
\end{equation}
where $\check
\psi_+$ is given by (\ref{eq:20}b).
\end{subequations}\\  Furthermore, we have by
\eqref{eq:93} that
\begin{equation}
\label{eq:108}
  \big|A_0\big(i\beta^{1/3}\lambda\big)\big| \geq C^{-1}\, \beta^{1/3}|\lambda- \lambda_0(\beta))|
\end{equation}
which is the reason why \cite[Eq. (6.17)]{AH-arma} is not valid for
$\psi_-$.  (Note that while using \cite[Eq.
(6.17)]{AH-arma} we assume, for large $\beta$, $|\lambda_-| \approx \beta^{-1/3} $
and $|\lambda_+| \approx 1$.)  Since we divide by $A_0\big(i\beta^{1/3}\lambda)$ in
\eqref{eq:106} we must rephrase \cite[Eq. (6.17)]{AH-arma} to make it
useful in the present context.  Thus, by \eqref{eq:108}, we have for
all $k\in[0,4]$
\begin{equation}
  \label{eq:109}
  \|(1+x)^k\, \psi_-\|_2 \leq C\,|\lambda- \lambda_0(\beta)|^{-1} \, \beta^{-(3+2k) /6}\,.
\end{equation}
For $\psi_+$ we use \eqref{eq:24}, \eqref{eq:25}, and \cite[Eq.
(8.87)]{AH-arma} to obtain, for some $\gamma>0$,
  \begin{equation}
  \label{eq:110}
\|\psi_+\|_{L^2(-1,0)} \leq Ce^{-\gamma\beta^{1/2}}\,.
\end{equation}
For later reference we also note, by similar amendment of
\cite[Eq. (6.27)]{AH-arma},  that for all $k\in[0,3]$  
\begin{equation}
\label{eq:111}
   \|(1+x)^k\, \psi_-\|_1 \leq C \, |\lambda- \lambda_0(\beta)|^{-1} \, \beta^{-(2+k) /3}\,,
\end{equation}
(note that, for $|\lambda-\lambda_0|\approx\beta^{-1/3}$, \eqref{eq:111} coincides with
\eqref{eq:29}) and
\begin{equation}
\label{eq:112}
  \|(1-x)^k\, \psi_+\|_1 \leq C\, \beta^{-(2+3k) /6}\,,
\end{equation}
 which is precisely the same as \eqref{eq:30}.
 We now observe that by \eqref{eq:31} and \eqref{eq:108} it holds
  that
\begin{equation}
  \label{eq:113}
 \|\psi_-^\prime\|_2 \leq C\,|\lambda- \lambda_0(\beta)|^{-1} \, \beta^{-1/6}\,.
\end{equation}
 As in \eqref{eq:110} we use (\ref{eq:20}b), \eqref{eq:22}, and \cite[Eq.
(8.87)]{AH-arma} to obtain that for some $\gamma>0$
\begin{equation}
  \label{eq:114}
 \|\psi_+^\prime\|_{L^1(-1,0)} \leq Ce^{-\gamma\beta^{1/2}}\,.
\end{equation}

Next, we set 
\begin{equation}
\label{eq:115}
  g_\pm  =
  \begin{cases}
    \big(-\frac{d^2}{dx^2} +i\beta(U+i\lambda)\big)\psi_\pm    &\mbox{ for }  x\in (-1,1) \\
    0 & \text{otherwise}\,,
  \end{cases}
\end{equation}
and then introduce 
\begin{equation}
\label{eq:116}
  \tilde{v}_\pm  = \Gamma_{(-1,1)}\Big( (\LL_{\beta,\tilde U}^{\R} -\beta\lambda)^{-1}g_\pm \Big) \,,
\end{equation}
where $ \Gamma_{(-1,1)} $ denotes the restriction to $(-1,+1)$.
In the same manner  \eqref{eq:109} and \eqref{eq:110} were obtained
from \cite[Eq. (6.17)]{AH-arma},  we can obtain from  
\cite[Eq. (6.18)]{AH-arma} that
\begin{equation}
  \label{eq:117}
 \|g_-\|_2\leq  C \, \beta^{-1/6}\, |\lambda-\lambda_0(\beta) |^{-1} \,,
\end{equation}
and
\begin{equation}
  \label{eq:118}
\|g_+\|_2\leq C\, \beta^{-1/12} \,. 
\end{equation}
As in the proof of \cite[Eq. (6.19)]{AH-arma} (with
$\widetilde{\LL}_{\beta,\mathbb R}$ now denoted by $\LL_{\beta,\tilde
  U}^{\mathbb R}$), an immediate application of \cite[Eq.
(5.4)--(5.5)]{AH-arma} then yields
\begin{equation}
  \label{eq:119}
\|(U+i\lambda)\tilde{v}_-\|_2 +\beta^{-1/3}\|\tilde{v}_-\|_2+  \beta^{-2/3}\|\tilde{v}_-^\prime\|_2 \leq
C\,\beta^{-7/6}\, |\lambda-\lambda_0(\beta) |^{-1} \,,
\end{equation}
and
\begin{equation}
  \label{eq:120}
\|(U+i\lambda)\tilde{v}_+\|_2 +\beta^{-1/3}\|\tilde{v}_+\|_2+  \beta^{-2/3}\|\tilde{v}_+^\prime\|_2 \leq
C\beta^{-13/12} \,. 
\end{equation}
Finally, let $(v,g)\in D(\LL_{\beta,U}^\zeta)\times L^2(-1,1)$ satisfy 
\begin{equation}
\label{eq:121}
  (\LL_{\beta,U}-\beta\lambda)v=g\,.
\end{equation}
Define the extension of $g$ to $\R$
\begin{equation}\label{eq:122}
  \tilde{g}(x) =
  \begin{cases}
    g(x) & x\in[-1,1] \\
    0 & \text{otherwise}\,,
  \end{cases}
\end{equation}
and set 
\begin{equation}
\label{eq:123}
  u = \Gamma_{(-1,1)}\big((\LL_{\beta,\tilde U}^{\R} -\beta\lambda)^{-1}\tilde{g}\big)\,.
\end{equation}
We now look as in \cite[Eq. (6.20)]{AH-arma}  for a solution of \eqref{eq:121} in the form
\begin{equation}
  \label{eq:124}
v = A_+(\psi_+-\tilde{v}_+) + A_-(\psi_--\tilde{v}_-) + u \,.
\end{equation}
Note that \eqref{eq:124} satisfies \eqref{eq:121} for all $A_\pm\in\C$. To have $v\in D(\LL^\zeta_{\beta,U})$, we
thus need to satisfy the integral conditions $\langle v,\zeta_\pm\rangle=0$ for a
specific pair $(A_-,A_+)\in \mathbb C^2$. We thus need to solve, as in the proof of 
\cite[Lemma 6.2]{AH-arma},  the system 
\begin{equation}
\label{eq:125}
      \begin{bmatrix}
      \langle\zeta_+,(\psi_+-\tilde{v}_+)\rangle &   \langle\zeta_+,(\psi_--\tilde{v}_-)\rangle \\
       \langle\zeta_-,(\psi_+-\tilde{v}_+)\rangle &   \langle\zeta_-,(\psi_--\tilde{v}_-)\rangle
    \end{bmatrix}
\begin{bmatrix}
  A_+ (g)\\
  A_- (g)
\end{bmatrix}
=
\begin{bmatrix}
  \langle\zeta_+,u\rangle \\ 
  \langle\zeta_-,u\rangle
\end{bmatrix}\,.
\end{equation}

To obtain an estimate for the various terms in \eqref{eq:125} we use
the estimates in \cite{AH-arma} for all $\cdot_+$ terms. For
terms involving the subscript $-$ we shall need to introduce
modifications  in view of \eqref{eq:111} and \eqref{eq:119}. Thus, we have,
by \eqref{eq:112}
\begin{equation}
\label{eq:126}
   |\langle\zeta_+ -1, \psi_+ \rangle| \leq
  \|\zeta^\prime_+ \|_\infty\|[1- x]\psi_+ \|_1\leq C\beta^{-5/6} \,.
\end{equation}
Whereas by \eqref{eq:111} it holds that
\begin{equation}
\label{eq:127}
   |\langle\zeta_- -1, \psi_- \rangle| \leq
  \|\zeta^\prime_- \|_\infty\|[1- x]\psi_- \|_1\leq C\, 
  |\lambda-\lambda_0(\beta) |^{-1} \, \beta^{-1}\,.
\end{equation}
Furthermore, by \eqref{eq:120} 
\begin{equation}
\label{eq:128}
  |\langle\zeta_+,\tilde{v}_+ \rangle|  \leq 2 \,
  \|(U-\nu+i\beta^{-1/3})^{-1}\|_2\,\|(U-\nu+ i\beta^{-1/3})\tilde{v}_+ \|_2 
  \leq C\, \beta^{-11/12} \,,
\end{equation}
whereas by \eqref{eq:119} we have
\begin{equation}
\label{eq:129}
  |\langle\zeta_-,\tilde{v}_- \rangle| \leq C\,
  |\lambda-\lambda_0(\beta) |^{-1}  \beta^{-1} \,.
\end{equation}
Since \cite[Eq. (6.28)]{AH-arma} applies as well we obtain that
\begin{equation}
\label{eq:130}
 \langle1,\psi_\pm\rangle =(J_\pm \beta)^{-1/3}\big[1+\OO(\beta^{-1})\big] \,,
\end{equation}
where $J_-=1$. \\
Combining \eqref{eq:130} and (\ref{eq:126})-(\ref{eq:129})
yields
\begin{equation}
  \label{eq:131}
|\langle\zeta_- , (\psi_- -\tilde{v}_-)\rangle-\beta^{-1/3}| \leq
C\, \beta^{-1}\,|\lambda-\lambda_0(\beta) |^{-1} \,, 
\end{equation}
and
\begin{equation}
  \label{eq:132}
|\langle\zeta_+ , (\psi_+ -\tilde{v}_+)\rangle-\beta^{-1/3}|\leq C\beta^{-5/6} \,.
\end{equation}
By \eqref{eq:100} and \eqref{eq:111} it holds that
\begin{equation}
\label{eq:133}
   |\langle\zeta_+ , \psi_-\rangle| \leq C\, 
   \beta^{-1}\,|\lambda-\lambda_0(\beta) |^{-1}\,.
\end{equation}
Similarly, by \eqref{eq:100} and \eqref{eq:112} we have
\begin{equation}
\label{eq:134}
   |\langle\zeta_- , \psi_+\rangle| \leq  C\, \beta^{-5/6}\,.
\end{equation}
Finally, as in \eqref{eq:128} and \eqref{eq:129} we obtain that
\begin{equation}
  \label{eq:135}
 |\langle\zeta_-,\tilde{v}_+ \rangle|  \leq C\, \beta^{-11/12} \,,
\end{equation}
and
\begin{equation}
  \label{eq:136}
 |\langle\zeta_+,\tilde{v}_- \rangle|  \leq C\, \beta^{-1}\,|\lambda-\lambda_0(\beta) |^{-1} \,.
\end{equation}
Substituting (\ref{eq:126})-(\ref{eq:136}) into \eqref{eq:125} yields,
in view of \eqref{eq:98},  
\begin{displaymath}
      \begin{bmatrix}
  (J_+\beta)^{-1/3}\big[1+\OO(\beta^{-1/2})\big]      & \OO(\varkappa\beta^{-1/3})  \\
    \OO(\beta^{-5/6})   &    \beta^{-1/3}\big[1+\OO(\varkappa)\big]
    \end{bmatrix}
\begin{bmatrix}
  A_+ (g)\\
  A_- (g)
\end{bmatrix}
=
\begin{bmatrix}
  \langle\zeta_+,u\rangle \\ 
  \langle\zeta_-,u\rangle
\end{bmatrix}\,.
\end{displaymath}
From the above system we obtain, for sufficiently small $\varkappa$, that
there exists $\beta_0(\varkappa)$  such that for all $\beta\geq \beta_0(\varkappa)$ it holds that
\begin{equation}
\label{eq:137}
  |A_\pm(g)|\leq C\,\big(\beta^{1/3} |\langle\zeta_\pm,u\rangle|+ \varkappa |\langle\zeta_\mp,u\rangle| \big)\,,
\end{equation}
yielding, by \cite[Eq. (5.36)]{AH-arma}, 
\begin{equation}
\label{eq:138}
   |A_\pm(g)|\leq C\beta^{-1/2}\|g\|_2 \,.
\end{equation}
Together with \eqref{eq:109}, \eqref{eq:110},  \eqref{eq:119},
  \eqref{eq:120}, \eqref{eq:123}, \eqref{eq:124},  \eqref{eq:98},
and \cite[Eq. (5.4)]{AH-arma} the above inequality yields
(\ref{eq:102}a). To establish (\ref{eq:105}a) we need in addition
\eqref{eq:113} and \eqref{eq:114}.   Next, we conclude
  (\ref{eq:103}a) from \eqref{eq:124}, \eqref{eq:111}, \eqref{eq:114},
  \eqref{eq:119}, \eqref{eq:120} and \cite[Lemma 5.6]{AH-arma} (applied
  to \eqref{eq:123}). Note that to apply \cite[Lemma 5.6]{AH-arma}, we
  must have $\Re\lambda\lesssim\beta^{-1/3}$, which is clearly guaranteed by
  \eqref{eq:98}.

To prove (\ref{eq:103}b) we use again \cite[Lemma 5.6]{AH-arma}
  to obtain from \eqref{eq:137} that
\begin{equation}
\label{eq:139}
   |A_\pm(g)|\leq C\beta^{-2/3}\log\beta \, \|g\|_\infty  \,.
\end{equation}
We can now conclude (\ref{eq:103}b) in the same manner we have obtained
(\ref{eq:103}a). 
Finally, by \eqref{eq:108} together with (\ref{eq:106}a) it holds that
\begin{equation}\label{eq:140}
  |\psi_-(-1)|  \leq  C\, \beta^{-1/3}|\lambda-\lambda_0(\beta) |^{-1}\,,
\end{equation}
and by  \eqref{eq:25}, (\ref{eq:107}b), and  \cite[Eq.
(8.87)]{AH-arma} it holds that
\begin{equation*}
  |\psi_+(-1)|  \leq  C\, \beta^{ -1/12}\exp - \beta^{1/2}\,,
\end{equation*}
which clearly implies, using \eqref{eq:98}, that
\begin{equation}
\label{eq:141}
  |\psi_+(-1)|  \leq  \widehat C\, \beta^{-1/3}|\lambda-\lambda_0(\beta) |^{-1}\,.
\end{equation}
 By \eqref{eq:124} it holds that 
\begin{multline}
\label{eq:142}
  |v(-1)|\leq A_-(g)\big[|\psi_-(-1)|+|\tilde{v}_-(-1)|\big]+\\ A_+(g)\big[|\psi_+(-1)|+|\tilde{v}_+(-1)|\big ]+|u(-1)| \,.
\end{multline}

 We now estimate some of the  terms on the right-hand-side of
\eqref{eq:142}.\\
 Let $\tilde{u}:=\big(\LL_{\beta,\tilde U}^{\R}
-\beta\lambda)^{-1}\tilde{g}$. 
We begin by observing that by  \cite[Proposition 5.1]{AH-arma} and \eqref{eq:122}, it holds that
\begin{subequations}
\label{eq:143}
\begin{equation}
  \|\tilde{u}\|_{L^2(\R)} \leq C \beta^{-2/3}\|g\|_2 \,,
\end{equation}
and that, by \cite[Lemma 5.5]{AH-arma}, for any $a>0\,$, 
\begin{equation}
  \|\tilde{u}\|_{L^2(-a,+a)} \leq C_a  \beta^{-5/6}\|g\|_\infty\,.
\end{equation}
Combining the above with \eqref{eq:139}, \eqref{eq:124},
\eqref{eq:119}, \eqref{eq:120}, \eqref{eq:109}, and \eqref{eq:110} yields
(\ref{eq:102}b).

By \cite[Proposition 5.1]{AH-arma}, it holds that
\begin{displaymath}
   \|[\tilde U-\Im\lambda]\tilde{u}\|_{L^2(\R)} \leq C  \beta^{-1}\|\tilde g\|_2\,,
\end{displaymath}
and hence, for any $a>2$ it holds that
\begin{displaymath}
     \|\tilde{u}\|_{L^2(-\infty,-a)}+  \|\tilde{u}\|_{L^2(a, +\infty)} \leq \check  C
     \beta^{-1}\|g\|_2 \,.
\end{displaymath}
 Combining the above with (\ref{eq:143}b)   yields,   as $\|g\|_2 \leq \sqrt{2}\, \|g\|_\infty\,$,
\begin{equation}
  \|\tilde{u}\|_{L^2(\R)} \leq C  \beta^{-5/6}\|g\|_\infty\,.
\end{equation}
\end{subequations}
Furthermore, since
\begin{displaymath}
  \|\tilde{u}^\prime\|_{L^2(\R)}^2 =  \Re \langle\tilde{u},\tilde{g}\rangle + \beta\, \Re\lambda  \|\tilde{u}\|_{L^2(\R)}^2\,,
  \end{displaymath}
we obtain, in view of \eqref{eq:98}, 
\begin{displaymath}
  \|\tilde{u}^\prime\|_{L^2(\R)}^2 \leq \|\tilde{u}\|_{L^1(-1,+1)}\|g\|_\infty + C\beta^{2/3}\,\|\tilde{u}\|_{L^2(\R)}^2 \,.
\end{displaymath}
Hence, by (\ref{eq:143}c) and \cite[Lemma 5.6]{AH-arma} 
\begin{equation}
\label{eq:144}
    \|\tilde{u}^\prime\|_{L^2(\R)}^2 \leq C\beta^{-1/2}\log^{1/2}\beta \, \|g\|_\infty
\end{equation}
Recalling that $u$ is the restriction of $\tilde u$ to $(-1,+1)$
 we may now combine \eqref{eq:144} with \eqref{eq:98}, \eqref{eq:139},
  \eqref{eq:119}, \eqref{eq:120}, \eqref{eq:113}, and \eqref{eq:114} to
  establish (\ref{eq:105}b).\\
By  \cite[Proposition 5.1]{AH-arma} it holds that
\begin{displaymath}
   \|\tilde{u}^\prime\|_{L^2(\R)}\leq C \beta^{-1/3}\|g\|_2\,,
\end{displaymath}
Combining the above with \eqref{eq:144} yields 
\begin{displaymath}
   \|\tilde{u}^\prime\|_{L^2(\R)}\leq C \min \big(\beta^{-1/2} \, \log^{1/2}\beta
   \,\|g\|_\infty,\beta^{-1/3}\|g\|_2\big)\,.
\end{displaymath}
Hence, by \eqref{eq:123}, the above inequality, (\ref{eq:143}a), and
(\ref{eq:143}c) 
\begin{equation}
\label{eq:145}
  |u(-1)|\leq \sqrt{2}\,
  \|\tilde{u}\|_{L^2(\R)}^{1/2}\,\|\tilde{u}\|_{H^1(\R)}^{1/2} \leq
  C\min(\beta^{-1/2}\|g\|_2\,,\,\beta^{-2/3}\log\beta\|g\|_\infty)\,, 
\end{equation}
Finally,  in a similar manner to \eqref{eq:145} we
obtain from \eqref{eq:116} and \eqref{eq:117} that  
\begin{displaymath}
  |\tilde{v}_-(-1)|\leq C\beta^{-1/2}\|g_-\|_2\leq C\,\beta^{-2/3}\, |\lambda-\lambda_0(\beta) |^{-1} \,,
\end{displaymath}
and by  \eqref{eq:118}, 
\begin{displaymath}
  |\tilde{v}_+(-1)|\leq C\beta^{-1/2}\|g_+\|_2\leq C\beta^{-7/12} \,.
\end{displaymath}
 We can finally obtain \eqref{eq:104} from the above, \eqref{eq:145},
\eqref{eq:138}, and \eqref{eq:139}.\\
\end{proof}

\section{Existence of eigenvalues.}
\label{sec:4}
 In the following we establish the proof of Theorem \ref{thm:main}
stating the existence of a critical value for
$\B_{\lambda,\alpha,\beta}^\D$ near $\lambda_0(\beta)$ which is defined in \eqref{eq:7abcd}. We
note that this critical value is an eigenvalue of the linearized
Navier-Stokes operator (see \S1 and 
\cite{AH-arma,drre04}).

\subsection{Boundary layer estimates}
\label{sec:4.1}
 For $\alpha \geq0$, $\beta >0$ and  $\lambda \in \mathbb C$, we consider (see \eqref{eq:6}) 
$(\phi,f)\in D(\B^\D_{\lambda,\alpha,\beta})\times L^\infty(-1,1)$ satisfying
\begin{subequations}
\label{eq:146}
\begin{equation}
  \B_{\lambda,\alpha,\beta}\,\phi = f\,,
\end{equation}
and associate with this pair
\begin{equation}
  v:=-\phi^{\prime\prime}+\alpha^2\phi \,,
\end{equation}
\begin{equation}
\tilde{\phi}(x):=   \phi(x)+ v(-1)\,\alpha^{-1}\int_{-1}^x\sinh\alpha(x-\xi) \,\check{\psi}_- (\xi,\lambda,\beta)\,d\xi \,,
\end{equation}
 and 
\begin{equation}
  \tilde{v}:=-\tilde{\phi}^{\prime\prime}+\alpha^2\tilde{\phi} \,,
\end{equation}
 where $\check{\psi}_-$ is given by (\ref{eq:20}a)\,.\\
\end{subequations}
Note  that
\begin{subequations}
\label{eq:147}
\begin{equation}
  v=\tilde{v}+v(-1)\check{\psi}_-(\cdot,\lambda,\beta) \,.
\end{equation}
 and
\begin{equation}
\tilde v(-1) =0\,.
\end{equation}
\end{subequations}

Set, for some $s \in (0,1/3)$,
\begin{equation}
  \label{eq:148}
\delta(\beta) =\beta^{-s} 
\end{equation}
and recall that for $\delta>0$,  $\eta_\delta$  is given by \eqref{eq:50} and
\eqref{eq:76}. 

We begin by  obtaining an estimate  
of $\tilde{v}$ , associated with some suitable $(\alpha,\beta,\lambda, f,\phi)$,
near  $x=-1$. 
\begin{lemma}
 Let $\alpha_0>0$. Then,  for any $k\in\N$
there exist positive  positive $C$, 
$\varkappa_0$ and  $\beta_0$, such that for all $\beta\geq
\beta_0\,$, $\varkappa_0/2<\varkappa<\varkappa_0\,$, $0\leq\alpha \leq \alpha_0\,,$
   $\lambda$ satisfying (cf \eqref{eq:98})
 \begin{displaymath}
(K)_{\varkappa,\beta}  \qquad   \varkappa^{-1}\beta^{-1/3}\leq \beta^{1/3}|\lambda- \lambda_0(\beta) |\leq\varkappa \,,
 \end{displaymath}
$(\phi,f,v,\tilde v,\alpha,\lambda,\beta)$ satisfying \eqref{eq:146}, and
 $\delta(\beta)$ given by \eqref{eq:148}, it holds that 
\begin{equation}
    \label{eq:150}
   I_0(\beta) \leq C  \,\big[\varepsilon(\beta) ^kI_k(\beta) +\delta(\beta)\,|\lambda-\lambda_0(\beta) |\log\beta|\,|v(-1)|+\beta^{-5/6}\|f\|_2\big] \,,
  \end{equation}
where
\begin{equation}
\label{eq:151}
  I_k(\beta)= \|\eta_{2^k\delta(\beta)}\,\tilde{v} \|_1 + \beta^{-1/6} \| \eta_{2^k\delta(\beta)}\,\tilde{v}\|_2 +
  \beta^{-1/2}\|(\eta_{2^k\delta(\beta)}\,\tilde{v})^\prime\|_2\,.
\end{equation}
and
\begin{equation}
\label{eq:149}
\varepsilon(\beta)=\delta(\beta)\log\beta+\delta(\beta)^{-1}\beta^{-1/3}\,.
\end{equation}
\end{lemma}

\begin{proof}
Using \eqref{eq:6} and \eqref{eq:146}, the equation for $\tilde{\phi}$ assumes the form
\begin{equation}
\label{eq:152}
  \Big(-\frac{d^2}{dx^2}+i\beta(U+i\lambda)\Big)\tilde{v}+i\beta U^{\prime\prime}\tilde{\phi}=
  f+v(-1)\hat{g}_--\tilde{g}\,, 
\end{equation}
where
\begin{equation}
\label{eq:153}
  \tilde{g}(x)= i\beta\, v(-1)\, U^{\prime\prime}(x)\,\alpha^{-1} \int_{-1}^x\sinh\alpha(x-\xi)\,\check{\psi}_-(\xi, \lambda,\beta)\,d\xi\,,
\end{equation}
and
\begin{equation}\label{eq:154}
 \hat{g}_-(\cdot,\lambda) =\Big(-\frac{d^2}{dx^2} +i\beta(U+i\lambda)\Big)\,\check{\psi}_-(\cdot,\lambda,\beta) \,.
\end{equation}

Clearly, for any $x\in(-1,-1+2\delta)$, we have
\begin{equation}
\label{eq:155}
  |\tilde{\phi}(x)|=\Big|\int_{-1}^x (x-\xi)\tilde{\phi}^{\prime\prime}(\xi)\,d\xi\Big|\leq
2\delta \, \big(\|\tilde{v}\|_{L^1(-1,x)}+\alpha_0^2 \,\|\tilde{\phi}\|_{L^1(-1,x)}\big)\,. 
\end{equation}
Here $\beta$ is chosen large enough so that $\delta=\delta(\beta) <1/2$.
Hence, integrating the above inequality 
over $[-1,x]$ yields,  for sufficiently large $\beta_0$, 
\begin{equation}\label{eq:156}
  \|\tilde{\phi}\|_{L^1(-1,x)}\leq C\,\delta^2\, \|\tilde{v}\|_{L^1(-1,x)}\,.
\end{equation}
Substituting \eqref{eq:156}  into \eqref{eq:155} leads to
\begin{equation}\label{eq:157}
   |\tilde{\phi}(x)|\leq C\,\delta\,\|\tilde{v}\|_{L^1(-1,x)}\,,\, \forall x \in (-1,-1+2 \delta)\,.
\end{equation}
Similarly, we write
\begin{equation}\label{eq:158}
  |\tilde{\phi}^\prime(x)|=\Big|\int_{-1}^x\tilde{\phi}^{\prime\prime}(\xi)\,d\xi\Big|\leq
  \|\tilde{v}\|_{L^1(-1,x)}+C\,\|\tilde{\phi}\|_{L^1(-1,x)} \,, 
\end{equation}
which together with \eqref{eq:157} yields
\begin{equation}
  \label{eq:159}
 |\tilde{\phi}^\prime(x)|\leq \check C\, \|\tilde{v}\|_{L^1(-1,x)}\,,\, \forall x \in (-1,-1+2 \delta)\,.
\end{equation}
By \eqref{eq:157} it holds, in view of \eqref{eq:76}
 that
\begin{equation}
  \label{eq:160}
\| \eta_\delta \,\tilde{\phi} \|_\infty \leq C\, \delta\, \|\eta_{2\delta}\, \tilde v\|_1\,.
\end{equation}
~\\
From Lemma \ref{lem:no-zeroes}, $(K)_{\varkappa,\beta}$ (or \eqref{eq:98}), and \eqref{eq:74} we
may conclude that
\begin{equation}
\label{eq:161}
\|( \LL^{\D,(-1,+1)}_{\beta,U} -\beta\lambda)^{-1}\| +\beta^{-1/3}\|\frac{d}{dx}\,( \LL^{\D,(-1,+1)}_{\beta ,U}-\beta\lambda)^{-1}\| \leq C\beta^{-2/3}\,.
\end{equation}
Similarly, we obtain  by $(K)_{\varkappa,\beta}$ and \eqref{eq:75} that
for any $g\in L^\infty(-1,1)$
\begin{equation}
\label{eq:162}
  \|( \LL^{\D,(-1,+1)}_{\beta,U} -\beta\lambda)^{-1} g\|_1 \leq C\, \min \Big(\frac{\log\beta}{\beta}\|g\|_\infty,\beta^{-5/6}\|g\|_2\Big)\,.
\end{equation}

Next, we estimate $\|\eta_{2\delta} \tilde{v}\|_1$ using \eqref{eq:162}. To this
end we first write, in view of (\ref{eq:152}),
\begin{equation}
\label{eq:163}
   ( \LL_{\beta,U} -\beta\lambda)(\eta_\delta\,\tilde{v})=-i\beta
   U^{\prime\prime}\eta_\delta\,\tilde{\phi}+\eta_\delta\,\big[-\tilde{g} +v(-1)\,\hat{g}_-+f\big] +2\,\eta_\delta^\prime\,\tilde{v}^\prime+\eta_\delta^{\prime\prime}\, \tilde{v}\,.
\end{equation}
Then, by \ \eqref{eq:147},  $\eta_\delta \, \tilde{v}$ belongs to $
D\big(\LL^{\D,(-1,+1)}_{\beta,U}\big)$ and   it holds  by \eqref{eq:162}
(using both the $L^2$ and the $L^\infty$ estimate)) that 
\begin{multline*}
  \| \eta_\delta\,\tilde{v}\|_1\leq C\,\Big(\log\beta\,\|\eta_\delta\,\tilde{\phi}\|_\infty
  +\beta^{-5/6}\big[\|f\|_2+|v(-1)|\,\|\hat{g}\|_2\big]\\ + \beta^{-1}\log\beta \,
  \|\eta_\delta\,\tilde{g}\|_\infty +\beta^{-5/6}\big[\|\eta_\delta^\prime\,\tilde{v}^\prime\|_2+\|\eta_\delta^{\prime\prime}\,\tilde{v}\|_2\big]\Big)\,.
\end{multline*}
For the last two terms on the right-hand-side we have, 
\begin{displaymath}
 \|\eta_\delta^\prime\tilde{v}^\prime\|_2+\|\eta_\delta^{\prime\prime}\tilde{v}\|_2 \leq
 C\delta^{-1} \,\Big[\|(\eta_{2\delta}\,\tilde{v})^\prime\|_2+\delta^{-1}\|\eta_{2\delta}\,\tilde{v}\|_2\Big] \,.
\end{displaymath}
By \eqref{eq:160} we then obtain that
\begin{equation}
\label{eq:164}
\begin{array}{l}
   \| \eta_\delta \,\tilde{v}\|_1\leq C\,\Big(\delta\log\beta\,\|\eta_{2\delta}\,\tilde{v}\|_1 
   +\beta^{-5/6}\big[\|f\|_2+|v(-1)|\,\|\hat{g}_-\|_2\big] \\ \qquad \qquad  \qquad +\beta^{-1}\log\beta \,
  \|\eta_\delta\,\tilde{g}\|_\infty+\delta^{-1}\beta^{-5/6}\big[\|(\eta_{2\delta}\,\tilde{v})^\prime\|_2+\delta^{-1}\|\eta_{2\delta}\,\tilde{v}\|_2\big]\Big)\,.
\end{array}
\end{equation}
In view of \eqref{eq:160},  \eqref{eq:161}, \eqref{eq:162},  and
\eqref{eq:163} it holds that 
\begin{multline}
\label{eq:165}
  \|\eta_\delta\,\tilde{v}\|_2\leq C\beta^{-2/3}\Big[\beta^{-1/6}\|\eta_\delta\,\tilde{g}\|_\infty+\|f\|_2+|v(-1)|\,\|\hat{g}_-\|_2+\beta^{5/6}\delta\,\|\eta_{2\delta}\,v\|_1\\ +
  \delta^{-1}\,\|(\eta_{2\delta}\,\tilde{v})^\prime\|_2+\delta^{-2}\,\|\eta_{2\delta}\,\tilde{v}\|_2\Big] \,.  
\end{multline}
Furthermore, it holds that
\begin{multline}\label{eq:166}
  \|(\eta_\delta\,\tilde{v})^\prime\|_2\leq
  C\beta^{-1/3}\Big [\beta^{-1/6}\log\beta\|\eta_\delta\,\tilde{g}\|_\infty+\|f\|_2+|v(-1)|\,\|\hat{g}_-\|_2 \\ \quad  +\beta^{5/6}\delta\,\log\beta\,\|\eta_{2\delta}\,\tilde v\|_1+
  \delta^{-1}\|(\eta_{2\delta}\,\tilde{v})^\prime\|_2+\delta^{-2}\|\eta_{2\delta}\,\tilde{v}\|_2\Big] \,.  
  \end{multline}
Combining \eqref{eq:166} with \eqref{eq:164} and \eqref{eq:165} yields 
\begin{multline}
\label{eq:167}
  \| \eta_\delta \,\tilde v\|_1 + \beta^{-1/6}\| \eta_\delta\,\tilde v \|_2 +
  \beta^{-1/2}\|(\eta_\delta\,\tilde{v})^\prime\|_2\leq \\
 \leq  C\Big[\delta\log\beta\,\|\eta_{2\delta}\,\tilde{v}\|_1+\beta^{-5/6}\big[\delta^{-1}\|(\eta_{2\delta}\,\tilde{v})^\prime\|_2+\delta^{-2}\|\eta_{2\delta}\,\tilde{v}\|_2\big]\\
  +\beta^{-1}\log\beta\, \|\eta_\delta\,\tilde{g}\|_\infty+\beta^{-5/6}\big[\|f\|_2+v(-1)\|\hat{g}_-\|_2\big]\Big]\,.
\end{multline}
Let $I_k$ be given by \eqref{eq:151}.
From \eqref{eq:167} we can now conclude that for  $\beta_0$ large enough and $\beta \geq \beta_0$, we have with $\delta=\delta(\beta)$, 
\begin{displaymath}
  I_0(\beta) \leq C\,\Big[(\delta  \log\beta+\delta^{-1}\beta^{-1/3})I_1 (\beta)+\beta^{-1}\log\beta\|\eta_\delta\,\tilde{g}\|_\infty+\beta^{-5/6}(\|f\|_2+v(-1)\|\hat{g}_-\|_2)\Big] \,. 
\end{displaymath}
Note that since $s <\frac 13$ in \eqref{eq:148} it holds that 
\begin{displaymath}
\delta(\beta)^{-1}  \beta^{-1/3}=  \beta^{s-\frac 13} \leq \beta_0^{s-\frac 13}\,,
\end{displaymath}
and hence we can choose $\beta_0$ such that $\epsilon(\beta) < 1$ for all $\beta \geq \beta_0$.
Moreover, for any $k\in\N$  we can, possibly by
increasing $\beta_0$ with $k$ (so that $2^k \delta(\beta) < 1/2$), obtain for all
$\beta \geq \beta_0$ that
\begin{displaymath}
   I_{k-1}(\beta)\leq C \,\Big[(\delta\log\beta+\delta^{-1}\beta^{-1/3} )I_k(\beta)  +
   \beta^{-1}\log\beta\,
   \|\eta_{2^k\delta}\tilde{g}\|_\infty+\beta^{-5/6}(\|f\|_2+v(-1)\|\hat{g}_-\|_2)\Big]\,.  
\end{displaymath}
Applying the above inequality  $k$  times recursively yields, with $\delta=\delta(\beta)$
given by \eqref{eq:148},
\begin{multline}
\label{eq:168}
   I_0(\beta)\leq C\Big[(\delta\log\beta+\delta^{-1}\beta^{-1/3})^k\,I_k(\beta) + \\
   \beta^{-1}\log\beta\,\|\eta_{2^k\delta}\tilde{g}\|_\infty+\beta^{-5/6}(\|f\|_2+v(-1)\|\hat{g}_-\|_2)\Big]\,.  
 \end{multline}

{\bf Estimate of $\eta_{2^k\delta}\,\tilde{g}$ in $L^\infty$.
   \\} 

From \eqref{eq:153} it follows that for sufficiently large $\beta$ (hence
imposing sufficiently small $\delta(\beta)$) 
\begin{multline}
\label{eq:169}
  \Big\|\eta_{2^k\delta}\frac{U^{\prime\prime}}{\alpha}\int_{-1}^x\sinh\alpha(x-\xi)\check{\psi}_-(\xi,\lambda)\,d\xi\Big\|_\infty \\
 \leq C\,\Big\|\eta_{2^k\delta}\int_{-1}^x[(1+x)-(1+\xi)]\,\check{\psi}_-(\xi,\lambda)\,d\xi\Big\|_\infty  \,.  
\end{multline}
To estimate the right-hand-side of \eqref{eq:169}  we first
observe that  there exists $C>0$ such that, for any $x\in(-1,-1+2\delta )$, 
 and $\lambda$ satisfying $(K)_{\varkappa,\beta}$ 
it holds by  \eqref{eq:29} that
\begin{equation}
\label{eq:170}
  \int_{-1}^x(1+\xi)|\check{\psi}_-(\xi,\lambda)|\,d\xi\leq
  \int_{-1}^{ \infty}(1+\xi)|\check{\psi}_-(\xi,  \lambda )|\,d\xi\leq C\beta^{-2/3} \,.
\end{equation}
To estimate the term involving $(1+x)\check{\psi}_-(\xi,\lambda)$   we first write
\begin{displaymath}
  \Big|\int_{-1}^x \check{\psi}_-(\xi,\lambda)\,d\xi\Big|\leq\Big|\int_x^\infty
  \check{\psi}_-(\xi,\lambda)\,d\xi\Big|+\Big|\int_{-1}^\infty \check{\psi}_-(\xi,\lambda)\,d\xi\Big|\,.
\end{displaymath}
For the first term on the right-hand-side we have by  \eqref{eq:28}
(see also \eqref{eq:44}),  that for some $\gamma>0$
\begin{equation}
\label{eq:171}
  \Big|\int_x^\infty  \check{\psi}_-(\xi,\lambda)\,d\xi\Big|\leq C\,  \beta^{-1/3} e^{-\gamma\beta^{1/3}(1+x)}\,,
\end{equation}
and hence, for all $x\in(-1,-1+2\delta)$,
\begin{equation}
\label{eq:172}
   \Big|(1+x)\int_x^\infty \check{\psi}_-(\xi,\lambda)\,d\xi\Big|\leq
   C\, \beta^{-1/3} \,(1+x)e^{-\gamma(1+x)\beta^{1/3}} \leq \check C\,\beta^{-2/3}\,.
\end{equation}
For the second term we use \eqref{eq:93} and dilation to obtain that 
\begin{equation}
\label{eq:173}
  \Big|\int_{-1}^\infty \check{\psi}_-(\xi,\lambda)\,d\xi\Big|\leq C\, |\lambda-\lambda_0(\beta) |\,.
\end{equation}
Consequently, for all $x\in(-1,-1+2\delta)$, 
\begin{displaymath}
   \Big|(1+x)\,\eta_{2^k\delta}\int_{-1}^\infty \check{\psi}_-(\xi,\lambda)\,d\xi\Big|\leq C\, \delta\,
   |\lambda-\lambda_0(\beta) |  \,.
\end{displaymath}
Substituting the above together with \eqref{eq:170} and \eqref{eq:172}
into \eqref{eq:169} yields, in view of $(K)_{\varkappa,\beta}$,
\begin{equation}
  \label{eq:174}
 \|\eta_{2^k\delta}\tilde{g}\|_\infty\leq C\, \delta\,\beta\,|\lambda-\lambda_0(\beta) |\, |v(-1)| \,. 
\end{equation}
Recall that $\hat g_-$ is  defined by
\eqref{eq:154}. 
By  \cite[Eq. (8.96)]{AH-arma} it holds that
\begin{displaymath} 
  \|\hat{g}_-\|_2\leq C\beta^{1/6}\,.
\end{displaymath}
Substituting the above together with \eqref{eq:174} into
\eqref{eq:168},
leads,  with the aid of \eqref{eq:148}, to 
\begin{displaymath}
   I_0\leq C \,\big[(\delta\log\beta+\delta^{-1}\beta^{-1/3})^kI_k +
  \delta\,\log\beta\, |\lambda-\lambda_0(\beta) |\, |v(-1)|+\beta^{-5/6}\|f\|_2\big] \,,
\end{displaymath}
which is precisely \eqref{eq:150}.
\end{proof}

\begin{lemma}  Let $\alpha_0 >0$. 
For any $k\in\N$ there exist positive $C$, $\varkappa_0$ and $\beta_0$ such that for
  all $\beta\geq \beta_0$, $\varkappa_0 /2 \leq \varkappa \leq \varkappa_0$, $0\leq \alpha \leq \alpha_0$,  $(f,\phi)\in L^2(-1,1)\times D(\B_{\lambda,\alpha,\beta})$ satisfying 
  \eqref{eq:146}, $v=-\phi^{\prime\prime}+ \alpha^2 \phi$, and $\lambda$ satisfying $(K)_{\varkappa,\beta}$, it holds that
  \begin{subequations}
\label{eq:175}
      \begin{multline}
    |\phi(-1+\delta(\beta))|\leq C \,\Big[\delta(\beta) \,\varepsilon(\beta)^kI_k(\beta)\\ \qquad +
    \big(\delta(\beta)^2\log\beta\,|\lambda-\lambda_0(\beta) | +\beta^{-2/3}\big) \, |v(-1)|\\ +\delta(\beta)\beta^{-5/6}\|f\|_2\Big]\,,
  \end{multline}
and
  \begin{equation}
    |\phi^\prime(-1+\delta(\beta))|\leq C\,\Big[\varepsilon(\beta)^kI_k(\beta)+
    |\lambda-\lambda_0(\beta) | \, |v(-1)|\,+\beta^{-5/6}\|f\|_2\Big] \,,
  \end{equation}
  \end{subequations}
  where $\delta(\beta)$ and $\varepsilon(\beta)$ are introduced in \eqref{eq:148} and
  \eqref{eq:149} and $I_k$ is given by \eqref{eq:151}.
\end{lemma}

\begin{proof}
  As in \eqref{eq:155} we write that for any $x\in(-1,-1+\delta)$
  \begin{displaymath}
     |\phi(x)|\leq \Big|\int_{-1}^x (x-\xi)\phi^{\prime\prime}(\xi)\,d\xi\Big|\leq
     2\,\|(1+\xi)v\|_{L^1(-1,x)}+C\delta\|\phi\|_{L^1(-1,x)}\,, 
  \end{displaymath}
from which we conclude, as in the previous lemma,  that for $\beta$ large
enough, $\delta$ given by \eqref{eq:148},   and for any $x\in(-1,-1+\delta]$
\begin{equation}
\label{eq:176}
   \|\phi\|_{L^1(-1,x)}\leq C\, \delta\, \|(1+\cdot)v\|_{L^1(-1,x)}\,.
\end{equation}
Combining the above pair of inequalities yields   and for any $x\in(-1,-1+\delta]$
\begin{equation}
  \label{eq:177}
  |\phi(x)|\leq C\, \|(1+\cdot)v\|_{L^1(-1,x)}\,.
\end{equation}
Similarly, using (\ref{eq:146}b) we integrate over $(-1,-1+\delta)$ to
obtain with the aid of \eqref{eq:176} 
\begin{equation}
\label{eq:178}
  \Big|\phi^\prime(-1+\delta)+  \int_{-1}^{-1+\delta}v\,dx\Big|\leq \alpha_0^2 \, \|\phi\|_{L^1(-1,-1+\delta)}\leq  C\,\delta\, \|(1+\cdot )v\|_{L^1(-1,-1+\delta)}\,,
\end{equation}
With the aid of \eqref{eq:171} and \eqref{eq:173} we obtain that
\begin{displaymath}
  \Big|\int_{-1}^{-1+\delta}\check{\psi}_-(x,\lambda)\,dx\Big|\leq C\,|\lambda-\lambda_0(\beta)|\,.
\end{displaymath}
To obtain the above inequality we used the facts that $\delta(\beta)=\beta^{-s}$
for some $0<s < 1/3\,$,  that $\beta^{-1/3} e^{-\gamma \beta^{1/3-s}}\leq C \beta^{-2/3}$
for $\beta$ large enough,  and that $\lambda$ satisfies $(K)_{\varkappa,\beta}$. 
We can now use \eqref{eq:147} and \eqref{eq:178} to obtain that
\begin{displaymath}
   |\phi^\prime(-1+\delta)|\leq C\,\Big[\|\tilde{v}\|_{L^1(-1,-1+\delta)}
   +\big( |\lambda-\lambda_0(\beta)|\,+\delta\|(1+x)\check{\psi}_-\|_{L^1(-1,-1+\delta)}\big)\,|v(-1)| \Big]\,. 
\end{displaymath}
By  \eqref{eq:29} (with $s=1$)  and $(K)_{\varkappa,\beta}$  we then obtain that
\begin{displaymath}
  |\phi^\prime(-1+\delta)|\leq C\,\big[\|\tilde{v}\|_{L^1(-1,-1+\delta)}+ |\lambda-\lambda_0(\beta)|\,|v(-1)|\big]\leq C\,\big[I_0 + |\lambda-\lambda_0(\beta)|\,|v(-1)|\big]  \,.
\end{displaymath}
By \eqref{eq:150} we then obtain that for every $k\in\N$ there exists
$C>0$ such that 
\begin{equation}
\label{eq:179}
    |\phi^\prime(-1+\delta)|\leq C \, \big[\epsilon^kI_k+  |\lambda-\lambda_0(\beta)|\,|v(-1)| +\beta^{-5/6}\|f\|_2\big] \,. 
\end{equation}
To complete the proof we combine \eqref{eq:177} and \eqref{eq:147} to
obtain, with the aid of \eqref{eq:29}  and for any $x\in(-1,-1+\delta)$
\begin{equation}
\label{eq:180}
  |\phi(x)|\leq C\,\big[\delta\,\|\tilde{v}\|_{L^1(-1,x)}+\beta^{-2/3}|v(1)|\big] \,.
\end{equation}
Using \eqref{eq:150} we then obtain that for every $k\in\N$ there exists
$C>0$ such that
\begin{equation}\label{eq:181}
   |\phi(-1+\delta)|\leq C\, \big[\delta\epsilon^kI_k+(\delta^2\log\beta|\lambda-\lambda_0(\beta) |+\beta^{-2/3})|v(-1)|+\delta\beta^{-5/6}\|f\|_2\big] \,, 
\end{equation}
which together with \eqref{eq:179} completes the proof of
\eqref{eq:175}. 
\end{proof}

\subsection{Outer estimates}
\label{sec:4.2}
\subsubsection{Preliminaries: a dilated model}

Our goal in this subsection is to obtain estimates of
$(\B_{\lambda,\alpha,\beta}^\Df))^{-1}$, for a laminar flow $U$ satisfying
\eqref{eq:9}, on the interval $(a,1)$ for some $-1<a<0$. Nevertheless,
given that the results in \cite{AH-arma} and \cite{AH-ems} are stated
for the interval $(-1,1)$, we begin by obtaining results on $(-1,1)$
with a modified laminar flow, which satisfies similar properties to
those of $U$ on $(a,1)$. We then use dilation to establish the results
on $(a,1)$.  We begin by obtaining the following auxiliary result for
a laminar floe $\check U$ satisfying 
\begin{equation}
  \label{eq:182}
\check U(x)-\check U(-1)\geq  \Xi\, (1+x) 
\end{equation}
for some $\Xi>0$ (see \eqref{eq:10}).

\begin{lemma}
  \label{lem:large-nu}
Let $M$, $\Xi $,  $\alpha_0$ and $\mu_0$ denote positive
  constants, and  $ \gamma \in (0,1/6)$. Then, there exist
  positive $\beta_0$ and $C$, such that, for all   $\check U\in C^4([-1,1])$
  satisfying \eqref{eq:182} and $\|\check U\|_{C^4([-1,+1])}\leq M$, for
  all $\beta\geq \beta_0$\,, $\alpha\in[0,\alpha_0)$, and  $\lambda=\mu+i\nu$ satisfying 
  \begin{equation}
\label{eq:183}
-\beta^{-\gamma}<\tilde{\nu} <-\beta^{ -1/6} \mbox{ and } |\mu|<\mu_0\,\beta^{-1/3}\,,
  \end{equation}
  where $\tilde{\nu}=\nu-\check U(-1)$, $\B_{\lambda,\alpha,\beta}^\Df$ (associated by \eqref{eq:6} to $\check U$) is invertible and its
  inverse satisfies
   \begin{equation}
\label{eq:184}
 \big\|(\B_{\lambda,\alpha,\beta}^\Df))^{-1}\big\|+
       \Big\|\frac{d}{dx}\, (\B_{\lambda,\alpha,\beta}^\Df)^{-1}\Big\|\leq  C  \beta^{-1} \,.
   \end{equation}
   \end{lemma}
 \begin{proof}
  The proof is  somewhat similar to the proof of \cite[Proposition
  5.12.2]{AH-ems}. Note that unlike in \cite{AH-ems} we do not
  assume $\check U^{\prime\prime}\neq0$ in $[-1,1]$ in this case and the range
  of $\nu$ values is different.

  {\em Step 1: }Under the assumptions of the lemma, prove that, there
  exist positive $C$ and $\beta_0$ such that, for any $\varepsilon \in (0,1)$,
  $\beta \geq \beta_0$, and $(\phi,f)\in D(\B^\D_{\lambda,\alpha,\beta})\times L^2(-1,1)$ satisfying
   \begin{equation}
\label{eq:185}
    \B_{\lambda,\alpha,\beta}\,\phi=f \,,
  \end{equation}
  we have
 \begin{multline}
\label{eq:186}
  \|(\check U-\nu)w^\prime\|_2^2+|\tilde{\nu}|^2\|w^\prime\|_2^2 
  \leq C\big(\varepsilon^{-1}\beta^{-2}\|f\|_2^2 \\ +\varepsilon\|w\|_2^2 +
  \alpha^2\big[\beta^{-4/3}\|w^\prime\|_2^2+\beta^{-2/3}\|\phi^\prime\|_2^2\big]+\beta^{-1}\|w^{\prime\prime}\|_2\,\|\phi^{\prime\prime}\|_2\big)\,,
\end{multline}
where $w=(\check U-\nu)^{-1}\phi$. Taking the scalar product of
\eqref{eq:185} with $w=(\check U-\nu)^{-1}\phi$, and integrating by parts
yields for the imaginary part
  \begin{multline}\label{eq:187}
    -\Im\langle w,f\rangle =
    \Im\langle w^{\prime\prime},\phi^{\prime\prime}\rangle
   +  \alpha^2\Im\langle w^\prime,\phi^\prime\rangle \\ +
    \beta\Big(\|(\check U-\nu)w^\prime\|_2^2 + \alpha^2\|\phi\|_2^2 
    -|\mu|^2\Big\langle\frac{\phi}{\check U-\nu},\frac{\check U^{\prime\prime}\phi}{|\check U+i\lambda|^2}\Big\rangle\Big) \,.
  \end{multline}
In view of  \eqref{eq:183} we obtain, with the aid of Poincar\'e's inequality,  
\begin{displaymath}
  \|(\check U-\nu)w^\prime\|_2^2 \geq
  \frac{1}{2}\Big[\|(\check U-\nu)w^\prime\|_2^2+  |\tilde{\nu}|^2\, \|w^\prime\|_2^2\Big]\geq
  \frac{1}{2}\|(\check U-\nu)w^\prime\|_2^2+\frac 1C\,  |\tilde{\nu}|^2 \,\|w\|_2^2 \,.
\end{displaymath}
Furthermore, it holds that (note that by \eqref{eq:183} $
|\tilde{\nu}|^2\geq \beta^{ -1/3}$),
\begin{displaymath}
  |\mu|^2\,\Big|\Big\langle\frac{\phi}{\check U-\nu},\frac{\check U^{\prime\prime}\phi}{|\check U+i\lambda|^2}\Big\rangle\Big|\leq
  C\frac{|\mu|^2}{|\check U(-1)-\nu|} \|w\|_2^2 \leq \widehat C\, \beta^{ -2/3+1/6} \|w\|^2_2 \,.
\end{displaymath}
Substituting the above pair of  inequalities into 
  \eqref{eq:187}, we obtain for  $\beta$   sufficiently large
 \begin{multline*}
   \frac{1}{4}(\|(\check U-\nu)w^\prime\|_2^2+  |\tilde{\nu}|^2 \|w^\prime\|_2^2) + \alpha^2\|\phi\|_2^2 \leq
   \beta^{-1} \Big(\|w\|_2\,\|f\|_2 + \|w^{\prime\prime}\|_2\,\|\phi^{\prime\prime}\|_2+ \alpha^2 \|\phi^\prime\|_2
   \,\|w^\prime\|_2\Big)   
 \,,
  \end{multline*}
and consequently,  we obtain, that there exists $C>0$ such that,  for any
$\varepsilon>0$
\begin{multline*}
  \frac{1}{4}\big(\|(\check U-\nu)w^\prime\|_2^2+|\tilde{\nu}|^2\|w^\prime\|_2^2\big) + \alpha^2\|\phi\|_2^2\\
  \leq C\Big(\varepsilon^{-1}\beta^{-2}\|f\|_2^2 +\varepsilon\|w\|_2^2 +
  \alpha^2[\beta^{-4/3}\|w^\prime\|_2^2+\beta^{-2/3}\|\phi^\prime\|_2^2]+\beta^{-1}\|w^{\prime\prime}\|_2\,\|\phi^{\prime\prime}\|_2\Big)\,,
\end{multline*}
 from which \eqref{eq:186} readily follows.\\

{\em Step 2: Estimate $\phi^{\prime\prime} $ }\noindent \\

To estimate $\phi^{\prime\prime} $ we continue as in \cite[Proposition 5.12.2]{AH-ems}. 
Let 
\begin{equation}
\label{eq:188}
  \hat{v}_\Df :=
  -\phi^{\prime\prime}+ \phi^{\prime\prime}(-1)\check{\psi}_- +\phi^{\prime\prime}(1)\check{\psi}_++\alpha^2\phi+\frac{\check U^{\prime\prime}}{\check U+i\lambda}\phi  \,.
\end{equation}
It can be easily verified (see \eqref{eq:106}--\eqref{eq:115}) that
\begin{subequations}
\label{eq:189}
  \begin{equation} 
  (\LL_\beta-\beta \lambda)\hat{v}_\Df= h\,,
\end{equation}
where
\begin{equation}
  h= -f- \left(\frac{\check U^{\prime\prime}\phi}{\check U+i
    \lambda}\right)^{\prime\prime}+\phi^{\prime\prime}(-1)\hat{g}_-+\phi^{\prime\prime}(1)\hat g_+\,.
\end{equation}
\end{subequations}
By \eqref{eq:59} (with $a=-1$)  it holds, for
    $\tilde{\nu}=\check U(-1)-\nu\,$, that
\begin{equation}
\label{eq:190}
  \|\hat{v}_\Df\|_2\leq  C\, (\beta|\tilde{\nu}|)^{-1} \, \|h\|_2\,.
\end{equation}
By (\ref{eq:189}b) we have
\begin{multline*}
  \|h\|_2\leq C \Big(\|f\|_2+|\phi^{\prime\prime}(1)|\,\|\hat{g}_+\|_2 +|\phi^{\prime\prime}(-1)|\,\|\hat{g}_-\|_2 \\
+  \Big\|\frac{\phi}{(\check U+i\lambda)^3}\Big\|_2 + \Big\|\frac{\phi^\prime}{ (\check U+i\lambda)^2}\Big\|_2 + 
  \Big\|\frac{\phi^{\prime\prime}}{(\check U+i\lambda)}\Big\|_2 \Big)\,.
    \end{multline*}
To estimate the last three terms, we use the estimate
  \begin{displaymath}
\| (\check U+ i\lambda)^{-m}\|_\infty \leq C \, |\tilde{\nu}|^{-m}\,.
\end{displaymath}    
We can then write,
\begin{multline*}
 \|h\|_2\leq C \Big(\|f\|_2+|\phi^{\prime\prime}(1)|\,\|\hat{g}_+\|_2+|\phi^{\prime\prime}(-1)|\,\|\hat{g}_-\|_2\\ +
|\tilde{\nu}|^{-1}\|\phi^{\prime\prime}\|_2 +|\tilde{\nu}|^{-2}\|\phi^\prime\|_2
 + |\tilde{\nu}|^{-2}\Big\|\frac{\phi}{\check U+i\lambda}\Big\|_2  \Big)\,.
\end{multline*}
By Hardy's inequality \eqref{eq:57} for the last inequality it holds that
\begin{equation*} 
\Big\|\frac{\phi}{\check U+i\lambda}\Big\|_2 \leq 
  \Big\|\frac{\phi}{\check U-\check U(-1)}\Big\|_2 \leq C\,
  \Big\|\frac{\phi}{1+x}\Big\|_2 \leq \widehat C \,
  \|\phi^\prime\|_2\,. 
\end{equation*}
Combining the above yields
\begin{multline}
\label{eq:191} 
 \|h\|_2\leq C \Big(\|f\|_2+|\phi^{\prime\prime}(1)|\,\|\hat{g}_+\|_2+|\phi^{\prime\prime}(-1)|\,\|\hat{g}_-\|_2 \\ 
+ |\tilde{\nu}|^{-1}\|\phi^{\prime\prime}\|_2+|\tilde{\nu}|^{-2}\|\phi^\prime\|_2
\Big)\,.
\end{multline}
By  \cite[Eq. (8.96)]{AH-arma}
it holds that
\begin{equation}
\label{eq:192}
  |\phi^{\prime\prime}(1)|\,\|\hat{g}_+\|_2+|\phi^{\prime\prime}(-1)|\,\|\hat{g}_-\|_2\leq C\Big(\beta^{-1/4}
  |\phi^{\prime\prime}(1)|+|\phi^{\prime\prime}(-1)| \beta^{-1/4}\, |\tilde{\nu}|^{-5/4}\Big)\,.
\end{equation}

{\bf Estimate $\phi^{\prime\prime}(\pm1)$. \\}

  Let $\zeta_\pm$ be given by
 \eqref{eq:18}. Given that $\langle\phi^{\prime\prime}-\alpha^2\phi,\zeta_\pm\rangle=0$, we may write
 \begin{equation}
 \label{eq:193}
    (\LL_{\beta,\check{U}}^\zeta-\beta\lambda)(\phi^{\prime\prime}-\alpha^2\phi)=i\beta U^{\prime\prime}\phi+f\,.
  \end{equation}
   To estimate $\phi^{\prime\prime}(\pm1)$ we use \cite[Eq. (6.35)]{AH-arma} (with
    $v=\phi^{\prime\prime}-\alpha^2\phi$, $\lambda_-=\lambda\,$, $\lambda_+=\lambda-i\,U(1)\,,$ $g_1=f$ and $g_2=i\beta
    U^{\prime\prime}\phi$)  and Sobolev
    embeddings to obtain that 
\begin{subequations}
 \label{eq:194}
  \begin{equation}
|\phi^{\prime\prime}(-1)|
\leq C\,|\tilde{\nu}|^{1/2}\, \big(\beta^{1/2}\log\beta\,\|\phi\|_{1,2}+\beta^{-1/3}\|f\|_2\big)\,,
\end{equation}
For the estimate of $|\phi^{\prime\prime}(1)|$ , we obtain from \eqref{eq:193}, 
 \cite[Eq. (6.35)]{AH-arma}  (see  the erratum \cite{typo}) and Sobolev embeddings
\begin{equation}
    |\phi^{\prime\prime}(1)|\leq C \, \big(\beta^{1/2}\log\beta \, \|\phi\|_{1,2}+\beta^{-1/3}\|f\|_2\big)\,.
\end{equation}
\end{subequations}
Hence, by \eqref{eq:192}, we obtain that
\begin{equation}
  \label{eq:195}
  |\phi^{\prime\prime}(1)|\,\|\hat{g}_+\|_2+|\phi^{\prime\prime}(-1)|\,\|\hat{g}_-\|_2\leq C\big(
  \beta^{1/4}\, |\tilde{\nu}|^{-3/4}\,\|\phi\|_{1,2}+\beta^{-7/12}\|f\|_2\big)\,.
\end{equation}
Substituting \eqref{eq:195} into \eqref{eq:191} we obtain 
\begin{equation}
\label{eq:196}
    \|h\|_2\leq
C\big(\|f\|_2+ |\tilde{\nu}|^{-1}\|\phi^{\prime\prime}\|_2 
+[|\tilde{\nu}|^{-2}+\beta^{1/4}|\tilde{\nu}|^{-3/4}]\|\phi^\prime\|_2\big)\,.
\end{equation}

 To bound $\|\phi^{\prime\prime}\|_2$ in the right hand side of \eqref{eq:196},
  we  use \cite[Eq. (5.10.54)]{AH-ems} 
and \eqref{eq:188} to obtain
\begin{multline}
\label{eq:197}
  \|\phi^{\prime\prime}\|_2\leq \|\phi^{\prime\prime}-\alpha^2\phi\|_2 \\ \leq
  \|\hat{v}_\Df\|_2+\|{\check U}^{\prime\prime}\|_\infty\Big\|\frac{\phi}{\check
    U+i\lambda}\Big\|_2+|\phi^{\prime\prime}(1)|\,\|\check{\psi}_+\|_2+|\phi^{\prime\prime}(-1)|\,\|\check{\psi}_-\|_2 \,. 
\end{multline}
The first inequality above can be justified as follows 
  \begin{displaymath}
     \|\phi^{\prime\prime}-\alpha^2\phi\|_2^2=\|\phi^{\prime\prime}\|_2^2 +\alpha^4\|\phi\|_2^2
     -2\alpha^2\Re\langle\phi^{\prime\prime},\phi\rangle=\|\phi^{\prime\prime}\|_2^2 +\alpha^4\|\phi\|_2^2 +2\alpha^2\|\phi^\prime\|_2^2\geq \|\phi^{\prime\prime}\|_2^2 \,.
  \end{displaymath}
By \cite[Eqs. (8.87)  and (6.17)]{AH-arma}, (note that the definitions
of $\psi_-$ in \eqref{eq:107} and in  \cite[Eq. (6.8b)]{AH-arma} are
precisely the same),  it holds that
\begin{displaymath}
  \|\check{\psi}_-\|_2 \leq C\, [\tilde{\nu}\beta]^{-1/4}\,.
\end{displaymath}
Combining the above with \eqref{eq:27} and (\ref{eq:194}a,b) yields
\begin{equation}
\label{eq:198}
    |\phi^{\prime\prime}(-1)|\,\|\check{\psi}_-\|_2 +   |\phi^{\prime\prime}(1)|\,\|\check{\psi}_+\|_2  \leq C\, \big(\beta^{1/4}\log\beta\|\phi\|_{1,2}+\beta^{-7/12}\|f\|_2\big) 
\end{equation}
Furthermore, \eqref{eq:190}, and \eqref{eq:196} it holds that 
\begin{equation}
\label{eq:199}
    \|\hat{v}_\Df\|_2\leq C\, \big(\beta^{-1}|\tilde{\nu}|^{-1}\|f\|_2+\beta^{-1}|\tilde{\nu}|^{-2} \|\phi^{\prime\prime}\|_2
+\beta^{-3/4}|\tilde{\nu}|^{-7/4}\|\phi^\prime\|_2\big)  \,.
  \end{equation}
  Finally, using Hardy's inequality \eqref{eq:57} we may conclude that
  \begin{displaymath}
    \|\check U^{\prime\prime}\|_\infty\Big\|\frac{\phi}{\check
    U+i\lambda}\Big\|_2\leq C\,\Big\|\frac{\phi}{\check
    U-U(-1)}\Big\|_2\leq \check C\, \|\phi^\prime\|\,.
  \end{displaymath} 
Substituting the above, together with \eqref{eq:199} and
\eqref{eq:198} into \eqref{eq:197}  yields , with the aid of Poincar\'e's inequality,
\begin{displaymath}
  \|\phi^{\prime\prime}\|_2\leq 
C\,\big(\beta^{-7/12}\|f\|_2+
\beta^{-1}|\tilde{\nu}|^{-2}\|\phi^{\prime\prime}\|_2+  \beta^{1/4}\log\beta \|\phi^\prime\|_2)\,.
\end{displaymath}
Since $|\tilde \nu| \geq \beta^{ -1/6}$, we then obtain (for a possibly greater $\beta_0$) that for $\beta \geq \beta_0$
\begin{equation}
\label{eq:200}
  \|\phi^{\prime\prime}\|_2\leq 
C\,\big(\beta^{ -7/12}\,\|f\|_2+  \beta^{1/4}\log\beta \|\phi^\prime\|_2)\,.
\end{equation}

{\em Step 3}: Prove \eqref{eq:184}.\\

Given that $w=(\check U-\nu)^{-1}\phi$, we easily obtain that
\begin{equation}\label{eq:143aa}
    \|w^{\prime\prime}\|_2\leq C\,  |\tilde{\nu}|^{-1}\,\big( \|w\|_2 + \|w^\prime\|_2 + \|\phi^{\prime\prime}\|_2\big) \,.
\end{equation}
Using \eqref{eq:182} together with Hardy's inequality \eqref{eq:54} we
obtain that
\begin{displaymath}
  \|w\|_2\leq C\, \Big\|\frac{\phi}{U-U(-1)}\Big\|_2\leq \check C\, \|\phi^\prime\|_2 \,,
\end{displaymath}
and hence, using Poincar\'e's  inequality,  we can conclude from \eqref{eq:143aa} that 
\begin{equation}\label{eq:201}
  \|w^{\prime\prime}\|_2\leq C\,  |\tilde{\nu}|^{-1}\,
  \big[\|w^\prime\|_2+\|\phi^{\prime\prime}\|_2\big] \,.
\end{equation}

Substituting \eqref{eq:201}  together with \eqref{eq:200} into the right hand side of
\eqref{eq:186} 
leads to 
\begin{equation}
\label{eq:202}
\begin{array}{l}
  \|(\check U-\nu)w^\prime\|_2^2+|\tilde{\nu}|^2\|w^\prime\|_2^2 + \alpha^2\|\phi\|_2^2 \leq \\  \leq C\,\Big(\varepsilon^{-1}\beta^{-2}\|f\|_2^2 +\varepsilon\|w\|_2^2 \\ \qquad \quad +
  \alpha_0^2\, [\beta^{-4/3}\|w^\prime\|_2^2+\beta^{-2/3}\|\phi^\prime\|_2^2] \\\qquad  \qquad  +
   \beta^{-1}|\tilde{\nu}|^{-1}\|w^\prime\|_2^2+ \beta^{-13/6}\tilde{\nu}^{-1}\|f\|_2^2+
     \beta^{-1/2} |\tilde{\nu}|^{-1}\log^2\beta\|\phi^\prime\|_2^2\Big)\,.
 \end{array}
 \end{equation}

We now use Hardy's inequality \eqref{eq:56} to obtain
\begin{equation}
\label{eq:203}
  \|(\check U-\nu)w^\prime\|_2^2 \geq  \|[\check U-\check U(-1)]w^\prime\|_2^2 + |\tilde{\nu}|^2\|w^\prime\|_2^2 \geq C^{-1}\|w\|_2^2+ |\tilde{\nu}|^2\|w^\prime\|_2^2 \,.
\end{equation}
Using \eqref{eq:203}, we can obtain from \eqref{eq:202},  that for
sufficiently large $\beta_0$ and $\varepsilon^{-1}$ it holds that
\begin{multline}
\label{eq:204}
  \|(\check U-\nu)w^\prime\|_2^2+ \|w\|_2^2+|\tilde{\nu}|^2\|w^\prime\|_2^2 \leq \\ C\, \Big(\beta^{-2}\|f\|_2^2
  +  \big[|\tilde{\nu}^2\beta|^{-1/2}\log^2\beta + \beta^{-2/3}\big] \,\|\phi^\prime\|_2^2\|\Big) 
\end{multline}
Since
\begin{displaymath}
  (\check U-\nu)w^\prime=\phi^\prime - { \check U}^\prime w\,,
\end{displaymath}
we obtain by \eqref{eq:203} that 
\begin{displaymath}
  \|\phi^\prime\|_2\leq  C\,\|(\check U-\nu)w^\prime\|_2 \,.
\end{displaymath}
Combining the above with \eqref{eq:204} yields, with the aid of
\eqref{eq:183},  for sufficiently large 
$\beta_0$  that for $\beta \geq \beta_0$ and some $C>0$
\begin{displaymath}
   \|\phi^\prime\|_2\leq  C\, \beta^{-1}\, \|f\|_2 \,.
\end{displaymath}
Poincar\'e's inequality now completes the proof of the lemma.
\end{proof}

\subsubsection{Estimate in $[a,1)$}
In the sequel we need an estimate like \eqref{eq:184} for the
Orr-Sommerfeld operator defined on the interval $(a,1)$ for some
$a\in(-1,0)$. To this end, we introduce, for some $a\in(-1,0)$, $\B_{\lambda,\alpha,\beta}^{\Df,a}$ as the  operator whose
  differential form is given by \eqref{eq:6} and its domain by
  \begin{displaymath}
    D(\B_U^{\Df,a})=\{u\in H^4(a,1)\,,\, u(a)=u^\prime(a)= u(1) =u^\prime (1) =0 \}\,.
  \end{displaymath}
  We thus state and prove the following:
\begin{proposition}
\label{prop:outer}
  Let $U$ satisfy \eqref{eq:9} and \eqref{eq:10} in $(-1,1)$. Let
  further $\alpha_0$ and $\mu_0$ denote positive constants, and $
  0<\gamma<1/6$. 
  Then, there exist $a_0\in (-1,1)$, positive $\beta_0$ and $C$, such that,
  for all $a\in [-1,a_0]$, for all $\beta\geq \beta_0$ and $\alpha\in[0,\alpha_0)$, and $\lambda=\mu+i\nu$
  satisfying $U(a)-\beta^{-\gamma}<\nu<U(a)-\beta^{ -1/6}$ and $|\mu|<\mu_0\beta^{-1/3}$, it
  holds that
   \begin{equation}
\label{eq:205}
 \big\|(\B_{\lambda,\alpha,\beta}^{\Df,a}))^{-1}\big\|+
       \Big\|\frac{d}{dx}\, (\B_{\lambda,\alpha,\beta}^{\Df,a})^{-1}\Big\|\leq  C  \beta^{-1} \,.
   \end{equation}
\end{proposition}
\begin{proof}
  Applying the affine transformation
  \begin{displaymath}
   T_af(x)=  f\Big(\frac{2x-1-a}{1-a}\Big)\,,
  \end{displaymath}
 we obtain the laminar flow $T_a^{-1}U=\check{U}_a \in C^4([-1,1])$ satisfying
  \begin{displaymath}
    \check{U}_a(x)=U\Big(\frac{1-a}{2}x+\frac{1+a}{2}\Big)\,.
  \end{displaymath}
By \eqref{eq:9}, it holds for sufficiently small $a_0+1$ that there
exists $c>0$, such that, for all $a\in [-1,a_0]$,  
\begin{displaymath}
   U^\prime(a) \geq c \,,
\end{displaymath}
and that
\begin{displaymath}
  U(x)>U(a)\,,\; \forall x\in(a,1]\,. 
\end{displaymath}
It follows that there exists $\Xi>0$ such that
\begin{displaymath}
  U(x)-U(a)\geq \Xi(x-a)\,.
\end{displaymath}
Hence,
\begin{displaymath}
    \check{U}_a(x)-   \check{U}_a(-1)=
    U\Big(\frac{1-a}{2}x+\frac{1+a}{2}\Big)-U(a)\geq \Xi\Big[
    \frac{1-a}{2}x+\frac{1+a}{2}-a\Big]\geq\Xi\,\frac{1-a}{2}(1+x)
\end{displaymath}
Hence, \eqref{eq:182} is satisfied.  Moreover $\check U_a$ is
uniformly bounded in $C^4([-1,+1])$ for $a \in [-1,a_0]$ and,  given that
$\lambda - \check{U}_a(-1)=\lambda-U(a)$, we may use \eqref{eq:184} to obtain
\begin{displaymath}
 \big\|(\B_{\lambda,\alpha,\beta}^{\Df,a})^{-1}\big\|=\frac{(1-a)^4}{16}\big\| (T_a^{-1}\B_{\lambda,\tilde{\alpha},\tilde{\beta}}^\Df
 T_a)^{-1}\big\|\leq  \big\| (\B_{\lambda,\tilde{\alpha},\tilde{\beta}}^\Df)^{-1}\big\|\leq\frac{C}{\tilde{\beta}}\,,
\end{displaymath}
where
\begin{displaymath}
  \tilde{\beta}=\frac{(1-a)^2}{4}\beta \quad ; \quad
  \tilde{\alpha}=\frac{(1-a)}{2}\alpha \,,
\end{displaymath}
and hence, we have obtained \eqref{eq:205}. 
\end{proof}
\newpage 
\subsubsection{A uniform estimate}
We now return to the interval $(-1,1)$ to obtain a $v(1)$-dependent uniform
estimate of $\phi$. To allow for the application of Proposition
\ref{prop:outer} we set in the sequel in
\eqref{eq:148}, the stronger condition for $\delta(\beta)$:
\begin{equation}\label{eq:206}
\delta(\beta)=\beta^{-s} \mbox{ with }  s\in(0,1/6)\,.
\end{equation} 
For convenience of notation we use $\delta$ for $\delta(\beta)$ in the sequel, and
similarly $\varepsilon$ for $\varepsilon(\beta)$, which is defined in \eqref{eq:149}.  We
continue by setting a new cutoff function $\chi_\delta\in C^\infty(\R,[0,1])$
\begin{subequations}
\label{eq:chidelta}
\begin{equation}
  \chi_\delta(x)=\eta\Big(\frac{x+1}{2\delta^{1/2}}-1\Big)\,,
\end{equation}
where $\eta$ is given by \eqref{eq:50}. It can be easily verified that
\begin{equation}
  \chi_\delta(x)=
  \begin{cases}
    1 & x<-1+\delta^{1/2} \\
    0 & x>-1+2\delta^{1/2}
  \end{cases}\,.
\end{equation}
\end{subequations}

We can now obtain the following estimate for $\phi$ satisfying \eqref{eq:146} which depends of $f$
and $v(-1)$ only.
\begin{lemma}
  \label{lem:outer}
Let $\alpha_0>0$. Then, there exist positive $\beta_0$, $\varkappa_0$ and $C$, such
that for any $\beta\geq \beta_0$, $0<\varkappa\leq \varkappa_0$, $0\leq\alpha\leq\alpha_0$,  $\delta(\beta)$ satisfying
 \eqref{eq:206}, and $\lambda$ satisfying  \eqref{eq:98} (or equivalently
$(K)_{\varkappa,\beta}$),  it holds that
\begin{equation}
\label{eq:207}
 \|\phi\|_\infty\leq C\,\Big[( \delta^{1/4}|\lambda-\lambda_0(\beta) |+\beta^{-2/3}\delta^{-1/4})|v(-1)|+\delta^{-1/4}\beta^{-5/6}\|f\|_2\Big]\,,
\end{equation}
for any $(\phi,f)\in D(\B_{\lambda,\alpha,\beta}^\Df)\times L^2(-1,1)$ satisfying \eqref{eq:185}.
\end{lemma}

\begin{proof} 
Set
\begin{equation}
  \label{eq:208}
\hat{\phi}(x)=\phi(x)-[\phi(-1+\delta)+\phi^\prime(-1+\delta)\cdot (x+1-\delta)]\chi_\delta \,,
\end{equation}
It can be easily verified that
\begin{displaymath}
  \hat{\phi}^{\prime\prime}=\phi^{\prime\prime}-2\chi_\delta^\prime\phi^\prime(-1+\delta)-\chi_\delta^{\prime\prime}[\phi(-1+\delta)+\phi^\prime(-1+\delta)\cdot (x+1-\delta)] \,.
\end{displaymath}
It follows that
\begin{displaymath}
  \B_{\lambda,\alpha,\beta}\hat{\phi}=f+ \beta(U+i\lambda)g_1 + g_2 \,,
\end{displaymath}
where
\begin{equation}
\label{eq:209}
  g_1 = 2\chi_\delta^\prime\phi^\prime(-1+\delta)+(\chi_\delta^{\prime\prime}-\alpha^2\chi_\delta)[\phi(-1+\delta)+\phi^\prime(-1+\delta)\cdot (x+1-\delta)] \,,
\end{equation}
and 
\begin{multline}
\label{eq:210}
  g_2=-i\beta U^{\prime\prime}[\phi(-1+\delta)+\phi^\prime(-1+\delta)\cdot(x+1-\delta)]\chi_\delta -
  [4\chi_\delta^{(3)}-2\alpha^2\chi_\delta^\prime]\phi^\prime(-1+\delta) \\
  -  (\chi_\delta^{(4)}-\alpha^2\chi_\delta^{\prime\prime})[\phi(-1+\delta)+\phi^\prime(-1+\delta)\cdot (x+1-\delta)]\,.  
\end{multline}
We also have  that 
\begin{displaymath}
  \hat{\phi}(-1+\delta )=\hat{\phi}^\prime(-1+\delta)=\hat{\phi}(1)=\hat{\phi}^\prime(1)=0\,. 
\end{displaymath}

By \eqref{eq:205} applied with $a=-1 + \delta(\beta)$ (notice that $\beta$ can be  chosen large
enough so that $-1 +\delta(\beta) < a_0$) and the pair
$(\hat \phi,f+ \beta(U+i\lambda)g_1 + g_2 $), it holds that
\begin{displaymath}
  \|\hat{\phi}\|_2+\|\hat{\phi}^\prime\|_2\leq C\,\big[\|(U+i\lambda)g_1\|_2 +\beta^{-1}(\|g_2\|_2+\|f\|_2) \big] \,.
\end{displaymath}

By \eqref{eq:209} and \eqref{eq:175} it holds that, for any $k$, there
exists $C>0$ such that 
\begin{displaymath}
  \|(U+i\lambda)g_1\|_2 \leq C\, \big[\delta^{1/4}\varepsilon^kI_k(\beta) +
    |v(-1)|  \delta^{1/4}( |\lambda-\lambda_0(\beta) |+\beta^{-2/3} \delta^{-1/2})+\delta^{1/4}\beta^{-5/6}\|f\|_2\big] \,. 
\end{displaymath}
Similarly, by \eqref{eq:210} and \eqref{eq:175} we have
\begin{displaymath}
   \|g_2\|_2 \leq C \beta\big[\delta^{3/4}\varepsilon^kI_k(\beta)+
    |v(-1)|\big( \delta^{5/4}|\lambda-\lambda_0(\beta) |+\delta^{1/4}\beta^{-2/3})+\delta^{3/4}\beta^{-5/6}\|f\|_2\big] \,. 
\end{displaymath}
Consequently,
\begin{multline}
\label{eq:211}
   \|\hat{\phi}\|_2+\|\hat{\phi}^\prime\|_2\leq C \, \Big[\delta^{1/4}\varepsilon^kI_k(\beta)+ \\
   +  |v(-1)|( \delta^{1/4}|\lambda-\lambda_0(\beta) |+\beta^{-2/3}\delta^{-1/4})+\delta^{-1/4}\beta^{-5/6}\|f\|_2\Big]\,.
\end{multline}
Since by \eqref{eq:208} and \eqref{eq:175} it holds that
\begin{multline*}
   \|\hat{\phi}-\phi\|_{L^2(-1+\delta,1)}+\|\hat{\phi}^\prime-\phi^\prime\|_{L^2(-1+\delta,1)}\\ \leq
   C \big[ \delta^{3/4} \epsilon^kI_k(\beta) +
 + (\delta^{3/4}|\lambda-\lambda_0(\beta) | +\delta^{1/4}\beta^{-2/3})|v(-1)|+\delta^{1/4}\beta^{-5/6}\|f\|_2\big] \,, 
\end{multline*}
we may conclude from \eqref{eq:211} that 
\begin{multline}
  \label{eq:212}
 \|\phi\|_{L^\infty(-1+\delta,1)}\leq C\, \|\phi\|_{H^1(-1+\delta,1)}\\
 \leq \hat C \Big[\delta^{1/4}\varepsilon^kI_k(\beta)+ 
    |v(-1)|\, \big(\delta^{1/4}|\lambda-\lambda_0(\beta) |+\beta^{-2/3}\delta^{-1/4}\big)\\ +\delta^{-1/4}\beta^{-5/6}\|f\|_2\Big]\,.
\end{multline}

{\bf   Estimate of $I_k(\beta)$.\\}
We now turn to estimate $I_k(\beta)$, whose definition is given in \eqref{eq:151}.
We have, using the definition of $\eta_\delta$ in \eqref{eq:76}
\begin{displaymath}
  I_k(\beta)\leq\|\tilde{v}\|_1+\beta^{-1/6}\|\tilde{v}\|_2+\beta^{-1/2}\|\tilde{v}^\prime\|_2 +\beta^{-1/2} \| \eta_{2^k\delta}^\prime \tilde v\|_2 \,,
\end{displaymath}
Obviously, for all $k\in\N\,$, 
\begin{displaymath}
  I_k(\beta)\leq\|\tilde{v}\|_{L^1(-1,0)}+ C \beta^{-1/6}  \|\tilde{v}\|_{L^2(-1,0)}
  +\beta^{-1/2}\|\tilde{v}^\prime\|_{L^2(-1,0)} \,.
\end{displaymath}
By \eqref{eq:147} it holds that
\begin{multline*}
    I_k(\beta)\leq\|v\|_{L^1(-1,0)}+ |v(-1)|\,\|\check{\psi}_-\|_{L^1(-1,0)} \\ + C\big[
    \beta^{-1/6}  (\|v\|_{L^2(-1,0)}+ |v(-1)|\,\|\check{\psi}_-\|_{L^2(-1,0)})\\
  +\beta^{-1/2}(\|v^\prime\|_{L^2(-1,0)}+ |v(-1)|\, \|\check{\psi}_-\|_{L^2(-1,0)})\big]  \,.
\end{multline*}
Hence, by
\eqref{eq:29}, \eqref{eq:30}, and \eqref{eq:31}, it holds that
\begin{equation}
\label{eq:213}
    I_k (\beta) \leq\|v\|_{L^1(-1,0)}+ C \big(\beta^{-1/6}\|v\|_{L^2(-1,0)}+\beta^{-1/2}\|v^\prime\|_{L^2(-1,0)}\big) +C\beta^{-1/3}|v(-1)|\,.
\end{equation}
Given that $v\in D(\LL_{\beta,U}^\zeta)$,  where $\zeta_\pm$ are given by
(\ref{eq:18}b), it holds that 
\begin{equation}
\label{eq:214}
  (\LL_{\beta,U}^\zeta -\beta\lambda)v=-i\beta\phi+f \,.
\end{equation}
Hence, we obtain from  (\ref{eq:102}a,b),  (\ref{eq:103}a,b), 
and (\ref{eq:105}ab),
applied with $g_1=i\beta\phi$ and $g_2= f  $\,,
 that
\begin{multline}
\label{eq:215}
  \|v\|_{L^1(-1,0)}+\beta^{-1/6}\|v\|_{L^2(-1,0)}+\beta^{-1/2}\|v^\prime\| _{L^2(-1,0)}\\ \leq C \Big[\log
    \beta \,\beta^{-1/3}\,|\lambda-\lambda_0(\beta)|^{-1}\,\|\phi\|_\infty+ \beta^{-7/6}\,|\lambda-\lambda_0(\beta)|^{-1}\,
  \|f\|_2\Big]\,.
\end{multline}
Substituting the above inequality into \eqref{eq:213} yields
\begin{multline}
\label{eq:216}
  I_k(\beta) \leq C \Big[\log
    \beta \,\beta^{-1/3}\,|\lambda-\lambda_0(\beta)|^{-1}\,\|\phi\|_\infty \\ +\beta^{-1/3}|v(-1)|+ \beta^{-7/6}\,|\lambda-\lambda_0(\beta)|^{-1}\,
 \|f\|_2\Big]\,, 
\end{multline}
which, combined with \eqref{eq:212}, leads to 
\begin{multline*}
 \|\phi\|_{L^\infty(-1+\delta,1)} \leq C_k\Big[ \delta^{1/4} \varepsilon^k\big(\log
    \beta \,\beta^{-1/3}\,|\lambda-\lambda_0(\beta)|^{-1}\,\|\phi\|_\infty +\beta^{-1/3}|v(-1)|\\ +\beta^{-7/6}\,|\lambda-\lambda_0(\beta)|^{-1}\,
 \|f\|_2\big)\\
 + \big( \delta^{1/4}|\lambda-\lambda_0(\beta) |+\beta^{-2/3}\delta^{-1/4}\big)\, |v(-1)| +\delta^{-1/4}\beta^{-5/6}\|f\|_2 \Big]\,.
\end{multline*}
For sufficiently large $k$, we can drop from the right hand side the
terms $\delta^{1/4} \varepsilon^k\beta^{-1/3}|v(-1)|$ and $\delta^{1/4}
\varepsilon^k\beta^{-7/6}\,|\lambda-\lambda_0(\beta)|^{-1}\, \|f\|_2$ to obtain 
\begin{multline}
  \label{eq:217}
\|\phi\|_{L^\infty(-1+\delta,1)}\leq C \, \Big[ \delta^{1/4}\varepsilon^k \log
    \beta \,\beta^{-1/3}\,|\lambda-\lambda_0(\beta)|^{-1}\,\|\phi\|_\infty \\
\qquad +\big(\delta^{1/4}|\lambda-\lambda_0(\beta) |+\beta^{-2/3}\delta^{-1/4}\big)|v(-1)|+\delta^{-1/4}\beta^{-5/6}\|f\|_2 \Big]\,.
\end{multline}
As in \eqref{eq:180}, 
we can obtain that for any $x\in(-1,-1+\delta)$
\begin{displaymath}
   |\phi(x)|\leq C\,\Big[\delta\,\|\tilde{v}\|_{L^1(-1,x)}+\beta^{-2/3}\, |v(1)|\Big] \,,
\end{displaymath}
and hence, using the definition of $I_0$ in \eqref{eq:151}, 
\begin{displaymath}
  \|\phi\|_{L^\infty(-1,-1+\delta)}\leq C\,\big [\delta I_0+\beta^{-2/3}|v(1)|\big] \,.
\end{displaymath}
By \eqref{eq:150} it holds that
\begin{displaymath}
   \|\phi\|_{L^\infty(-1,-1+\delta)}\leq C\,\big[\delta\varepsilon^kI_k+\delta^2|\lambda-\lambda_0(\beta) |\log\beta\,|v(-1)|+\delta\beta^{-5/6}\|f\|_2\big] \,.
\end{displaymath}
By \eqref{eq:216} we then obtain, for sufficiently large $k$ 
\begin{multline*}
    \|\phi\|_{L^\infty(-1,-1+\delta)}\leq C\, \Big[\delta\varepsilon^k\log
    \beta \,\beta^{-1/3}\,|\lambda-\lambda_0(\beta)|^{-1}\,\|\phi\|_\infty \\ +\delta^2|\lambda-\lambda_0(\beta) |\log\beta\,|v(-1)|+\delta\beta^{-5/6}\|f\|_2\Big] \,.
\end{multline*}
By combining the above with \eqref{eq:217} and setting $k$ and then
$\beta_0$ to be sufficiently large so that for $\beta\geq \beta_0$ and $\lambda$
satisfying $(K)_{\varkappa,\beta}$ we have
\begin{displaymath}
 C\, \delta\varepsilon^k\frac{\log
    \beta}{\beta^{1/3}|\lambda-\lambda_0(\beta) |} <1/2\,, 
\end{displaymath}
we can conclude \eqref{eq:207}.
\end{proof}
\subsection{Existence of eigenvalues}
\label{sec:4.3}
Using the estimate of $\|\phi\|_\infty$ obtained in \eqref{eq:207} it is now
possible to prove the existence of critical values of $\lambda$ (that serve
as eigenvalues of the linearized Navier-Stokes operator). We begin by
establishing estimates of $(\B^\D_{\lambda,\alpha,\beta})^{-1}$, and then proceed to
prove the existence of eigenvalues.

\begin{lemma}
\label{lem:edge-value}
Let $U$ satisfy \eqref{eq:9}, $\alpha_0>0$,  $ 0<\varepsilon<1/6$\,. Then, there
exist positive $\beta_0$, $\varkappa_0$ and $C$ such that 
for all $0\leq\alpha\leq\alpha_0$, $0<\varkappa<\varkappa_0$,  $\beta\geq \beta_0$,  $\lambda$ satisfying 
\begin{equation}
\label{eq:218}
  \beta^{-1/6+\varepsilon} \leq \beta^{1/3}|\lambda- \lambda_0(\beta) |\leq\varkappa \,,
\end{equation} 
and $(\phi,f)\in D(\B_{\lambda,\alpha,\beta}^\Df)\times L^2(-1,1)$ satisfying \eqref{eq:146} 
 it holds, for
$v=-\phi^{\prime\prime}+\alpha^2\phi\,$, 
\begin{equation}
    \label{eq:219}
|v(-1)|\leq C \, \beta^{-\frac{5}{6}+\varepsilon}\, |\lambda-\lambda_0(\beta) |^{-1}\,\|f\|_2 \,.  
  \end{equation}
\end{lemma}
\begin{proof}
  We use \eqref{eq:214} together with (\ref{eq:104}ab) with $g_1=\beta \phi$ and $g_2=f$
to obtain that
  \begin{displaymath}
    |v(-1)|\leq C\Big[\log \beta\,(\lambda-\lambda_0(\beta) )|^{-1}\, \|\phi\|_\infty + \beta^{- 5/6}(\lambda-\lambda_0(\beta) )^{-1}\|f\|_2\Big]
  \end{displaymath}
Combining the above with \eqref{eq:207} 
yields
\begin{multline}
\label{eq:220}
    |v(-1)|\leq C\, \log
      \beta \, |\lambda-\lambda_0(\beta)|^{-1} \Big[ (\delta^{1/4}|\lambda-\lambda_0(\beta) |\\ +\beta^{-2/3}\delta^{-1/4})|v(-1)|+\delta^{-1/4}\beta^{-5/6}\|f\|_2\Big] \,.
\end{multline}
From \eqref{eq:220} and \eqref{eq:206} we  obtain \eqref{eq:219}
by choosing $\beta_0$ large enough and $s<\varepsilon$ in \eqref{eq:206}. 
\end{proof}
\begin{lemma} 
 Let $U$ satisfy \eqref{eq:9} and $\alpha_0>0$. Let
further $ 0<\varepsilon<1/6$. Then, there exist positive $\beta_0$, $\varkappa_0$ and $C$ such that
for all $0\leq\alpha\leq\alpha_0$, $0<\varkappa<\varkappa_0$,  $\beta\geq \beta_0$ and and  $\lambda$ satisfying \eqref{eq:218}.
it holds that 
  \begin{equation}
    \label{eq:221}
\Big\|\Big(-\frac{d^2}{dx^2}+\alpha^2\Big)(\B^\D_{\lambda,\alpha,\beta})^{-1}\Big\|\leq
C \, \beta^{-1+\varepsilon}\, |\lambda-\lambda_0(\beta) |^{-1} \,.
  \end{equation}
\end{lemma}
The proof easily follows from \eqref{eq:219}, \eqref{eq:207}, and
\eqref{eq:215}. 
Note that 
\begin{displaymath}
  v= \Big(-\frac{d^2}{dx^2}+\alpha^2\Big)(\B^\D_{\lambda,\alpha,\beta})^{-1}\, f\,.
\end{displaymath}

\paragraph{Proof of Theorem \ref{thm:main}}
We now proceed to prove Theorem  \ref{thm:main}. The proof uses the
holomorphic dependence of $\B_{\lambda,\alpha,\beta}$ on $\lambda$, which allows for the
application of Cauchy's Theorem.

Let $ 0<\epsilon<1/12$ and $\lambda=\lambda_0+\rho e^{i\theta}$, where $\lambda_0(\beta)$ is given by
  \eqref{eq:7abcd},
  \begin{equation}\label{eq:rho}
    \beta^{ -\frac{1}{2}+\epsilon }<\rho< \beta^{-\frac{1}{3}-\epsilon }\,.
  \end{equation}
Let $\tilde{\phi}$ be the quasimode given by \eqref{eq:51} and set
\begin{displaymath}
  \tilde{f}=\B_{\lambda_0,\alpha,\beta}\, \tilde{\phi} \,.
\end{displaymath}
Clearly
\begin{displaymath}
  \B_{\lambda,\alpha,\beta}\tilde{\phi}=  \tilde{f} +\beta  (\lambda-\lambda_0(\beta) )\tilde{v}\,,
\end{displaymath}
where
\begin{displaymath}
  \tilde{v}=-\tilde{\phi}^{\prime\prime}+\alpha^2\tilde{\phi} \,.
\end{displaymath}
It follows that
\begin{displaymath}
   (\B^\D_{\lambda,\alpha,\beta})^{-1}\tilde{v}=(\lambda-\lambda_0(\beta))^{-1} \beta^{-1}  \, [\tilde{\phi}+(\B^\D_{\lambda,\alpha,\beta})^{-1}\tilde{f}]\,,
\end{displaymath}
and hence
\begin{displaymath}
   \Big(-\frac{d^2}{dx^2}+\alpha^2\Big)(\B^\D_{\lambda,\alpha,\beta})^{-1}\tilde{v}=
   \beta^{-1} |\lambda-\lambda_0|^{-1}
   \Big[\tilde{v}+\Big(-\frac{d^2}{dx^2}+\alpha^2\Big)(\B^\D_{\lambda,\alpha,\beta})^{-1}\tilde{f}\Big]\,.  
\end{displaymath}
Integrating over $\partial B(\lambda_0,\rho)$ yields
\begin{multline*}
\frac{1}{2i\pi}\,  \int_{\partial B(\lambda_0,\rho)}
  \Big(-\frac{d^2}{dx^2}+\alpha^2\Big)(\B^\D_{\lambda,\alpha,\beta})^{-1}\tilde{v}\,d\lambda\\ =
   \beta^{-1}\, \tilde{v}+\frac{1}{2i\pi}\,\beta^{-1}  \int_{\partial B(\lambda_0,\rho)}
   \frac{1}{\lambda-\lambda_0(\beta) }\Big(-\frac{d^2}{dx^2}+\alpha^2\Big)(\B^\D_{\lambda,\alpha,\beta})^{-1}\tilde{f}\Big]d\lambda\,.  
\end{multline*}
By \eqref{eq:221} it holds that for all $\epsilon>0$ there exists $C>0$
such that
\begin{displaymath}
  \Big\|\Big(-\frac{d^2}{dx^2}+\alpha^2\Big)(\B^\D_{\lambda,\alpha,\beta})^{-1}\tilde{f}\Big\|_2\leq C \frac{\beta^{-1+\epsilon}}{|\lambda-\lambda_0(\beta) |}\|\tilde{f}\|_2\,.
\end{displaymath}
We now use \eqref{eq:53} 
to obtain that
\begin{displaymath}
    \Big\|\Big(-\frac{d^2}{dx^2}+\alpha^2\Big)(\B^\D_{\lambda,\alpha,\beta})^{-1}\tilde{f}\Big\|_2\leq C \frac{\beta^{-2/3+\epsilon}}{|\lambda-\lambda_0(\beta) |}\,.
\end{displaymath}
 Since,
\begin{displaymath}
  \tilde{v}=\eta v_0 -2\eta^\prime\hat{\phi}_0^\prime+(\alpha^2-\eta^{\prime\prime})\hat{\phi}_0\,,
\end{displaymath}
where $\eta$ is given by \eqref{eq:50}, $\hat{\phi}_0$ by \eqref{eq:42}, and
$v_0$ by \eqref{eq:39}, we obtain by \eqref{eq:40}, \eqref{eq:48}, and
\eqref{eq:49} that there exists some $C_0>0$ such that
\begin{displaymath}
  \|\tilde{v}\|_2\geq C_0^{-1} \beta^{-1/6} \,.
\end{displaymath}
Hence, we may conclude,  using \eqref{eq:rho},  that
\begin{displaymath}
  \Big\|\int_{\partial B(\lambda_0,\rho)}\Big(-\frac{d^2}{dx^2}+\alpha^2\Big)(\B^\D_{\lambda,\alpha,\beta})^{-1}
  \tilde{v}\,d\lambda\Big\| \geq \frac{1}{C} {\beta^{-1} } \Big[ \beta^{-1/6} -\frac{\beta^{-2/3+\epsilon}}{\rho} \Big] \geq \frac{1}{\hat C} \beta^{-7/6}\,.
\end{displaymath}
For sufficiently large $\beta$ we thus obtain that
\begin{displaymath}
 \frac{1}{2i\pi}\, \int_{\partial B(\lambda_0,\rho)}\Big(-\frac{d^2}{dx^2}+\alpha^2\Big)(\B^\D_{\lambda,\alpha,\beta})^{-1}
  \tilde{v}\,d\lambda \neq 0\,.
\end{displaymath}
By applying $(-d^2/dx^2+\alpha^2)^{-1}$ to both sides we can conclude that
\begin{displaymath}
   \int_{\partial B(\lambda_0,\rho)}(\B^\D_{\lambda,\alpha,\beta})^{-1}\tilde{v}\,d\lambda \neq 0\,.
\end{displaymath}
Theorem \ref{thm:main} is proved.
$\Box$


\newpage
\begin{thebibliography}{1}
\bibitem{abst72}
{\sc M.~Abramowitz and I.~A. Stegun}, {\em {Handbook of Mathematical
  Functions}}, Dover, 1972.

\bibitem{ag65}
{\sc S.~Agmon}, {\em Lectures on Elliptic Boundary Value Problems}, prepared
  for publication by B. Frank Jones, Jr. with the assistance of George W.
  Batten, Jr.
  \newblock  Van Nostrand Mathematical Studies, No. 2, D. Van Nostrand Co.,
  Inc., Princeton, N.J.-Toronto-London, 1965.

\bibitem{al08} {\sc Y.~Almog}, {\em {The stability of the normal state of superconductors in
  the presence of electric currents}}, SIAM Journal on Mathematical Analysis,
  40 (2008), 824--850.

\bibitem{AH-arma}
{\sc Y.~Almog and B.~Helffer}, {\em On the stability of laminar flows between
  plates}, Archive for Rational Mechanics and Analysis, 241 (2021),
  1281--1401.

\bibitem{AH-ems}
\leavevmode\vrule height 2pt depth -1.6pt width 23pt, {\em On the stability of
  symmetric flows in a two-dimensional channel}, arXiv preprint
  arXiv:2212.12827,  (2022). Mem. Eur. Math. Soc., 22, 
European Mathematical Society (EMS), Berlin, 2026, viii+204 pp.

\bibitem{almog2025stability} \leavevmode\vrule height 2pt depth
  -1.6pt width 23pt, {\em Stability of laminar monotone shear flows in
    a channel for high reynolds number}, arXiv preprint
  arXiv:2507.19106, (2025).
  
  \bibitem{Mas17} {\sc J. Bedrossian, P. Germain, and N. Masmoudi,} 
  \newblock {\em On the stability
  threshold for the 3D Couette flow in Sobolev regularity.}
  \newblock Annals of Mathematics 185 (2017), 541-608. 

\bibitem{BGM} \leavevmode\vrule height 2pt depth -1.6pt width 23pt, 
   {\em Stability of the Couette flow at high Reynolds numbers in 2D and 3D,}
   \newblock Bulletin of the American Mathematical Society 56 (2019),
    373-414.  


\bibitem{chen2020transition}
{\sc Q.~Chen, T.~Li, D.~Wei, and Z.~Zhang}, {\em Transition threshold for the
  2-d Couette flow in a finite channel},  Arch. Rat. Mech. Anal.,
238(1),  (2020), 125--183. 

\bibitem{chenetal23}
{\sc Q.~Chen, D.~Wei, and Z.~Zhang}, {\em Linear inviscid damping and enhanced
  dissipation for monotone shear flows}, Comm. Math. Phys., 400 (2023),
  215--276.

\bibitem{chen2024enhanced}
{\sc Q.~Chen and Z.~Li}, {\em Enhanced dissipation and transition threshold for
  the Poiseuille-Couette flow}, Journal of Differential Equations, 386 (2024),
  404--434.

\bibitem{chen2024transition}
\leavevmode\vrule height 2pt depth -1.6pt width 23pt, {\em Transition threshold
  for the 3D Couette flow in a finite channel},  Memoirs of the American Mathematical
  Society, vol.~296, (2024).

\bibitem{drre04}
{\sc P.~G. Drazin and W.~H. Reid}, {\em Hydrodynamic Stability}. \newblock  Cambridge
  Mathematical Library, Cambridge University Press, Cambridge, second~ed.,
 (2004).
\newblock With a foreword by John Miles.

\bibitem{grenier2016spectral}
{\sc E.~Grenier, Y.~Guo, and T.~T. Nguyen}, {\em Spectral instability of
  general symmetric shear flows in a two-dimensional channel}. Advances in
  Mathematics, 292 (2016), 52--110.
\bibitem{he11}
{\sc B.~Helffer}, {\em On pseudo-spectral problems related to a time-dependent
  model in superconductivity with electric current}, Confluentes Math., 3
  (2011), 237--251.

\bibitem{hen14}
{\sc R.~Henry}, {\em {Spectral instability for even non-selfadjoint anharmonic
  oscillators}}, Journal of Spectral Theory, 4 (2014), 349--364.

\bibitem{jezequel2025orr}
{\sc M.~J{\'e}z{\'e}quel and J.~Wang}, {\em Orr-Sommerfeld equation and complex
  deformation}, arXiv preprint arXiv:2503.22274,  (2025).
\bibitem{jia2023uniform}

{\sc H.~Jia}, {\em Uniform linear inviscid damping and enhanced dissipation
  near monotonic shear flows in high reynolds number regime (i): the whole
  space case}, Journal of Mathematical Fluid Mechanics, 25 (2023), p.~42.



\bibitem{Orr1907} {\sc W. Orr}, 
 \newblock {\em The stability or instability of the steady motions of a liquid. Part II,} 
 \newblock Proceedings of the Royal Irish Academy,  A. 27: 69--138 (1907).

\bibitem{wa53}
{\sc W.~Wasow}, {\em On small disturbances of plane {C}ouette flow}, J.
  Research Nat. Bur. Standards, 51 (1953), 195--202.

\bibitem{typo}
It would have been more natural to use here
\cite[Eq. (8.95)]{AH-arma}, unfortunately, there is a typo in the
formula which should read instead
    $$
    |\phi^\prime(\pm1)| \leq C \big[1+|\lambda_\pm|^{1/2} \beta^{1/6}\big]\,\big(\beta^{-1/2} \|f\| + \beta^{1/3}\log \beta \,\|\phi\|_\infty\big)\,.
    $$
\end{thebibliography}
\def\cprime{$^\prime$}

\section*{Acknowledgements}
  The authors wish to thank F. Nicoleau for contributing the appendix
  to this work. No new data were created or analysed in this
  study. Data sharing is not applicable to this article.

\appendix
\section{Numerical study of the zeros of a rotated Airy primitive by   F. Nicoleau}

In this appendix, we briefly discuss the numerical behavior of the zeros
of $A_0(iz)$, where $A_0$ is given by \eqref{eq:defA0}. We note that
\begin{displaymath}
  A_0(e^{i\frac{5\pi}{6}}z)=\int_{-z}^\infty\operatorname{Ai}(t)\,dt,
\end{displaymath}
where $\operatorname{Ai}$ denotes the Airy function.
Since $A_0(0)=1/3$ (see \cite[Eq. (40)]{wa53})  we have
\begin{displaymath}
   A_0(e^{i\frac{5\pi}{6}}z)=
   \frac{1}{3}+\int_{-z}^0\operatorname{Ai}(t)\,dt=
   \frac{1}{3}+\int_0^z\operatorname{Ai}(-t)\,dt\,.  
\end{displaymath}
We thus set $F(z)= A_0(e^{i\frac{5\pi}{6}}z)$ and search for the zeroes
of the entire function
\begin{equation}
	G(z)=F\!\left(e^{-i\pi/3}z\right),
\end{equation}
where
\begin{equation}
	F(z)=\frac13+\int_0^z \operatorname{Ai}(-t)\,dt\,.
\end{equation}

\medskip\noindent
Numerically, the zeros of $F$ appear in complex conjugate pairs
\[
z_k=x_k\pm i y_k,
\qquad y_k\to0,
\]
and therefore the zeros of $G$ accumulate asymptotically along the
half-line
\[
\arg z=\frac{\pi}{3}.
\]
\medskip\noindent
The following table lists several numerically computed zeros of
\[
G(z)=F(e^{-i\pi/3}z)
\]
in the first quadrant. The zeros are grouped into pairs corresponding
to the conjugate pairs
\[
z_k=x_k\pm i y_k
\]
of the original function \(F\).
\begin{table}[h]
	\centering
	\renewcommand{\arraystretch}{1.15}
	\begin{tabular}{ccc}
		\toprule
		$k$
		&
		$w_k^{+}$
		&
		$w_k^{-}$
		\\
		\midrule
		
		1
		&
		$1.062626 + 4.128843\,i$
		&
		$3.044370 + 2.984682\,i$
		\\
		
		2
		&
		$2.502602 + 6.404945\,i$
		&
		$4.295544 + 5.369790\,i$
		\\
		
		3
		&
		$3.675927 + 8.305705\,i$
		&
		$5.354988 + 7.336298\,i$
		\\
		
		4
		&
		$4.708299 + 10.000712\,i$
		&
		$6.306721 + 9.077863\,i$
		\\
		
		5
		&
		$5.648849 + 11.558221\,i$
		&
		$7.185289 + 10.671158\,i$
		\\
		
		6
		&
		$6.523024 + 13.014499\,i$
		&
		$8.009375 + 12.156354\,i$
		\\
		
		7
		&
		$8.128106 + 15.704826\,i$
		&
		$9.536725 + 14.891559\,i$
		\\
		
		8
		&
		$8.876115 + 16.964277\,i$
		&
		$10.253438 + 16.169079\,i$
		\\
		
		9
		&
		$10.289603 + 19.351860\,i$
		&
		$11.614401 + 18.586987\,i$
		\\
		\bottomrule
	\end{tabular}
	
	\caption{
		Numerically computed zeros of
		\(
		G(z)=F(e^{-i\pi/3}z)
		\)
		in the first quadrant.
		Each row corresponds to the pair of zeros obtained from a conjugate
		pair of zeros of the original function \(F\).
	}
	
	\label{tab:rotatedzeros}
	
\end{table}
	\begin{figure}[h]
		\centering
		\includegraphics[width=0.5\textwidth]{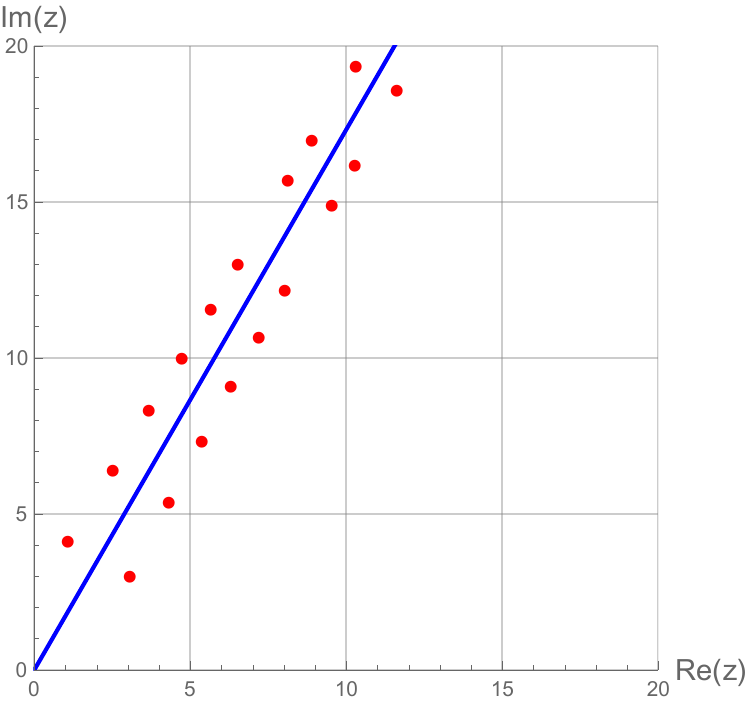}
		\caption{Zeros of $G(z)=F(e^{-i\pi/3}z)$ in the first quadrant.}
		\label{fig:airyzeros}
	\end{figure}
	Figure~\ref{fig:airyzeros} in the next page displays the corresponding point cloud.
	One clearly observes that the zeros lie on both sides of the half-line
	$\arg z=\pi/3$ and approach it asymptotically.

\end{document}